\documentclass[11pt,a4paper]{article}
\usepackage[utf8]{inputenc}
\pdfoutput=1 

\usepackage{authblk}
\usepackage{geometry}
\usepackage{array}
\usepackage{booktabs}
\usepackage{amsmath}
\usepackage{amsthm}
\theoremstyle{remark}

\usepackage{amsfonts}
\usepackage{amssymb}
\usepackage{bm}
\usepackage{caption}
\usepackage{overpic}  
\usepackage{graphicx,subfig}
\usepackage[colorlinks,linkcolor=blue,anchorcolor=red,citecolor=blue]{hyperref}
\usepackage{rotating}
\usepackage{tikz}
\usetikzlibrary{positioning, arrows.meta}
\usepackage{algorithm, algorithmic}
\usepackage{cite}
\usepackage{setspace}
\usepackage{multirow}
\numberwithin{equation}{section}  

\title{Pre-classification based stochastic reduced-order model for time-dependent complex system}
\author[a]{Meixin Xiong}
\author[a]{Liuhong Chen}
\author[a]{Ju Ming\thanks{Corresponding author: jming@hust.edu.cn}}
\author[b]{Zhiwen Zhang}
\affil[a]{School of Mathematics and Statistics, Huazhong University of Science and Technology, Wuhan 430074, China}
\affil[b]{Department of Mathematics, The University of Hong Kong, Hong Kong Special Administrative Region}

\date{} 

\begin{document}
      \newtheorem{theorem}{Theorem}[section]
      \newtheorem{assumption}[theorem]{Assumption}
      \newtheorem{corollary}[theorem]{Corollary}
      \newtheorem{proposition}[theorem]{Proposition}
      \newtheorem{lemma}[theorem]{Lemma}
      \newtheorem{definition}[theorem]{Definition}
      \newtheorem{algo}[theorem]{Algorithm}
      \newtheorem{remark}[theorem]{Remark}
      
      \maketitle
\noindent \rule[4pt]{17cm}{0.05em}  

\noindent   \textbf{Abstract:} 
We propose a novel stochastic reduced-order model (SROM) for complex systems by combining clustering and classification strategies. Specifically, the distance and centroid of centroidal Voronoi tessellation (CVT) are redefined according to the optimality of proper orthogonal decomposition (POD), thereby obtaining a time-dependent generalized CVT, and each class can generate a set of cluster-based POD (CPOD) basis functions. To learn the classification mechanism of random input, the naive Bayes pre-classifier and clustering results are applied. Then for a new input, the set of CPOD basis functions associated with the predicted label is used to reduce the corresponding model. Rigorous error analysis is shown, and a discussion in stochastic Navier-Stokes equation is given to provide a context for the application of this model. Numerical experiments verify that the accuracy of our SROM is improved compared with the standard POD method. \\

\noindent \textbf{Keywords:} 
naive Bayes pre-classifier, generalized centroidal Voronoi tessellation (gCVT), proper orthogonal decomposition (POD), stochastic reduced-order model (SROM), time-dependent.

\noindent \rule[4pt]{17cm}{0.05em} 

\maketitle

\section{Introduction}\label{sec1}

The reduced-order model (ROM) \cite{quarteroni2014reduced,Amsallem2011An,Amsallem2012Nonlinear} plays a vital role in large-scale simulations, real-time calculations, and optimal control problems, which first introduces a low-dimensional subspace of the state space, and then calculates the coordinates of the system state in this subspace through projection techniques, also known as  reduced state vector. It ensures the essential characteristics of the system while achieves the goal of reducing computational complexity. There are a variety of ways to construct the low-fidelity ROM. Among them, proper orthogonal decomposition (POD) based on the optimal Galerkin projection distance is one of the most successful methods, which has been widely applied in numerous fields, including signal analysis and pattern recognition \cite{han2003application,samir2018damage}, image processing \cite{lim2020short,Oberleithner2011Three}, geophysical fluid dynamics \cite{holmes2012turbulence,berkooz1993proper,Towne2018Spectral,Willcox2006Unsteady}, and biomedical engineering \cite{fathi2018denoising,Grinberg2009Analyzing}. 

In many practical problems, the collected data belongs to categorical data, such as countable qualitative data or grouped quantitative data. Then, the natures of these problems can be further explored through the categorical data analysis \cite{saha2019integrated, agresti2018introduction,yee2010vgam,Wang2020Fuzzy}. Clustering \cite{xu2008clustering,rokach2005clustering} and classification \cite{clifford1975introduction,Loh2011Classification} are two advanced tools. Clustering is a method for statistical analysis of data and has  become an important part of machine learning. It is a process of dividing a given data set into several subsets according to some defined distances. Its purpose is to maximize the intra-cluster similarity and minimize the inter-cluster similarity under the given distance measure. On the one hand, clustering itself is a statistical analysis technique. On the other hand, it is often used as a tool for data exploration, data cleaning, and data organizing in the pre-process stage of other data analysis methods. In the past few decades, clustering approaches have been applied to the numerical simulations of partial differential equations (PDEs), and one of the most popular methods is centroidal Voronoi tessellation (CVT) \cite{du1999centroidal}. Some of the notable works in this area are as follows: Burkardt et al. in \cite{burkardt2006centroidal} introduced a reduced-order modeling methodology based on CVT for complex systems and in \cite{Burkardt2006POD} compared the performance of ROMs based on POD and CVT, Du et al. in \cite{du2003centroidal} proposed a hybrid method named CVT based POD for model reduction, and Kaiser et al. in \cite{kaiser2014cluster} combined the cluster analysis and transition matrix models to propose a novel cluster-based reduced-order modelling strategy for unsteady flows. We refer to \cite{Du2003Voronoi,Du2005Approx,lee2007reduced,du2010advances,Gunzburger2014Stochastic} for further discussions.

Classification is another method of data statistical analysis, which assigns labels to samples according to their features. This method belongs to supervised learning and includes two parts: classifier learning and the prediction/classification of new samples. When a new sample is assigned to the class with the highest similarity, using the data in this class to study the sample can make full use of the existing information and eliminate the redundant information brought by the data in other classes. Recently, the ideas of classification have been applied to the study of PDEs. Bright et al. in \cite{bright2013compressive} combined classification and compressive sensing to determine the flow characteristics around a cylinder and in \cite{bright2016classification} proposed sparse measurements to classify and reconstruct time-dependent data, and Brunton et al. in \cite{brunton2014compressive} developed a classification scheme to determine the region to which the nonlinear dynamical system belongs. More discussions can consult the literatures \cite{brunton2020machine,Liang2013Solving,kramer2017sparse,san2019artificial,Lu2016A,Mehrkanoon2015Learning}.
For a stochastic system, there may be large differences between the realizations of its state in some cases. In order to reduce the model and reconstruct the state better, clustering and classification methods can be combined. The former is used to organize the given data according to similarity, while the latter trains a classifier based on the clustering results for assigning labels to new samples. Then the samples can be studied by using the predicted subsets instead of the entire data set.

In this work, we combine clustering and classification methods to propose a pre-classification based stochastic ROM (SROM) for improving the accuracy of the POD reduced-order solutions of stochastic evolution problems. The method mainly consists of two parts. In the first part, several groups of cluster-based POD (CPOD) basis functions are generated by constructing a time-dependent clustering method. Due to the generalizability of the distance in CVT method, the spatio-temporal projection distance from a function to a multidimensional space is used to define the time-dependent generalized Voronoi tessellation (t-gVT). The corresponding generalized centroid is defined as the subspace spanned by the POD basis functions according to the optimality of POD method. Similar to CVT, the time-dependent generalized CVT (t-gCVT) can be obtained when the generators coincide with the generalized centroids. In order to simplify the construction of t-gCVT, the modified version is introduced by using the snapshots generated at several discrete time points to approximately calculate the time integral in generalized distance, and using the Monte Carlo (MC) method to estimate the expectation of projection distance. From this, the spatio-temporal data is divided into several classes, and each class can generate a set of snapshot-based POD basis functions. In the second part, we construct the pre-classification based SROM. Considering the mapping relationship between the input and output of the system, we first use the clustering results to train a pre-classifier to provide predicted labels for the new inputs, and then use the CPOD basis functions associated with the labels to reduce their models. Here, the naive Bayes classifier \cite{Wickramasinghe2021naive,Zhang2009Feature} based on the principle of maximum posterior probability is adopted to establish the classification mechanism. We would like to point out that other classifiers,  such as  $k$-nearest neighbor \cite{Bradley1997The,Muja2014Scalabel},  decision trees \cite{quinlan1987simplifying},  support vector machine \cite{wang2005support}, etc, can also be combined with our CPOD basis functions without any difficulty. The main ideas of our method are shown in Figures \ref{F_framework1} and \ref{F_classifier_framework}. We call the method of combining CPOD basis functions and naive Bayes pre-classifier to construct SROM as the \emph{CPOD-NB method}.

The rest of this paper is organized as follows. In section \ref{Section_preliminaru}, we briefly introduce the traditional POD and CVT methods. In section \ref{Section_alg}, we describe in detail the modified t-gCVT for generating the CPOD basis functions and the naive Bayes method for pre-classification, then combine them to propose the CPOD-NB method for model reduction. The error estimation of the SROM based on CPOD-NB method and the strategy used for estimating the error rate of naive Bayes pre-classifier are given in section \ref{Section_error}. The stochastic Navier-Stokes equation we use as study background is presented in section \ref{Section_NS}. Numerical experiments are shown in section \ref{Section_Numerical}. Finally,  some conclusions are given in section \ref{Section_Conclusion}.

\section{Preliminary}  \label{Section_preliminaru}

We begin by some function spaces and notations needed, then briefly recall the classical POD and CVT methods related to this work.
        
Denote the system of stochastic partial differential equations (SPDEs) of unknown function $u$ as
\begin{equation}  \label{SPDE}
        F( u( t, \mathbf{x}, \bm{ \xi}) ; \bm{ \xi }  ) =0
        \qquad ( t, \mathbf{x} , \bm{ \xi})  \in [0,T] \times D \times \Gamma,
\end{equation}
where function $u$ has proper initial and boundary value conditions, $\mathbf{x}$ is the spatial variable, $t$ is the time variable and $\bm{ \xi}$ could be other parameters with image space $\Gamma$. Let  $L^2(D)$ be the set of square-integrable functions defined on domain $D$ with inner product $ \langle \cdot , \cdot \rangle$ and norm $\| \cdot \|_{L^2 (D)}$. We denote the space of all measurable functions $u: [0,T] \rightarrow L^2(D)$ by 
\begin{equation}
         \mathcal{L}^2 ( [0,T]; L^2(D) ) :=  \left\{  u: [0,T] \rightarrow L^2(D) \ \vert \ u \text{ measurable, }    \| u \|_{\mathcal{L}^2( [0,T]; L^2(D))} < \infty  \right\}
\end{equation}
where 
\begin{equation}  \label{norm}
        \| u \|_{\mathcal{L}^2( [0,T]; L^2(D))}=\left(  \int_0^T  \| u \|^2_{ L^2(D)} dt     \right)^{1/2}.
\end{equation}

\subsection{Proper orthogonal decomposition}

Given a positive integer $d$, for the system of SPDEs (\ref{SPDE}), the POD procedure is to find the orthonormal basis functions $\{ \phi_j (\mathbf{x}) \}_{j=1}^d$ that minimize error measure
\begin{equation}  \label{POD_error}
       \mathcal{E}^{ \text{POD}} ( \Pi^d ) 
       = \mathbb{E} \left[ \big\| u  -  \Pi^d u \big\|^2_{ \mathcal{L}^2([0,T]; L^2(D) ) } \right],
\end{equation}
where $\mathbb{E}[ \cdot]$ denotes expectation, $\Pi^d$ is a projection operator, and $\Pi^d u= \sum_{j=1}^d \langle u, \phi_j \rangle \phi_j$ represents the projection of $u$ onto the $d$-dimensional subspace spanned by $\{ \phi_j \}_{j=1}^d$. 
Note that the operator $\Pi^d$ and the basis function set $\{ \phi_j \}_{j=1}^d$ have a one-to-one correspondence, so without causing confusion, we can also denote $\Pi^d$ as the subspace spanned by these basis functions for the sake of simplicity, that is, $ \Pi^d = \text{span} \{ \phi_j \}_{j=1}^d $.
By the Lagrange multiplier method, the minimization problem is equivalent to
\begin{equation}  \label{POD_equiv_operator}
        \mathbb{E} \left[  \int_0^T \left\langle u,  \phi_j  \right\rangle u dt  \right]  = \lambda_j  \phi_j   \qquad  j=1, 2, \ldots, d,
\end{equation}
where $\left(  \{ \lambda_j \},  \{ \phi_j \} \right)$ is called the eigenpair of operator $ \mathbf{C}$ defined as
\begin{equation}  \label{operator_C1}
        \mathbf{C} \phi_j = \mathbb{E} \left[  \int_0^T \left\langle u,  \phi_j  \right\rangle u dt  \right]. 
\end{equation}

We use the MC method to estimate the expectation, and use snapshots obtained at discrete time points to calculate the time integral, then (\ref{operator_C1}) can be approximated as
\begin{equation}  \label{operator_C2}
       \mathbf{C} \phi_j =  \frac{\Delta t }{n} \sum_{i=1}^n  \sum_{k=1}^J   \left\langle  u(t_k, \mathbf{x}, \bm{\xi}_i),  \phi_j  \right\rangle  u(t_k, \mathbf{x}, \bm{\xi}_i),
\end{equation}
where $t_0=0$, $\Delta t = T/J$ and $t_k = t_{k-1}+ \Delta t$ for $k=1, 2, \ldots, J$. Denote the snapshot set as
\begin{align}  \label{snapshot_set}
        \mathcal{W}=& [v_1, \ldots, v_{nJ}]  \notag \\
        := & [u(t_1, \mathbf{x}, \bm{\xi}_1), \ldots, u(t_J, \mathbf{x}, \bm{\xi}_1), \ldots, u(t_1, \mathbf{x}, \bm{\xi}_n), \ldots, u(t_J, \mathbf{x}, \bm{\xi}_n)].  
\end{align}
Combining (\ref{POD_equiv_operator}) and (\ref{operator_C2}), the orthonormal POD basis functions can be represented as
\begin{equation} \label{POD_basis}
        \phi_j  =   \frac{ 1}{ \sqrt{nJ \sigma_j}}  \sum_{i=1}^{nJ} y_i^{(j)} v_i    \qquad   j=1, 2, \ldots,d.        
\end{equation}
Here, $\{ y_i^{(j)} \}$ and $\{ \sigma_j\}$ satisfy the following eigenvalue problem
\begin{equation}  \label{eigenvalue_pro}
        RY=Y \Lambda,
\end{equation}
where the components of matrices $R$ and $Y$ are defined as $R_{ij}= \frac{1}{nJ} \langle v_i, v_j \rangle$ and $Y_{ij}=y_i^{(j)}$ respectively, and $\Lambda =  \text{diag} (\sigma_1, \ldots, \sigma_{nJ}) $ with $\sigma_1 \geq  \sigma_2  \geq  \ldots \geq  \sigma_{nJ}  \geq 0$ and $\sigma_j = \frac{ \lambda_j}{T}$ for $j=1, 2, \ldots, nJ$.
Therefore, with snapshot set (\ref{snapshot_set}) and POD basis functions (\ref{POD_basis}), the minimum value of measure (\ref{POD_error}) can be approximated as
\begin{equation} \label{POD_energy}
       \mathcal{E}^{ \text{POD}} (\Pi^d) =  \sum_{j=d+1}^{nJ} \lambda_j   = T  \sum_{j=d+1}^{nJ}  \sigma_j,
\end{equation}
which is referred to as the ``POD energy''.

The above discussions of generating POD basis functions based on the MC method is summarized as follows: given time step $\Delta t$, use the MC method to sample an input set $X=\{ \bm{\xi}_i \}_{i=1}^n \subset \Gamma$, then the snapshot set $\mathcal{W}$ can be obtained by solving the system (\ref{SPDE}). Further, the POD basis functions $\{ \phi_j\}_{j=1}^d$ can be generated by solving the eigenvalue problem (\ref{eigenvalue_pro}), and the number of basis $d$ can be determined by the POD energy \cite{Kunisch2002Galerkin,Rathinam2003A}.

\subsection{Centroidal Voronoi tessellation}
        
Given a set of functions $ U=\{ v_i  \in L^2 (D)  \}_{i=1}^n $, the CVT of set $U$ is a special Voronoi tessellation with the centroids $\{ z^*_k  \in L^2(D) \}_{k=1}^K$ of Voronoi regions
\begin{equation} \label{Voronoi_rule}
        \mathcal{U}_k = \{ v \in U \ \vert \  \mathcal{D} (v, z_k) \leq \mathcal{D}( v, z_i )  \text{ for all } i \neq k \}    \quad k=1, 2,\ldots,K
\end{equation}
satisfying $z^*_k = z_k$ for $k=1, 2, \ldots,K$, where $\{ z_k  \in  L^2(D) \}_{k=1}^K$ are called the generators of  set $ \{ \mathcal{U}_k \}_{k=1}^K$, $K$ refers to the number of clusters, and the distance can be selected as any metric, for example the $L^2 (D)$ distance as $\mathcal{D}(v , z_k) = \| v-z_k \|_{ L^2(D)}$ \cite{du1999centroidal}.
When the distances between a point $v$ and two generators $z_i$, $z_j$ are same and the smallest, the principle of random assignment between these two classes is adopted.
According to the partition rule of CVT, we know that it minimizes the error measure
\begin{equation} \label{CVT_error}
        \mathcal{E}^{ \text{CVT}}  \left( \{ \mathcal{U}_k \}_{k=1}^K ; \{ z_k \}_{k=1}^K \right) = \sum_{k=1}^K \sum_{v \in  \mathcal{U}_k} \mathcal{D}^2 ( v, z_k ),
\end{equation}
and (\ref{CVT_error}) is referred to as the ``CVT energy''.

There are several methods that can be used to construct a CVT for a given data set \cite{du1999centroidal,Ju2002Probabilistic,Du2002Numerical,
Du2006Acceleration,Du2006Convergence,hamerly2015accelerating,Hateley2015Fast}. 
Among them, the most popular and simplest iterative-based Lloyd’s algorithm \cite{du1999centroidal} is used in this work.

\section{Pre-classification based SROM}  \label{Section_alg}

For a given SPDE, we first use the similarity and difference between sample solutions to cluster them, and each class can generate a set of POD basis functions. Then, a pre-classifier is trained by clustering results for assigning unlabelled input, and the corresponding model is reduced by the basis functions of the predicted class. In this section, we propose the t-gCVT clustering method for generating multiple sets of POD basis functions and the naive Bayes pre-classifier based SROM.

\subsection{Time-dependent generalized CVT}   
   
As mentioned above, the distance in CVT can be extended to other general distances. And from the measurement formula (\ref{POD_error}), we can see that the POD method is to find a subspace that minimizes the expected value of projection distance. Therefore, it is natural to consider combining the POD and CVT methods.

For a given solution $u$, in order to ensure that the basis functions of a subset after clustering can be used to generate its reduced-order approximation in the entire time domain, the time-dependent distance is defined as
\begin{equation}  \label{distance1}
       \widehat{ \mathcal{D}} (u, \Pi^d u)  =  \| u- \Pi^d u \|_{\mathcal{L}^2 ( [0,T]; L^2(D) ) }
\end{equation}
for any $d$-dimensional subspace $\Pi^d$.
Given a set of multidimensional subspaces $\{ \Pi^{d_k}_k \}_{k=1}^K$, $d_k \in \mathbb{N}^+$, for the solution $u $ of SPDE (\ref{SPDE}), the t-gVT is given as
\begin{equation}   \label{gVT1}
        \widehat{ \mathcal{U}}_k = \{ u \in U_s \ \vert  \widehat{ \mathcal{D}}^2  ( u, \Pi^{d_k}_k u ) \leq  \widehat{ \mathcal{D}}^2 ( u, \Pi^{d_i}_i u ) \text{ for all } i \neq k \}  \quad k=1, 2, \ldots,K,
\end{equation}
where $U_s$ denotes the solution space, which is composed of all functions $u$ satisfying system (\ref{SPDE}).
Similar to ($\ref{Voronoi_rule}$), the principle of random assignment in the appropriate classes is used to break the deadlock. It is well-known that the traditional CVT method clusters data by trying to separate samples into several classes that have the equal variance in the sense of the given distance. Therefore, the generalized centroid can be naturally defined as the subspace $\widehat{ \Pi}_k^{d_k}$ spanned by orthonormal basis functions, which minimizes
\begin{equation}  \label{gCVT_err1}
        \widehat{ \mathcal{E} }^{ \text{t-gCVT}}_k \left( \widehat{ \Pi}_k^{d_k} \right) 
        =  \mathbb{E}  \left[  \| u- \widehat{ \Pi}_k^{d_k} u \|^2_{\mathcal{L}^2 ( [0,T]; L^2(D) ) }  \right]   \quad   k=1, 2, \ldots, K,
\end{equation}
where $u \in \widehat{ \mathcal{U} }_k$ for $k=1, 2, \ldots, K$. Next, the t-gCVT is derived from the definition of CVT.
\begin{definition}  \label{def_tgCVT}
        The t-gVT $ ( \{ \widehat{ \mathcal{U}}_k \}_{k=1}^K; \{ \Pi_k^{d_k} \}_{k=1}^K  )$ of the solution space $U_s$ is called t-gCVT if and only if the generator $\Pi_k^{d_k}$ of class $ \widehat{ \mathcal{U}}_k$  is the corresponding generalized centroid,   i.e. $\Pi_k^{d_k} =  \widehat{ \Pi}_k^{d_k} $, for $k=1, 2, \ldots,K$.
\end{definition}

As can be seen from the above description, in the process of t-gCVT clustering, the calculation of distance (\ref{distance1}) involves time integral, and the construction of the generalized centroid is difficult because it is required to be optimal over the entire time domain in the sense of expectation. Therefore, the MC method with sample set $\widehat{U} = \{ u_i\}_{i=1}^n := \{ u (t, \mathbf{x}, \bm{\xi}_i) \}_{i=1}^n$ is used for the expectation, and the snapshots generated at equal time intervals are used to define a modified distance as
\begin{equation}  \label{distance2}
        \widetilde{ \mathcal{D}}^2 (u_i, \Pi^d u_i) 
        =  \sum_{j=1}^J \| u (t_j, \mathbf{x}, \bm{\xi}_i)- \Pi^d u(t_j, \mathbf{x}, \bm{\xi}_i ) \|^2_{L^2 (D)}    \quad  i=1, 2, \ldots ,n,
\end{equation}
where $\{ t_j \}_{j=1}^J$ are the corresponding time points of snapshots, $t_0=0$ and $t_{j+1} = t_j + \Delta t$ for $j=0, 1, \ldots, J-1$ with time interval  $\Delta t = T/J$. This is equivalent to using the snapshots to approximately calculate the integral with respect to time in distance (\ref{distance1}), and the scaling factor is $\Delta t$.        Then the modified t-gVT can be defined as
\begin{equation}  \label{gVT2}
        \widetilde{ \mathcal{U}}_k = \{ u \in \widehat{U} \ \vert \  \widetilde{ \mathcal{D}}^2  ( u, \Pi^{d_k}_k u ) \leq  \widetilde{ \mathcal{D}}^2 ( u, \Pi^{d_i}_i u ) \text{ for all } i \neq k \}  \quad k=1, 2, \ldots,K,
\end{equation}
and the modified generalized centroid $\widetilde{ \Pi}_k^{d_k} = \text{span} \{ \phi_1^k, \ldots, \phi_{d_k}^k \}$ minimizes
\begin{equation}  \label{gCVT_err2}
        \widetilde{ \mathcal{E} }^{ \text{t-gCVT}}_k \left( \widetilde{ \Pi}_k^{d_k} \right) 
        = \sum_{u \in \widetilde{ \mathcal{U}}_k}   \sum_{j=1}^J \| u (t_j, \mathbf{x}, \bm{\xi} )- \widetilde{ \Pi}_k^{d_k} u(t_j, \mathbf{x}, \bm{\xi} ) \|^2_{L^2 (D)}    \quad   k=1, 2, \ldots, K.
\end{equation}
Denote the cardinality of $\widetilde{ \mathcal{U}}_k$ as $n_k$, which satisfies $\sum_{k=1}^K n_k=n$.
According to the optimality of POD, for $k=1, 2, \ldots, K$, the modified generalized centroid $ \widetilde{ \Pi}_k^{d_k} $ is actually the subspace spanned by the POD basis functions, which are generated by the snapshots of set $ \widetilde{ \mathcal{U}}_k $.

If the approximate error of the time integral is negligible, that is,
\begin{equation}
       \widehat{D}^2(u, \Pi^d u) = \Delta t \widetilde{D}^2 (u, \Pi^d u)
\end{equation}
holds for any given subspace $\Pi^d$. Then the following inequality is known from the relationship between the minimum value of the expected value and the expectation of the minimum value
\begin{equation}  \label{min_relation}
        \min \mathbb{E} \left[  \widehat{D}^2(u, \Pi^d u) \right]  \geq   \Delta t \mathbb{E} \left[  \min \widetilde{D}^2 (u, \Pi^d u)   \right].
 \end{equation}
Therefore, $\{ \widehat{\mathcal{E}}^{ \text{t-gCVT}}_k \}_{k=1}^K$ and $\{ \widetilde{ \mathcal{E}}^{ \text{t-gCVT}}_k \}_{k=1}^K$ satisfy
\begin{equation}
        \min \widehat{\mathcal{E}}^{ \text{t-gCVT}}_k  \geq  \frac{\Delta t}{n_k} \min \widetilde{ \mathcal{E}}^{ \text{t-gCVT}}_k   \qquad   k=1, 2,  \ldots,K
\end{equation}
by using the MC method with $n_k$ samples of set $\widetilde{\mathcal{U}}_k$ to estimate the right-hand side of inequality (\ref{min_relation}).

Similar to Definition \ref{def_tgCVT}, the definition of modified t-gCVT is given as follows.
\begin{definition}  \label{def_tgCVT_modified}
         The modified t-gVT $ ( \{ \widetilde{ \mathcal{U}}_k \}_{k=1}^K; \{ \Pi_k^{d_k} \}_{k=1}^K  )$ of the set $\widehat{U}$ is called modified t-gCVT if and only if the generator $\Pi_k^{d_k}$ of set $ \widetilde{ \mathcal{U}}_k$  is the corresponding generalized centroid,   i.e. $\Pi_k^{d_k} =  \widetilde{ \Pi}_k^{d_k} $, for $k=1, 2, \ldots,K$. And the POD basis functions $  \{ \phi_j^{k}\}_{j=1}^{d_k} $ corresponding to the generalized centroid $\widetilde{ \Pi}_k^{d_k}$ of modified t-gCVT are called its subclass basis functions or  cluster-based POD (CPOD) basis functions.
\end{definition}

It can be seen from the above definition that the modified t-gCVT of set $\widehat{U}$ minimizes the error
\begin{equation}  \label{gCVT_err_total}
       \widetilde{ \mathcal{E} }^{ \text{t-gCVT}} = \sum_{k=1}^K \widetilde{ \mathcal{E} }^{ \text{t-gCVT}}_k \left( \widetilde{ \Pi}_k^{d_k} \right),
\end{equation}
and the minimum value is
\begin{equation}  \label{gCVT_err_MinimalValue}
       \widetilde{ \mathcal{E} }^{ \text{t-gCVT}}= \sum_{k=1}^K   J n_k \sum_{j=d_k+1}^{ J n_k} \sigma_j^k,
\end{equation}
where $ \{ \sigma^k_j \}_{j=1}^{J n_k}$ are the eigenvalues of correlation matrix $R$ associated with set $\widetilde{\mathcal{U}}_k$, as difined in (\ref{eigenvalue_pro}).
Here, (\ref{gCVT_err_MinimalValue}) is referred to as ``modified t-gCVT energy'', and 
\begin{equation}
        \nu_k = \sum_{j=1}^{ d_k} \sigma_j^k   \bigg/  \sum_{j=1}^{ J n_k} \sigma_j^k  \qquad  k=1, 2,  \ldots, K
\end{equation}
is called the energy ratio of CPOD basis functions $\{ \phi_j^k \}_{j=1}^{d_k}$.

To reduce the complexity of model construction, the modified t-gCVT is used in the following processes, and its structure is shown in Figure \ref{F_framework1}. Note that the modified t-gCVT is reduced to the standard snapshot-based POD method when $K=1$, and the number of CPOD basis functions $\{ \phi_j^k\}_{j=1}^{d_k} $ is not neccessarily equal for $k=1, 2,  \ldots,K$.
\begin{figure}[H]
       \centering   
       \includegraphics[width=0.6 \textwidth,trim=160 80 170 190,clip]{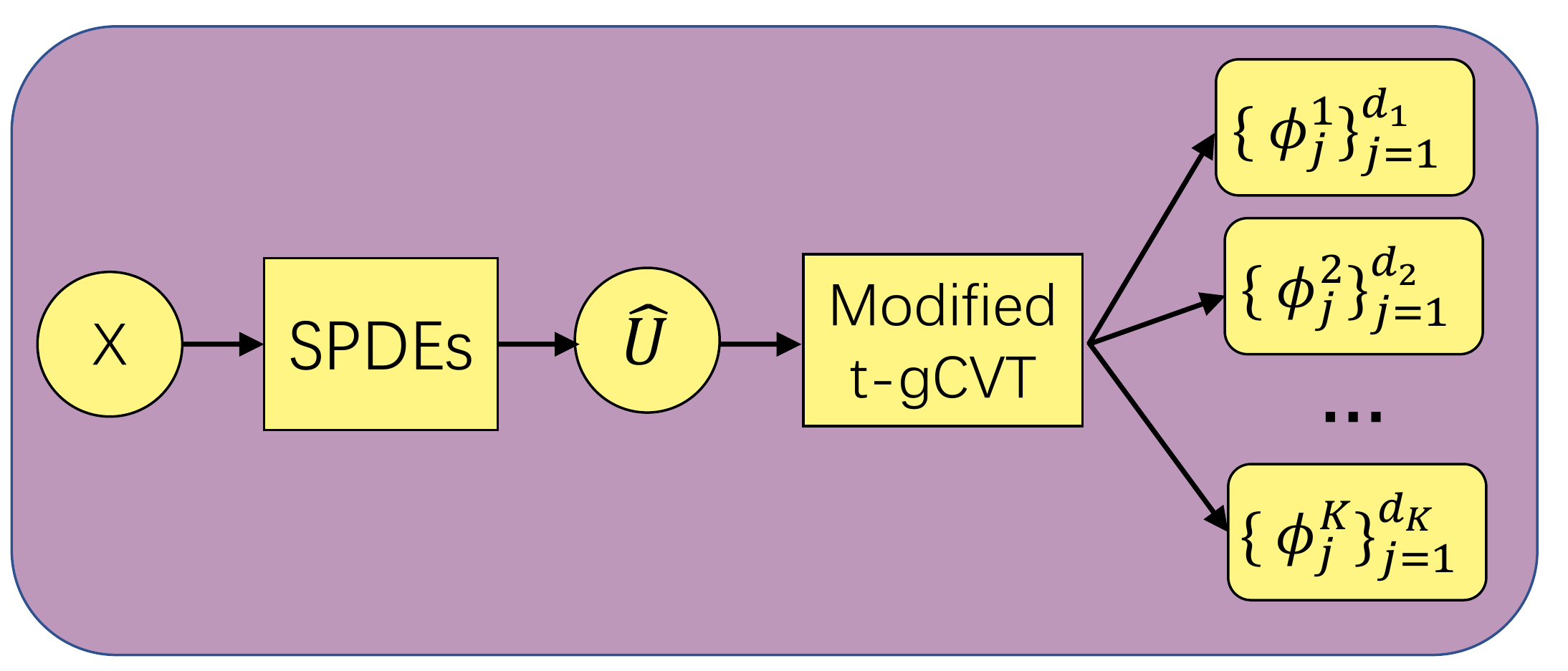}
       \caption{The framework of modified t-gCVT method   }
        \label{F_framework1} 
\end{figure}

\begin{remark}  \label{remark_1}
       When the modified t-gCVT $( \{ \widetilde{ \mathcal{U} }_k \}_{k=1}^K ; \{ \widetilde{ \Pi}_k^{d_k} \}_{k=1}^K )$ of set $\widehat{ U}$ is known, we can naturally cluster the inputs $\{ \bm{\xi}_i \}_{i=1}^n$ according to the clustering results of data $\widehat{U}$. Namely, the image space $\Gamma $ of input $\bm{\xi}$ can be divided into $\{ \Gamma_k \}_{k=1}^K$, which satisfies  $\Gamma_i \cap \Gamma_j =  \emptyset$ if $i \neq j$, $\Gamma_k \subset \Gamma $ for $k =1, 2,   \ldots, K$ and $\bigcup_{k=1}^K \Gamma_k =\Gamma$. If $u (t, \mathbf{x}, \bm{\xi} ) \in \widetilde{ \mathcal{U}}_k$, then the corresponding input, $ \bm{\xi}  \in \Gamma$, is belonging to $\Gamma_k$, i.e.,
       \begin{equation}  \label{image_space}
                 \Gamma_k = \{  \bm{ \xi } \in \Gamma  \ \vert \  u (t, \mathbf{x}, \bm{\xi} ) \in  \widetilde{ \mathcal{U} }_k \}   \qquad   k=1, 2,  \ldots,K,
        \end{equation}
        where $k$ is called the class label of $\bm{\xi}$.
\end{remark}

The details of using the modified t-gCVT method to generate the CPOD basis functions are given in Algorithm~\ref{alg_tgCVT}.     
\begin{algorithm}[ht]
        \renewcommand{\algorithmicrequire}{\textbf{Input:}}
        \renewcommand{\algorithmicensure}{\textbf{Output:}}
        \caption{The modified t-gCVT clustering method for generating CPOD basis functions}  \label{alg_tgCVT}                
        \begin{algorithmic}[1] 
                \REQUIRE  set $\widehat{U} = \{ u(t,\mathbf{x}, \bm{\xi}_i ) \}_{i=1}^n$, a positive integer $K$, dimensions $\{ d_k\}_{k=1}^K$, step size $\Delta t$.
                  \ENSURE modified t-gCVT $( \{ \widetilde{\mathcal{U}}_k \}_{k=1}^K ; \{ \widetilde{\Pi}_k^{d_k} \}_{k=1}^K )$ of set $\widehat{U}$, and $K$ groups of CPOD basis functions $ \{ \{ \phi_j^1\}_{j=1}^{d_1},  \ldots,  \{ \phi_j^K\}_{j=1}^{d_K} \} $.
                \STATE Select a set of initial generalized generators $\{ \Pi_k^{d_k} \}_{k=1}^K$ with dimensions $ \{ d_k \}_{k=1}^K$.
                \STATE Construct the modified t-gVT $ \{ \widetilde{ \mathcal{U}}_k \}_{k=1}^K $  of $\widehat{U}$ associated with  $\{ \Pi_k^{d_k} \}_{k=1}^K$.
                \STATE From $ \{ \widetilde{ \mathcal{U}}_k \}_{k=1}^K $ and step size $\Delta t$, determine the snapshot sets  $\{ \mathcal{W}_k \}_{k=1}^K$ defined in (\ref{snapshot_set}).
                 \STATE Generate $K$ groups CPOD basis functions $ \{ \{ \phi_j^1\}_{j=1}^{d_1},  \ldots,  \{ \phi_j^K\}_{j=1}^{d_K} \} $ defined in (\ref{POD_basis}) by solving eigenvalue problems (\ref{eigenvalue_pro}) associated with $\{ \mathcal{W}_k\}_{k=1}^K$.
                  \STATE For $k=1, 2, \ldots, K$, let $\widetilde{ \Pi }_k^{d_k} = \text{span} (\phi_1^k, \ldots, \phi_{d_k}^k)$, if $ \Pi_k^{d_k}=\widetilde{ \Pi }_k^{d_k} $, stop; otherwise, let $ \Pi_k^{d_k}=\widetilde{ \Pi }_k^{d_k} $ and return to step 2.
        \end{algorithmic}
\end{algorithm}

\subsection{Naive Bayes pre-classifier and pre-classification based SROM}  
 
Since the modified t-gCVT method is to cluster the spatio-temporal function $u$, then for a given $u$, a set of suitable CPOD basis functions can be used to calculate its reduced-order approximation in the whole time interval. In modified t-gCVT, the set $\widetilde{\mathcal{U}}_k$ with the highest similarity to the function $u$ is called its best-matched set, and the corresponding CPOD basis functions are called the best-matched basis functions. In general, the reduced-order approximation generated by the best-matched basis functions is better than the standard POD approximation with the same degree of freedom (DoF). This is because that the samples in the same class are similar after clustering, then the same number of basis functions can capture more useful information, which is beneficial for the reconstruction of function $u$. That is to say, if we know the best-matched basis functions of a given function, the accuracy of its reduced-order approximation can be improved compared with the standard POD method. Note that the spatio-temporal function $u (t, \mathbf{x}, \bm{ \xi})$ is determined by the random input $\bm{ \xi}$, and our aim is to construct a SROM such that the approximate solution can be obtained for any given input $\bm{ \xi}$. Therefore, a pre-classifier is constructed here to select the best-matched basis functions from the perspective of random input.       
        
In this paper, the naive Bayes pre-classifier based on Bayes' theorem and the assumption of feature condition independence is adopted. For a given integer $K \geq 1$, the image space $\Gamma$ is divided into disjoint subspace set $\{ \Gamma_k \}_{k=1}^K$ as introduced in Remark \ref{remark_1}. Suppose $\bm{ \gamma }$ is a random vector defined on the input space $\Gamma \subset \mathbb{R}^p$ composed of $p$-dimensional vectors. Its realization, also known as the feature vector, is denoted as $ \bm{ \xi }=[\xi_1, \ldots, \xi_p ]^\top \in \Gamma$. Let $\iota$ be a random variable defined on the class label set $\mathcal{L} = \{ 1, \ldots, K\}$. Its realization, also known as class label, is denoted as $k \in \mathcal{L}$. Let $X = \{ \bm{ \xi}_i \}_{i=1}^n$ be the independent and identically distributed (i.i.d.) input set of the given data $\widehat{U}$, and $\{ \iota_i \}_{i=1}^n$ be the corresponding class labels obtained by the modified t-gCVT method, then the training data set is given as        
\begin{equation}
        \mathbb{D} = \left\{ ( \bm{\xi}_1, \iota_1), \ldots,  ( \bm{\xi}_n, \iota_n) \right\}.
\end{equation}
Denote the prior probability distributions
\begin{equation}  \label{prior}
       \mathbb{P} ( \iota =k) = \pi_k   \qquad   k=1, 2, \ldots, K,
\end{equation}
and conditional probability distributions
\begin{equation}  \label{condition_prob}
       \mathbb{P} ( \bm{ \gamma}=\bm{ \xi} \vert \iota =k) 
       = \prod_{i=1}^p   \mathbb{P} ( \gamma_i= \xi_i \vert \iota =k)  
       = f_k ( \bm{ \xi} )   \qquad   k=1, 2, \ldots,K
\end{equation}
as
\begin{equation}
       \pi_k = \frac{n_k}{n}  \qquad k=1, 2, \ldots, K
\end{equation}
and 
\begin{equation}
       f_k ( \bm{\xi } ) 
       =  \prod_{i=1}^p  f_k ( \xi_i ) 
       = \prod_{i=1}^p \frac{1}{ \sqrt{2 \pi} \sigma_i^{ (k) }} \exp \left(  - \frac{ \vert \xi_i - \mu_{i}^{(k)} \vert ^2 }{ 2 ( \sigma_i^{(k)})^2} \right)   \quad  k=1, 2, \ldots,K,
\end{equation}
respectively. 
Here, the means $ \{ \mu_i^{ (k) } \}$ and variances $ \{ \sigma_i^{ (k) } \}$ can be estimated by
\begin{align}
        \mu_i^{ (k)} &= \frac{1}{n_k}  \sum_{ \bm{\xi} \in \Gamma_k } \xi_i    \qquad i=1, 2, \ldots,p,  \  k=1, 2, \ldots, K,  \\
        \sigma_i^{ (k) } &= \frac{1}{ n_k - 1 } \sum_{ \bm{ \xi } \in \Gamma_k } \left( \xi_i - \mu_i^{ (k)} \right)^2   \qquad i=1, 2, \ldots,p, \  k=1, 2, \ldots, K.
\end{align}
According to the Bayes' theorem, the posterior probability has form
\begin{equation}   \label{posterior}
        \mathbb{P} ( \iota=k  \vert  \bm{\gamma}= \bm{\xi} ) = \frac{  \pi_k f_k ( \bm{ \xi})  }{ \sum_{k=1}^K   \pi_k f_k ( \bm{ \xi}) }.     
\end{equation} 
        
The principle of naive Bayes pre-classifier is to assign input to the class with the largest posterior probability, that is, input $\bm{ \xi}$ is assigned to the subspace $\Gamma_k$ if
\begin{equation}  \label{classifier1}
         k = \mathop{ \arg \max}_{1 \leq i \leq K} \mathbb{P} ( \iota =i  \vert \bm{\gamma}= \bm{\xi} ).
\end{equation}
The denominator of (\ref{posterior}) is a fixed constant for a given $\bm{\xi}$,  so (\ref{classifier1}) is equivalent to
\begin{equation}   \label{classifier2}
        k = \mathop{ \arg \max}_{1 \leq i \leq K} \pi_i f_i ( \bm{ \xi}).
\end{equation}
If the result in (\ref{classifier2}) is not unique, a random assignment is used to break the tie. Here, $k$ is the predicted label of input $\bm{\xi}$, and the corresponding $\widetilde{\mathcal{U}}_k$ and $\{ \phi_j^k\}_{j=1}^{d_k}$ are called the predicted best-matched set and predicted best-matched basis functions of solution $u( t, \mathbf{x}, \bm{\xi})$, respectively.    

The naive Bayes classifier is based on the assumption of normality and independence of variables, which will affect the accuracy of the algorithm to a certain extent. But this algorithm is easy to implement and has high learning and prediction efficiency. Therefore, it is still one of the popular classification tools.

When the naive Bayes pre-classifier assigns an unlabelled input $\bm{ \xi }$ to the subspace $\Gamma_k$, that is to say, the probability of $\bm{\xi} \in \Gamma_k$ is the largest, then the continuity of the input-output mapping shows that its solution $u$ is most likely to belong to the set $ \widetilde{ \mathcal{U} }_k $. Therefore, it is feasible to use $k$-th group CPOD basis functions $\{ \phi_j^k \}_{j=1}^{d_k}$ of modified t-gCVT to evaluate the corresponding model, and the approximation of solution $u$ is given by
\begin{equation}   \label{gCVT_sol}
        \widetilde{u}^K ( t , \mathbf{x}, \bm{ \xi} ) = \sum_{j=1}^{d_k} \alpha_j ( t, \bm{ \xi} ) \phi_j^{k} ( \mathbf{x}),
\end{equation}       
where $\{ \alpha_j \}_{j=1}^{d_k}$ can be obtained by solving the following reduced system
\begin{equation}   \label{ROM}
         \left\langle  F \left(  \sum_{j=1}^{d_k} \alpha_j ( t , \bm{\xi})   \phi_j^k ( \mathbf{x} )   ; \bm{ \xi }  \right) , \phi_i^k ( \mathbf{x} )  \right\rangle = 0  \qquad  i=1, 2,  \ldots, d_k.
\end{equation}

We call the method of combining CPOD basis functions and naive Bayes pre-classifier to construct SROM as the \emph{CPOD-NB method}, and $\widetilde{u}^K$ defined in (\ref{gCVT_sol}) is the CPOD-NB reduced-order approximation of solution $u$ with the number of clusters $K$.  The structure of the model reduction based on CPOD-NB method is shown in Figure \ref{F_classifier_framework}.
\begin{figure}[ht]
       \centering   
       \includegraphics[width=0.7 \textwidth,trim=160 80 60 190,clip]{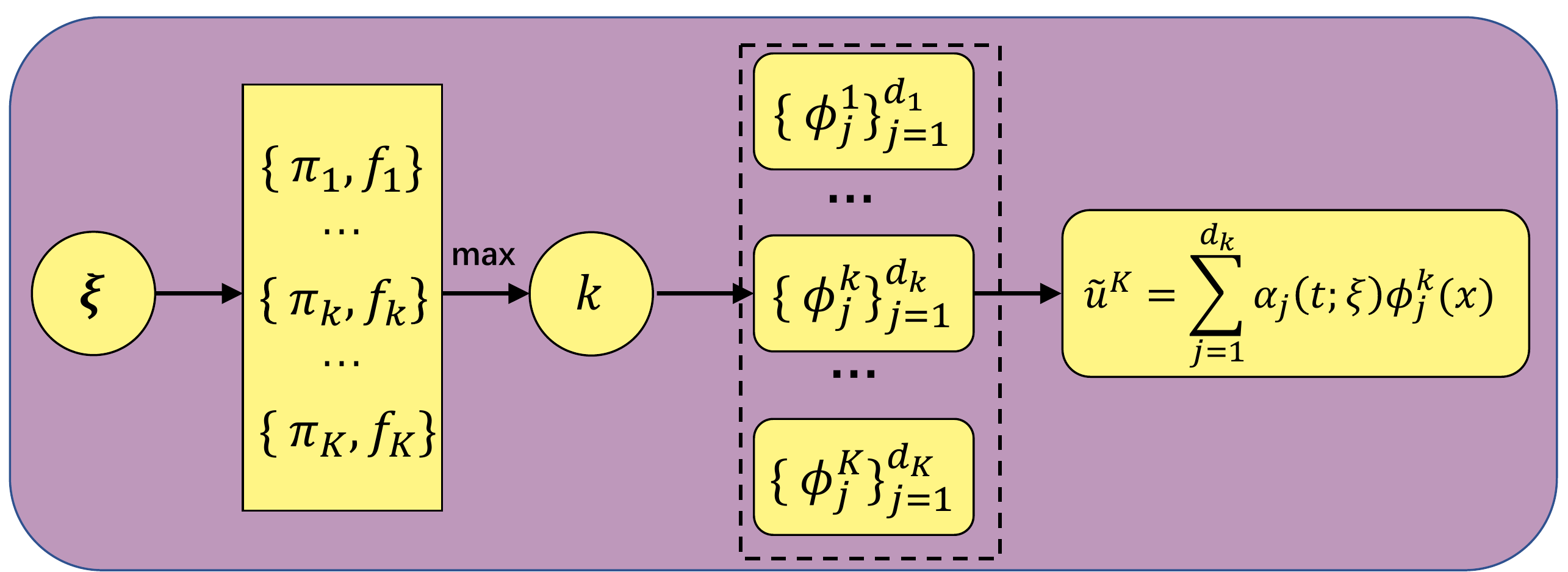}
       \caption{The framework of model reduction based on CPOD-NB method}    \label{F_classifier_framework}
\end{figure}
        
So far, the modified t-gCVT method and pre-classification based SROM have been introduced, and the details of CPOD-NB method for model reduction are described in Algorithm \ref{alg_CPOD}.
\begin{algorithm}[ht]
       \renewcommand{\algorithmicrequire}{\textbf{Input:}}
       \renewcommand{\algorithmicensure}{\textbf{Output:}}
       \caption{SROM based on CPOD-NB method}  \label{alg_CPOD}                
       \begin{algorithmic}[1] 
                 \REQUIRE input set $X=\{ \bm{ \xi}_i \}_{i=1}^n$, a positive integer $K$, dimensions $ \{ d_k \}_{k=1}^K$, step size $\Delta t$.
                 \ENSURE CPOD-NB approximate solution $\widetilde{ u}^K $ of new input $\bm{ \xi}$.
                 \STATE Generate data set $\widehat{U}$ by solving system (\ref{SPDE}) with inputs $X$.
                 \STATE Obtain the modified t-gCVT $ ( \{ \widetilde{\mathcal{U}}_k \}_{k=1}^K ; \{ \widetilde{\Pi}_k^{d_k} \}_{k=1}^K )$ of $\widehat{U}$ and  $K$ groups of CPOD basis functions $ \{ \{ \phi_j^1\}_{j=1}^{d_1},  \ldots,  \{ \phi_j^K\}_{j=1}^{d_K} \} $ by using Algirithm~\ref{alg_tgCVT}.
                 \STATE For $i=1, 2, \ldots,n$, if $u(t,\mathbf{x}, \bm{\xi}_i) \in \widetilde{ \mathcal{U}}_k$, then denote the label of $\bm{\xi}_i$ as $\iota_i=k$, where $k\in \{ 1, 2, \ldots, K\}$.
                 \STATE Use the input set $X$ and the labels $\{ \iota_i\}_{i=1}^n$ to form the training data set $\mathbb{D}$, then learn the prior probability distributions $\{ \pi_k \}_{k=1}^K$ and conditional probability density functions $\{ f_k \}_{k=1}^K$.
                 \STATE For a given new input $\bm{ \xi}$, compute the values of $ \{  \pi_k, f_k (\bm{ \xi}) \}_{k=1}^K$, then assign $\bm{ \xi}$ to $\Gamma_k $ if (\ref{classifier2}) holds.
                 \STATE Obtain the reduced states $\{ \alpha_j \}_{j=1}^{d_k}$ by solving the system (\ref{ROM}) with $k$-th group CPOD basis functions $\{ \phi_j^k \}_{j=1}^{d_k}$, then the CPOD-NB approximate solution $\widetilde{u}^K$ of $\bm{ \xi}$ can be represented as (\ref{gCVT_sol}).
        \end{algorithmic}
\end{algorithm}        
        
\begin{remark}
       For a given input $\bm{\xi}$, in the CPOD-NB method, we hope to find the set of CPOD basis functions such that the error between its finite element solution and the reduced-order solution is the smallest. Therefore, the true label of input $\bm{\xi}$ can be defined as
        \begin{equation}   \label{true_label}
                i = \mathop{ \arg\min}_{ 1 \leq k \leq K }   \bigg\| u(t,\mathbf{x}, \bm{\xi}) - \widetilde{\Pi}_k^{d_k} u(t,\mathbf{x}, \bm{\xi})  \bigg\|^2_{\mathcal{L}^2([0,T]; L^2 (D)) },    
        \end{equation} 
        and the corresponding $\widetilde{\mathcal{U}}_i$ and $\{ \phi_j^i\}_{j=1}^{d_i}$ are called the true best-matched set and true best-matched basis functions of solution $u(t,\mathbf{x}, \bm{\xi})$, respectively.
\end{remark}

\section{Error estimation}    \label{Section_error}

In this section, we first give the error estimation of the SROM based on CPOD-NB method, and then introduce the estimation method of error rate of the naive Bayes pre-classifier.     
        
\subsection{Error estimation of CPOD-NB based SROM} 

In order to characterize the validity of the CPOD-NB model, the error between the full finite element solution $u$ and the CPOD-NB approximate solution $\widetilde{u}^K $ is defined as
\begin{equation}  \label{error_E_definition}
        \widetilde{\mathcal{E} }_K = \mathbb{E} \left[  \| u - \widetilde{ u}^K  \|^2_{ \mathcal{L}^2( [0,T]; L^2(D) ) }  \right] 
\end{equation}
and
\begin{equation}   \label{error_V_definition}
       \widetilde{ \mathcal{V} }_K = \mathbb{V}  \left[  \| u - \widetilde{ u}^K  \|^2_{ \mathcal{L}^2( [0,T]; L^2(D) ) }  \right] ,
\end{equation}
where  $\mathbb{V}  [\cdot] $ represents the variance.       
        
The error estimation of the CPOD-NB reduced-order solution is given in following theorem.
\begin{theorem}
         In the naive Bayes pre-classifier, if the random input $\bm{\xi}$ can always get the true label with the maximum posterior probability, then there exist  constants $C_1, C_2 >0$, such that with probability close to one, the space-time $L^2(D)$-norm error $\widetilde{ \mathcal{E}}_K$ between the finite element solution $u$ and the CPOD-NB approximate solution $\widetilde{ u}^K$ satisfies
        \begin{equation}
                 \widetilde{\mathcal{E} }_K  \leq  \sum_{k=1}^K  \left(  \frac{T n_k }{n}   \sum_{j=d_k+1}^{J n_k}  \sigma^k_j  \right)  +  C_1  \sqrt{   \widetilde{ \mathcal{V}}_K /n } + C_2 \frac{T \Delta t }{2},
        \end{equation}
        where $C_2$ depends on the regularity of $\| u(t) - \widetilde{u}^K (t) \|^2_{L^2(D)}$, while constant $C_1$ is universal.
\end{theorem}             
        
\begin{proof}
       By using the MC method, the error can be estimated by
       \begin{equation*}
                \widetilde{ \mathcal{E}}_K = \frac{1}{n} \sum_{i=1}^n \|  u (t,\mathbf{x}, \bm{\xi}_i )- \widetilde{u}^K (t,\mathbf{x}, \bm{\xi}_i ) \|^2_{ \mathcal{L}^2( [0,T]; L^2(D) ) } + \widetilde{ \mathcal{E}}_s,
       \end{equation*}
       where  $\widetilde{ \mathcal{E}}_s$ denotes statistical error and satisfies 
        \begin{equation*}
                \widetilde{ \mathcal{E}}_s  \sim  N (0, \widetilde{ \mathcal{V}}_K / n)
        \end{equation*}
        according to the central limit theorem. For a constant $C_1 \geq 1.65$, the inequality
        \begin{equation*}
                    \vert  \widetilde{ \mathcal{E}}_s \vert  \leq  C_1  \sqrt{ \widetilde{ \mathcal{V}}_K / n }
        \end{equation*}
        can hold with probability close to 1.  Then using data $\widehat{U}$ and its clustering results, the following can be obtained
        \begin{align*}
                   \widetilde{ \mathcal{E}}_K - \widetilde{ \mathcal{E}}_s = \frac{1}{n}  \sum_{k=1}^K   \sum_{u \in \widetilde{ \mathcal{U}}_k}  \|  u - \widetilde{u}^K  \|^2_{ \mathcal{L}^2( [0,T]; L^2(D) ) }.
        \end{align*}
        The snapshots obtained at equal time intervals are used to approximate the time integral, that is
        \begin{equation*}
                 \|  u - \widetilde{u}^K  \|^2_{ \mathcal{L}^2( [0,T]; L^2(D) ) }  =  \Delta t  \sum_{j=1}^{J}   \|  u (t_j) - \widetilde{u}^K (t_j)   \|^2_{ L^2(D) } + R[u,\widetilde{u}^K],
        \end{equation*}
        where time step $\Delta t = T/J$, $t_0=0$ and $t_{j+1} = t_j + \Delta t$ for $j=0, 1, \ldots, J-1$. $R[u,\widetilde{u}^K]$ is the residual of the approximation which depends on the regularity of $f(t;u) := \|  u (t) - \widetilde{u}^K (t)   \|^2_{ L^2(D) }$ and satisfies
         \begin{equation*}
                R[u,\widetilde{u}^K] = \frac{T \Delta t}{2} f^{\prime} (\eta;u)
         \end{equation*}
         for some $\eta \in (0,T)$. Therefore, 
         \begin{equation*}
                 \widetilde{ \mathcal{E}}_K - \widetilde{ \mathcal{E}}_s  =    \frac{ \Delta t}{n} \sum_{k=1}^K   \sum_{u \in \widetilde{ \mathcal{U}}_k}   \sum_{j=1}^{J}   \|  u (t_j) - \widetilde{u}^K (t_j)   \|^2_{ L^2(D) }    
                 + \frac{1}{n}  \sum_{k=1}^K   \sum_{u \in \widetilde{ \mathcal{U}}_k}  \frac{T \Delta t}{2} f^{\prime} (\eta;u).
         \end{equation*}
         Let 
         \begin{equation*}
                 C_2 =  \max_{\substack{u \in \widehat{U} , \eta \in (0,T)} }  \vert f^{\prime} (\eta;u) \vert, 
         \end{equation*}
         then according to the energy (\ref{gCVT_err_MinimalValue})
         \begin{equation*}
                 \widetilde{ \mathcal{E}}_K - \widetilde{ \mathcal{E}}_s  \leq   \sum_{k=1}^K  \left(  \frac{J n_k  \Delta t}{n}   \sum_{j=d_k+1}^{J n_k}  \sigma^k_j  \right) + C_2 \frac{T \Delta t}{2}
          \end{equation*}
          holds, which completes the proof.
\end{proof}

\subsection{Error rate estimation of the naive Bayes pre-classifier}

In general, classification rules have their error rate.  When the Bayes classifier with the maximum posterior decision rule is used to classify the problem with known conditional probability density functions and prior probability distributions, its error rate should be fixed.  Next, we consider the error rate estimation of the naive Bayesian pre-classifier.
        
According to the statistical decision theory \cite{berger1985statistical}, denote the discriminant functions as
\begin{equation}   \label{discriminant}
        g_k ( \bm{ \xi} )  =  \pi_k  f_k( \bm{\xi} )    \qquad  k=1, 2, \ldots , K,
\end{equation}         
and their decision regions are defined by
\begin{equation}   \label{decision_region}
       \widetilde{ \Gamma}_k  =  \{ \bm{ \xi} \in \Gamma  \  \vert \  g_k ( \bm{ \xi} ) >  g_i ( \bm{ \xi} )  \text{ for } i=1, 2,  \ldots, K,  \  i  \neq  k \}  \qquad  k=1, 2, \ldots, K.
\end{equation}
Then the decision surface between regions $\widetilde{ \Gamma}_i$ and $\widetilde{ \Gamma}_j$ is given as
\begin{equation}   \label{decision_surface}
       \mathcal{S}_{ij} = \{ \bm{\xi} \in \Gamma \ \vert \ g_i ( \bm{\xi} ) = g_j ( \bm{\xi} ), \ i \neq j \}  \qquad  i,j  = 1, 2, \ldots, K.
\end{equation}
Note that the decision region set $\{ \widetilde{ \Gamma}_k \}_{k=1}^K$ is also a partition of the feature space $\Gamma$. Although we hope that it is consistent with the segmentation $\{ \Gamma_k \}_{k=1}^K$  in the modified t-gCVT so that the input samples can always be assigned  to the best subspace with the maximum posterior probability, it is difficult to achieve in practice due to the defects of the classifier itself and the lack of data.  Therefore, it is necessary to study the error rate of classifier.        
        
According to the classification rules of naive Bayes, its error rate $ \mathbb{P} (e)$ is the probability of assigning sample that belongs to subspace $\Gamma_k$ to other subspace $\Gamma_i$, where $i,k = 1, 2, \ldots, K$ and $i \neq k$. That is 
\begin{equation}   \label{error_rate1}
        \mathbb{P} (e)  
        =  \sum_{k=1}^K  \sum_{\substack{i=1\\ i\neq k} }^K \mathbb{P}  ( \bm{\xi} \in \widetilde{\Gamma}_i ,  \iota =k )
        =  \sum_{k=1}^K  \sum_{\substack{i=1\\ i\neq k} }^K  \mathbb{P}  ( \bm{\xi} \in \widetilde{\Gamma}_i  \vert  \iota =k )   \mathbb{P} (\iota =k)
        =  \sum_{k=1}^K  \sum_{\substack{i=1\\ i\neq k} }^K  \pi_k  \mathbb{P}_{ki} (e),
\end{equation}
where 
\begin{equation}   \label{error_rate_ki}
        \mathbb{P}_{ki} (e) 
        = \mathbb{P}  ( \bm{\xi} \in \widetilde{\Gamma}_i  \vert  \iota =k )  
        = \int_{\widetilde{\Gamma}_i}  \mathbb{P} (\bm{ \gamma} = \bm{\xi} \vert  \iota =k) d \bm{\xi}
        = \int_{\widetilde{\Gamma}_i}  f_k (\bm{\xi})  d \bm{\xi}.
\end{equation}
Then the correct rate of the classifier takes the form
\begin{equation}   \label{correct_rate}
        \mathbb{P} (c)  = 1- \mathbb{P}(e) 
        = \sum_{k=1}^K  \pi_k  \mathbb{P}_{kk} (e).
\end{equation}        
 
For high-dimensional stochastic problems, it is difficult to determine the decision regions $\{ \widetilde{\Gamma}_k \}$ and the decision surfaces $\{ S_{ij} \}$, so the calculation of integrals (\ref{error_rate_ki}) is a huge challenge. Here, a more practical method can be used to estimate the error rate for testing the performance of the classifier.

A test set $\mathbb{T}=\{ \bm{\xi}_i \}_{i=1}^N$ with size $N$ is randomly selected from the feature space $\Gamma$, and its components are mutually independent  and independent of the training data $\mathbb{D}$.
Let the total number of samples in $\Gamma_k$ be $N_k$ for $k=1, 2,  \ldots, K$, which satisfy $\sum_{k=1}^K N_k =N$. The number of samples belonging to subspace $\Gamma_k$ that are misjudged into subspace $\Gamma_i$ is denoted as $n_{ki}$ for $k,i=1, 2,  \ldots,K$ and $i \neq k$.
Obviously, $n_{ki}$ is a discrete random variable that obeys a binomial distribution and satisfies 
\begin{equation}   \label{binomial}
        \mathbb{P} (n_{ki}) = C_{N_k}^{n_{ki}}  \left[\mathbb{P}_{ki} (e) \right]^{n_{ki}}  \left[ 1- \mathbb{P}_{ki} (e) \right]^{N_k - n_{ki}} ,
\end{equation}
where $C_{N_k}^{n_{ki}} = \frac{N_k !}{n_{ki}! \left( N_k - n_{ki} \right)! }$.
By solving 
\begin{equation}   \label{solve_saddle}
        \frac{ \partial \ln \mathbb{P} (n_{ki})}{ \partial \mathbb{P}_{ki} (e)}  = 0
\end{equation}
can obtain the maximum likelihood estimation of $\mathbb{P}_{ki} (e)$ as
\begin{equation}   \label{MLE}
        \widehat{ \mathbb{P}}_{ki} (e) = \frac{ n_{ki} }{ N_k },
\end{equation}
which is also a random variable, and the mean has form
\begin{equation}   \label{P_E}
        \mathbb{E} \left[  \widehat{ \mathbb{P}}_{ki} (e) \right] 
        = \frac{ \mathbb{E} \left[ n_{ki} \right] }{N_k} 
        = \mathbb{P}_{ki} (e).
\end{equation}
Therefore, $\widehat{ \mathbb{P}}_{ki} (e)$ is an unbiased estimate of $ \mathbb{P}_{ki} (e)$, and further an unbiased estimate of $\mathbb{P} (e)$ can be obtained as
\begin{equation}   \label{error_rate2}
        \widehat{ \mathbb{P}} (e)  = \sum_{k=1}^K  \sum_{\substack{i=1\\ i\neq k} }^K  \pi_k  \widehat{\mathbb{P}}_{ki} (e).
\end{equation}
In numerical experiments of this work, we use formula (\ref{error_rate2}) to estimate the error rate of the naive Bayes pre-classifier.

\section{Stochastic Navier-Stokes equations}    \label{Section_NS}   

In this work, we use the proposed CPOD-NB based SROM to deal with stochastic flow over a backward-facing step \cite{gunzburger2011optimal} described as
\begin{align}
         \mathbf{u}_t -\frac{1}{Re} \Delta \mathbf{u} + ( \mathbf{u} \cdot \nabla ) \mathbf{u} + \nabla P &= 0  \qquad (0,T] \times D,  \label{NS_1} \\
         \nabla \cdot \mathbf{u} &= 0   \qquad (0,T] \times D,    \label{NS_2}
\end{align}
where $Re$ is the Reynolds number of the fluid, $ \mathbf{u} ( t, \mathbf{x}) = (u_1, u_2)^ \top $ and $P ( t, \mathbf{x})$ denote the velocity and pressure fields, respectively. The boundary of physical domain $D$ is denoted by $ \partial D $, which consists of six parts as depicted in Figure \ref{F_domain_D}. For $ t \in(0,T]$, the boundary conditions are given by 
\begin{align}
        \mathbf{u} &= (u_{ \text{in}}, 0)^ \top    \hspace{25pt}  \text{on } \partial D_{i},  \label{D_in} \\
        \mathbf{u} &= (0,0)^ \top            \quad \qquad   \text{on } \partial D_t \cup \partial D_b \cup \partial D_d \cup \partial D_c , \label{D_m}\\
        P \mathbf{n} - \frac{1}{Re} \frac{\partial \mathbf{u} }{ \partial \mathbf{n}} &= (0,0)^ \top      \quad \qquad   \text{on } \partial D_{o}, \label{D_out} 
\end{align}
and the initial velocity field satisfies
\begin{equation}  \label{u_0}
        \mathbf{u} (0,\mathbf{x}) = \mathbf{u}_0 (x,y) = \left\{
        \begin{array}{ll}
                u_{ \text{in}} (0,\mathbf{x})  & \text{ on } \partial D_{i},  \\
                0  & \text{ otherwise}.
        \end{array}
        \right.
\end{equation}
        
Assume that the fluid can be injected along $\partial D_{i}$, so $u_{ \text{in}} \geq 0$ is required. Further assume that the injected fluid contains uncertainties. Thus, for a properly defined probability space $(\Omega, \mathcal{F}, \mathbb{P})$, $u_{ \text{in} }$ can be  modelled with random variable $\omega \in \Omega$ as
\begin{equation}   \label{u_in}
        u_{ \text{in}} = A (t, \bm{\xi} ( \omega)) h(y),
\end{equation}
where $A$ is a time-dependent parameter that determines the strength of the parabolic inflow velocity profile $h(y)$.  For the sake of simplicity, denote $A (t, \bm{\xi} ( \omega))$ as $A (t, \bm{\xi} )$ or $A (t; \omega)$.
        
\begin{figure}[ht]
       \centering   
       \includegraphics[width=0.7 \textwidth,trim=0 0 0 0,clip]{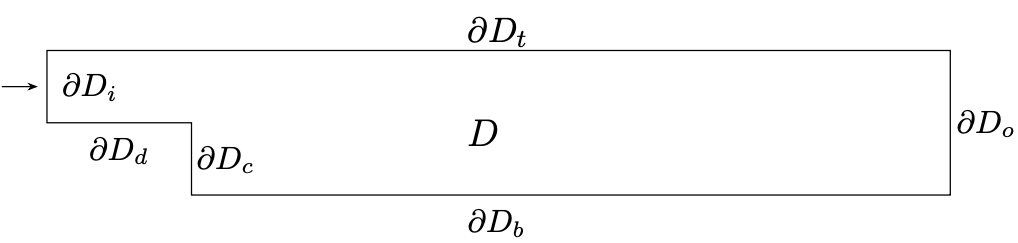}
      \caption{Physical domain $D$ of Navier-Stokes equation} 
\label{F_domain_D}
\end{figure}

\subsection{Full discrete and Newton linearization}

In this paper, finite element method, $\theta$-scheme and Newton's method are used for spatial discretization, time discretization, and the linearization of nonlinear convective term, respectively.
      
Let $\mathcal{T}^h$ be a shape-regular triangular finite element mesh of domain $D$, which is parameterized by mesh width $h=\max_{G \in \mathcal{T}^h} \text{diam} (G)$, where $G$ is a typical finite element in the triangulation $\mathcal{T}^h$.  The finite element mesh used in this work is shown in Figure \ref{F_mesh}. For vector valued function $\mathbf{u}$,  define the following finite element spaces 
\begin{align*}
         & \mathbf{V}^h =\{ \mathbf{v}^h = (v^h_1, v^h_2)^ \top: v^h_i \in C^0 (\bar{D}), v^h_i \vert _G \in \mathcal{P}^2 \text{ for any } G \in \mathcal{T}^h, i=1,2 \},  \\
         & \mathbf{V}^h_0 =\{ \mathbf{v}^h \in \mathbf{V}^h:  v^h_i =0  \text{ on } \partial D  \setminus \partial D_{o}   \text{ for }  i=1,2 \},   \\
         & Q^h = \{ q^h: q^h \in C^0 (\bar{D}), q^h \vert _G \in  \mathcal{P}^1  \text{ for any } G \in \mathcal{T}^h\},
\end{align*}
where $\mathcal{P}^r$ denotes the polynomial space with degree less than or equal to $r$, $r \in \mathbb{N}^+$. The Taylor-Hood finite element spaces are considered in our computation, i.e. quadratic finite element space for velocity field $\mathbf{u}$ and linear finite element space for pressure field $ P$.

Let $ \tau_m = \{ t_i\}_{i=0}^m $ be a partition of $ [0,T]$ with equal interval $ \delta t =T/m$, where $ t_0 =0$ and $t_i =t_{i-1} + \delta t$ for $ i=1, 2,  \ldots,m$. Then for $ i=0, 1, \ldots , m-1$, the linearized full discrete weak formulation of system (\ref{NS_1})-(\ref{D_out}) is given as:  find $\mathbf{u}^{i+1}_{\theta} \in \mathbf{V}^h$ and $P^{i+1}_{\theta} \in Q^h$ such that 
\begin{align}  \label{weak_formulation_theta}
        \frac{1}{\theta \Delta t}  \int_D \left( \mathbf{u}^{i+1}_{\theta} -  \mathbf{u}^i  \right) \mathbf{v}  d \mathbf{x}  
        &+ \frac{1}{Re} \int_D \nabla  \mathbf{u}^{i+1}_{\theta} : \nabla \mathbf{v} d \mathbf{x}  
        + \int_D \mathbf{u}^{i+1}_{\theta} \cdot \nabla \mathbf{u}^i \cdot \mathbf{v} d \mathbf{x}     \notag  \\
        &+ \int_D \mathbf{u}^i \cdot \nabla \mathbf{u}^{i+1}_{\theta} \cdot \mathbf{v} d \mathbf{x}  
        - \int_D P^{i+1}_{\theta} \nabla \cdot \mathbf{v} d \mathbf{x}
        = \int_D \mathbf{u}^i \cdot \nabla \mathbf{u}^i \cdot \mathbf{v} d \mathbf{x}  \\
        \int_D q \nabla \cdot \mathbf{u}^{i+1}_{\theta} d \mathbf{x} & = \left\{ 
        \begin{array}{ll}
               0                                                 & \text{ if } i>0, \\
               (1-\theta) \int_D q \nabla \cdot \mathbf{u}^0 d \mathbf{x}   & \text{ if } i=0,   \notag \\
        \end{array}
        \right.
\end{align}
for any test functions $\mathbf{v}^h \in \mathbf{V}^h_0$ and $q^h \in Q^h$.  Here, $ \mathbf{u}^i = \mathbf{u} (t_i, \mathbf{x})$ and $\theta$ is taken as $\frac{1}{2}$. By solving linear system (\ref{weak_formulation_theta}), the pair $( \mathbf{u}^{i+1}, P^{i+1})$ can be recovered from
\begin{equation}
         \mathbf{u}^{i+1} = 2 \mathbf{u}^{i+1}_{\theta} -  \mathbf{u}^i
         \qquad   \text{ and }   \qquad
         P^{i+1} =2 P^{i+1}_{\theta} -  P^i.
\end{equation}      

\begin{figure}[h]
        \centering 
        \includegraphics[width=0.6 \textwidth,trim=0 137 0 150,clip]{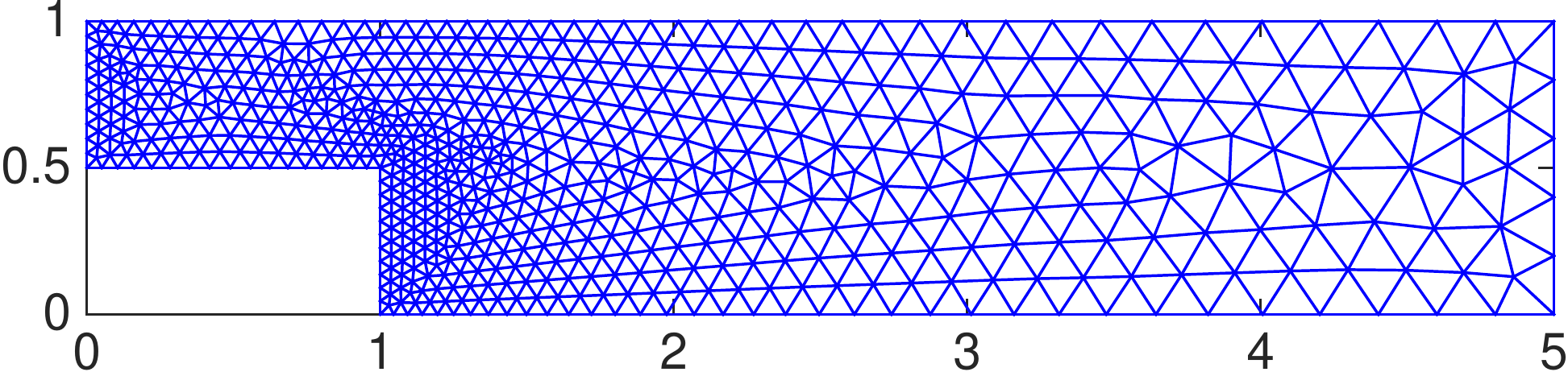}
        \caption{Finite element mesh with $h=0.2$, 1279 triangles and 703 vertices   }  \label{F_mesh}
\end{figure}

\subsection{Modified velocity field}

Here, instead of the original finite element solution $\mathbf{u}$, a CPOD-NB model is constructed for the modified velocity field with homogeneous Dirichlet boundaries.
        
Denote the solutions of the steady-state version of Navier-Stokes system (\ref{NS_1})-(\ref{D_out}) with constant strengths $A=a_1$ and $A=a_2$ in inflow velocity $u_{\text{in}}$ as $\mathbf{u}_{a_1}$ and $\mathbf{u}_{a_2}$, respectively.  Let
\begin{equation}   \label{st_NS}
        \mathbf{w}=  \frac{\mathbf{u}_{a_1} - \mathbf{u}_{a_2} }{a_1 -a_2},
\end{equation}
and denote the average of the velocity field as 
\begin{equation}   \label{average_u}
        \overline{ \mathbf{u} } ( \mathbf{x}) = \frac{1}{n} \sum_{i=1}^n \left( \frac{1}{J} \sum_{j=1}^J \left( \mathbf{u} ( t_j, \mathbf{x}, \bm{\xi}_i ) - A( t_j, \bm{\xi}_i) \mathbf{w}( \mathbf{x}) \right) \right).
\end{equation}
Then the modified state is given by
\begin{equation}   \label{modified_state}
        \mathbf{v} ( t, \mathbf{x}, \bm{\xi} ) = \mathbf{u} ( t, \mathbf{x}, \bm{\xi} ) - \overline{ \mathbf{u} } ( \mathbf{x})  - A(t, \bm{ \xi}) \mathbf{w} ( \mathbf{x} ), 
\end{equation}
which satisfies $\mathbf{v}=0$ on $ \partial D \setminus \partial D_o$.        
        
Using the modified t-gCVT method for modified state $\mathbf{v}$, we can obtain $K$ sets of basis functions $\{  \{ \phi_j^1 ( \mathbf{x}) \}_{j=1}^{d_1} ,  \ldots, \{ \phi_j^K ( \mathbf{x}) \}_{j=1}^{d_K} \}$.  If the class label of a given input $\bm{\xi}$ is $k$, the original system (\ref{weak_formulation_theta}) can be reduced to a $d_k$-dimensional ordinary differential equations by using $\{ \phi_j^k ( \mathbf{x}) \}_{j=1}^{d_k}$, then the reduced states $\{ \alpha_j (t, \bm{\xi}) \}_{j=1}^{d_k}$ can be calculated by Runge-Kutta method, finally the approximation of the original velocity field can be represented as
\begin{equation}
        \mathbf{u} ( t, \mathbf{x}, \bm{\xi} ) = \overline{ \mathbf{u} } ( \mathbf{x}) + A(t, \bm{ \xi}) \mathbf{w} ( \mathbf{x} ) + \sum_{j=1}^{d_k}  \alpha_j (t, \bm{\xi})  \phi_j^k ( \mathbf{x}).
\end{equation}

\section{Numerical experiments}   \label{Section_Numerical}

To illustrate the feasibility and effectiveness of the proposed CPOD-NB model, we provide comparisons with the standard POD method (i.e. $K=1$). All computations were performed using MATLAB R2017a on a personal computer with 2.3 GHz CPU and 256 GB RAM.                
        
In our computation, the physical domain $D$ and its triangulation used in the finite element method are shown in the Figure \ref{F_mesh}. 
The Reynolds number $Re$ is taken as 500.
The finite element solutions of steady-state version of Navier-Stokes system associated with $a_1 = 2$ and $a_2=1$ are used to generate the modified state, as defined in (\ref{modified_state}).
The time interval $ [0, T] $, $T=2$, is divided by the time step $ \delta t = 1/200 $, and the modified snapshots are obtained at each time point for computing the modified distance, i.e. $\Delta t = \delta t$.
The parabolic profile $h(y)$ of inflow velocity has form
\begin{equation}   \label{h_profile}
         h(y) =  (1-y)(y-0.5).   
\end{equation}
Let the random input of system (\ref{NS_1})-(\ref{u_0}) be the time-discrete form of strength $A(t;\omega )$, i.e.
\begin{equation}
         \bm{ \xi} (\omega) = [ \xi_1( \omega), \ldots, \xi_{m+1} (\omega)]^ \top = [A(t_0; \omega ),  A(t_1; \omega ), \ldots, A(t_m; \omega )]^ \top,
\end{equation}
where $t_0=0$, $t_j = t_{j-1} + \delta t$ for $j=1, 2, \ldots, m$.
The number of CPOD basis functions of each class is not necessarily equal in our method, but in order to compare with the standard POD method, it is set to be equal and determined by the 97\% cumulative energy ratio of the standard POD basis functions.                 
        
In addition to estimating absolute error statistics $\widetilde{ \mathcal{E}}_K$ and $\widetilde{ \mathcal{V}}_K$, we also give the estimations of relative error statistics defined as
\begin{equation}   \label{relateive_error_E}
       \widetilde{\mathcal{E} }_K^r = \mathbb{E} \left[  \frac{\| u - \widetilde{ u}^K  \|^2_{ \mathcal{L}^2( [0,T]; L^2(D) ) }}{ \| u \|^2_{ \mathcal{L}^2( [0,T]; L^2(D) ) } }  \right] 
\end{equation}
and 
\begin{equation}   \label{relateive_error_V}
       \widetilde{ \mathcal{V} }_K^r = \mathbb{V}  \left[  \frac{\| u - \widetilde{ u}^K  \|^2_{ \mathcal{L}^2( [0,T]; L^2(D) ) }}{ \| u \|^2_{ \mathcal{L}^2( [0,T]; L^2(D) ) } }  \right] .
\end{equation}  
These statistics are all estimated by the MC method.
Next, we consider two different strengths $A$, one is expanded by the trigonometric functions, and the other is hat-type functions of different heights with white noise.

\subsection{Strength expanded by trigonometric functions}   \label{sec_eg1}

In this experiment, the strength $A$ is given by
\begin{equation}     \label{inlet_trigonometric}
        A (t; \omega )= A_0 (t) + \sigma \sum_{i=1}^N \delta_i \left[  \sin(\pi i t) \eta_i^{(1)}(\omega)  + \cos( \pi i t) \eta_i^{(2)}(\omega)  \right],
\end{equation}
where the mean strength $A_0(t) \equiv 70$, amplification factor $\sigma = 12$, the number of expanded terms $N=100$, $\delta_i = 1/i$ for $i=1, 2, \ldots, N$, and $\{ \eta_i^{(j)} \}_{i=1}^{N}$, $j=1,2$, are i.i.d. random variables and satisfy $\eta_i^{(j)} \sim \mathcal{N} (0,1)$.
Here, 300 samples of velocity field are used to generate the CPOD basis functions and train the naive Bayes pre-classifier, and the other 100 samples form the test set to estimate the error of the SROM based on the pre-classifier.         
        
\subsubsection{Generating CPOD basis functions}

Figure~\ref{F_population1} shows the clustering results of these 300 samples with modified t-gCVT method. On the left is the number of samples in each class, $n_k$, under different cluster numbers $K$. The middle is the corresponding energy defined in (\ref{gCVT_err_MinimalValue}), which gradually decreases with the increase of $K$. On the right is the logarithm of eigenvalues corresponding to the first 30 CPOD basis functions in each class. The dimensions and cumulative energy ratios used in this experiment are given in Table \ref{T_dim1}. On the whole, for $K=2$ and 3, the energy ratios of the CPOD basis functions generated by our method are higher than that of the standard POD method. It is not difficult to understand that the samples in each class are similar after clustering, so their eigenvalues  decay faster, which leads to the same number of basis functions can obtain more information. That is to say, some information that is ignored by standard POD method can be captured after clustering. The contours of the first four CPOD basis functions in every class are given in Figure \ref{F_PODbasis1}.
Note that the first basis functions of these six cases are similar because they all describe the main characteristics of the velocity field, but the remaining basis functions of $K=2$ and $K=3$ have obvious differences, which shows that the clustering method can capture the local characteristics of the flow.

\begin{figure}[ht]
        \centering 
        \includegraphics[width=0.30 \textwidth,trim=35 20 45 20,clip]{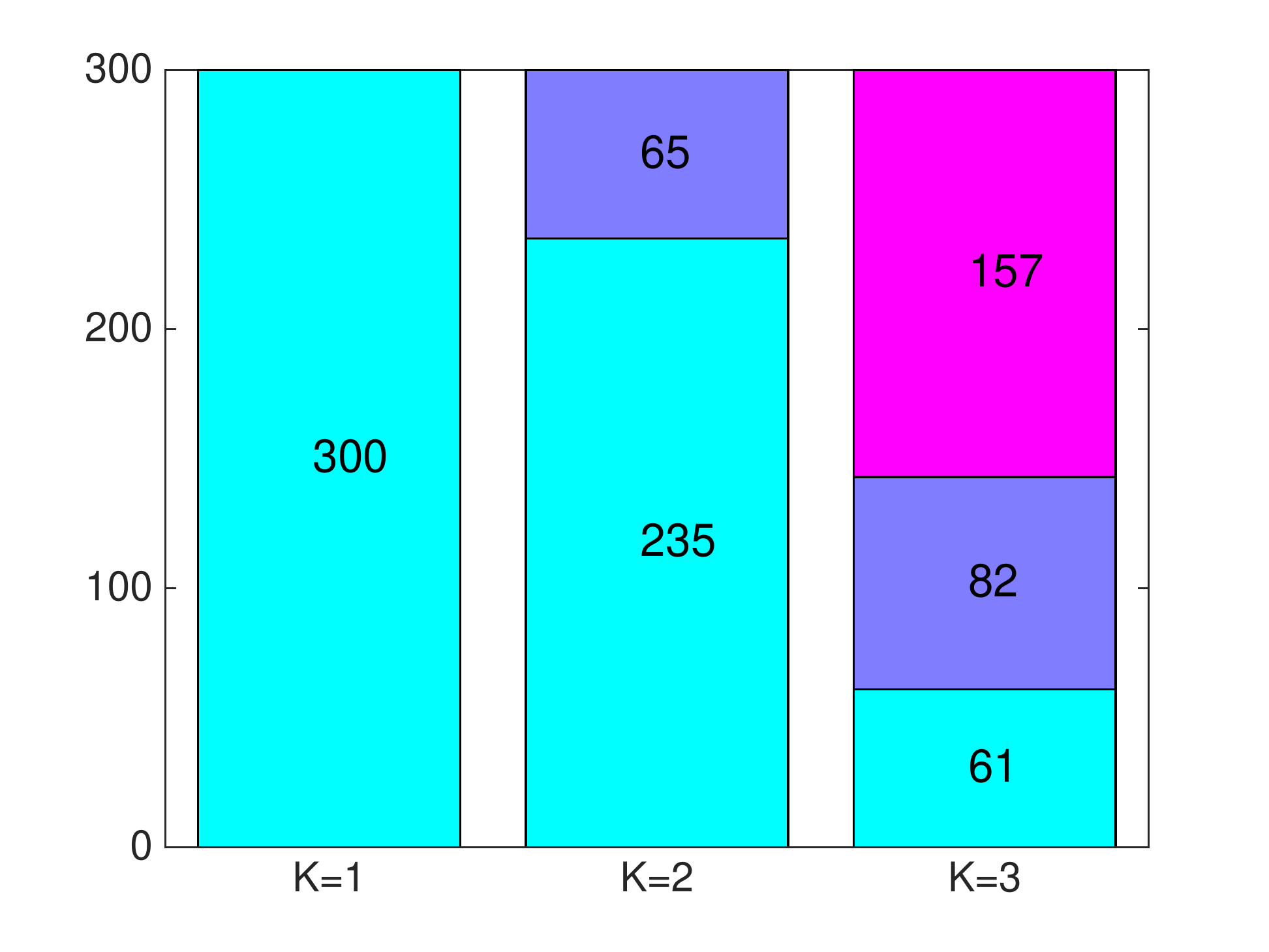}  
        \includegraphics[width=0.30 \textwidth,trim=20 20 45 20,clip]{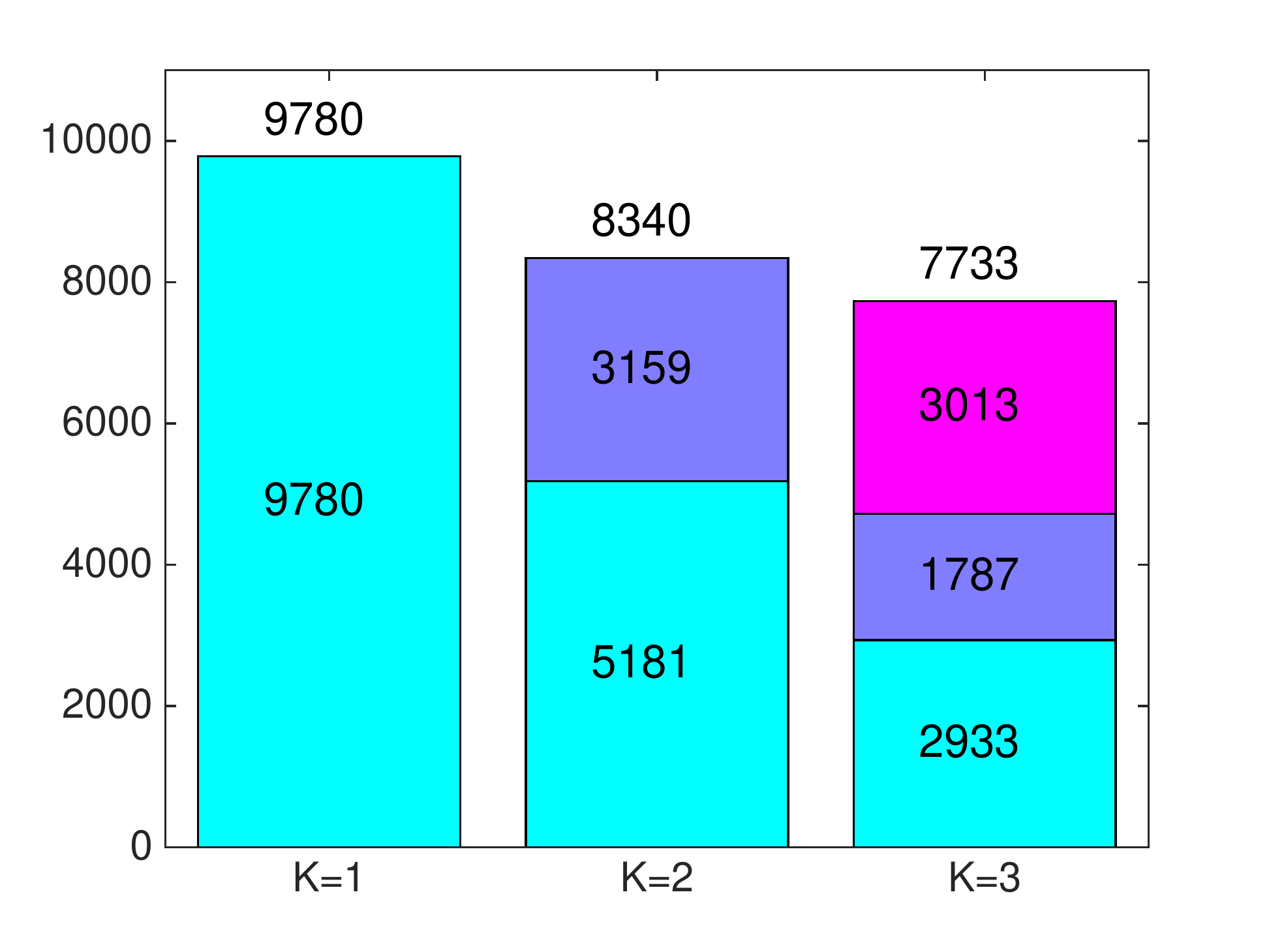}
        \quad
        \begin{rotate}{90}
                \hspace{35pt}   \footnotesize $ \log \lambda $
        \end{rotate}
        \includegraphics[width=0.30 \textwidth,trim=28 20 45 20,clip]{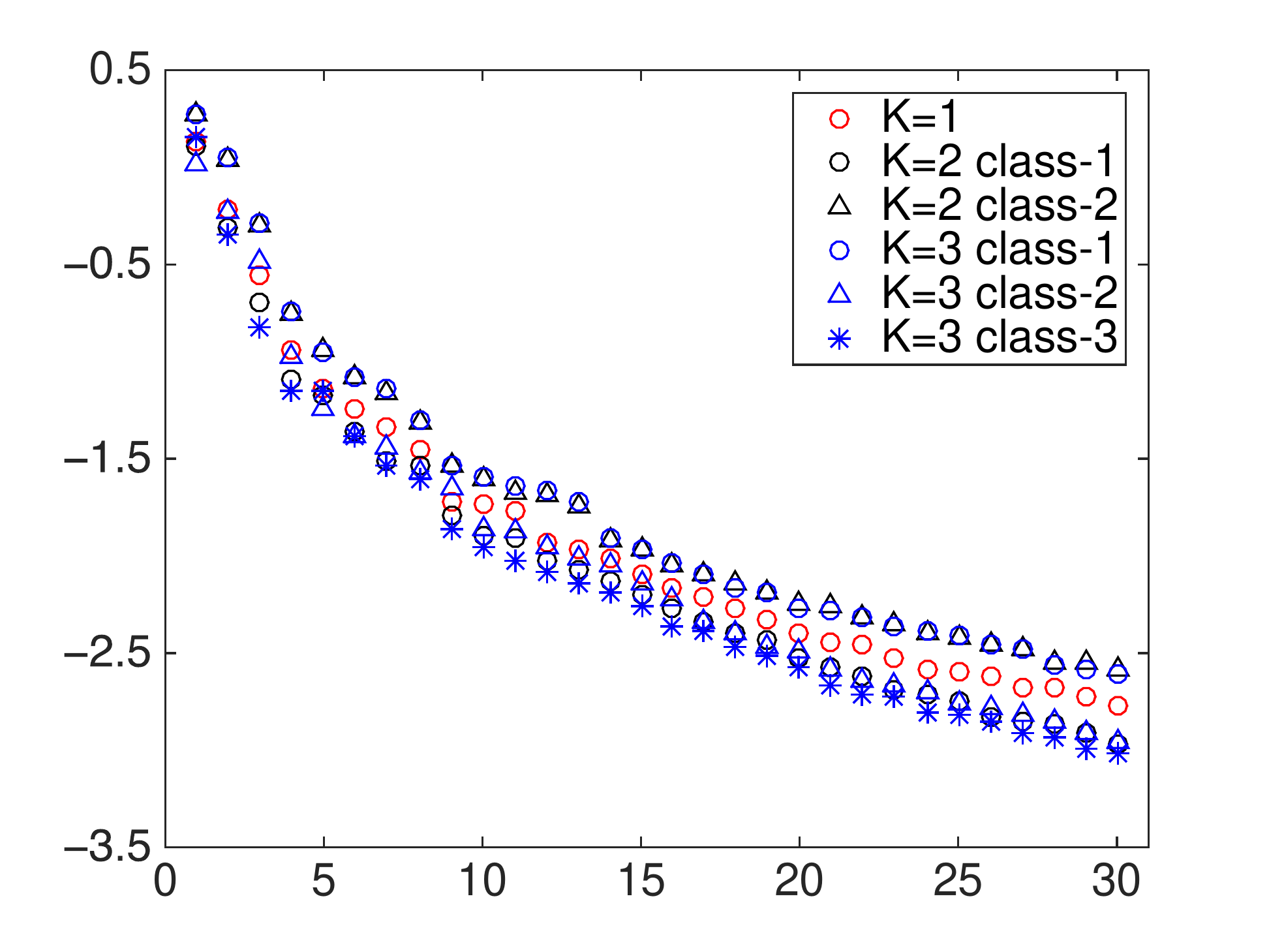}
        \caption{Population $n_k$ (left),  energy $\widetilde{\mathcal{E}}^{  \text{t-gCVT} }$ (middle) of data $\widehat{U}$, and the logarithm of eigenvalues (right) corresponding to the first 30 CPOD basis functions in each class for $K=1,2 $ and $3$   }  \label{F_population1}
\end{figure}
        
\begin{table}[ht]
      \centering
      \small
      \begin{spacing}{1}
      \caption{The dimension $d_k$ and cumulative energy ratio $\nu_k$ of CPOD basis functions in each class for $K=1, 2$ and $3$ }  \label{T_dim1}
      \begin{tabular}{@{}ccccccc@{}}
             \hline
             &  $ K=1 $
             &  \multicolumn{2}{c}{ $ K=2 $ }
             &  \multicolumn{3}{c}{ $ K=3 $ }    \\
             \hline
             class & - & 1 & 2 & 1 & 2 & 3  \\
             $d_k$ & 16 & 16 & 16 & 16 & 16 & 16 \\
             $\nu_k$ & 0.9704 & 0.9765  & 0.9713 & 0.9719  &  0.9768  &  0.9798  \\
             \hline
     \end{tabular}
     \end{spacing}
\end{table}        
        
\begin{figure}[H]
      \centering 
      \includegraphics[width=0.30 \textwidth,trim=50 30 50 0,clip]{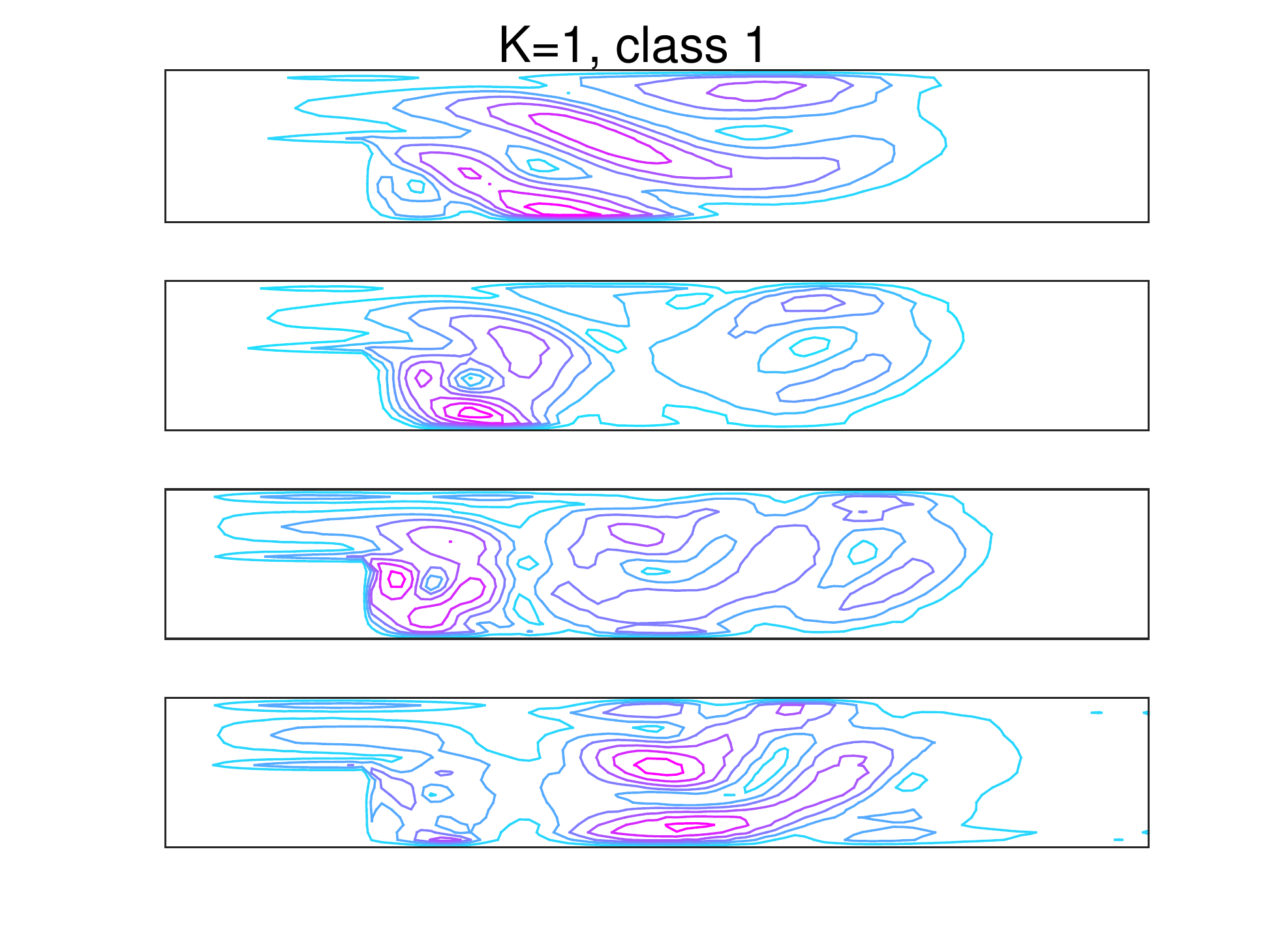}  
      \includegraphics[width=0.30 \textwidth,trim=50 30 50 0,clip]{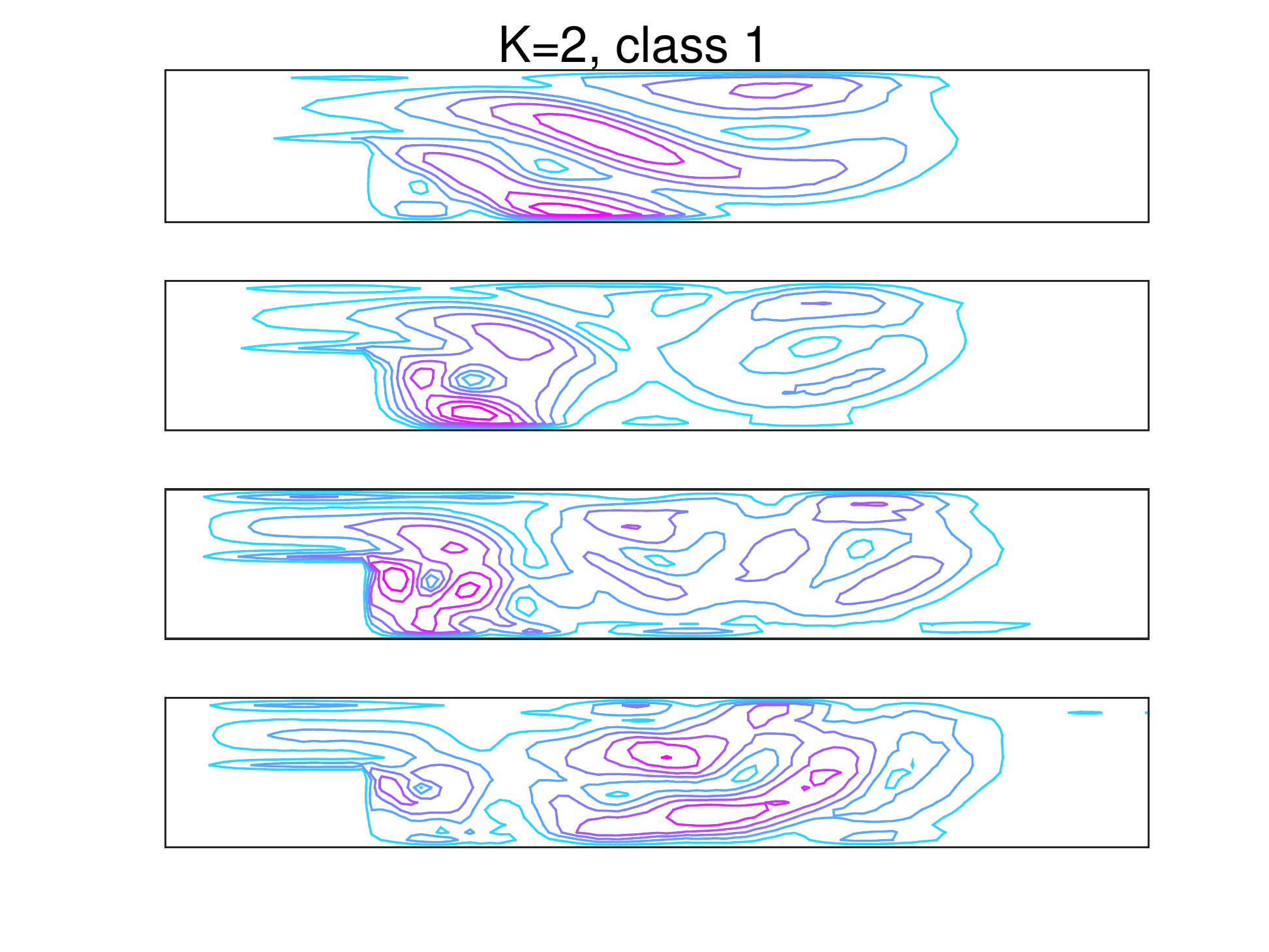} 
      \includegraphics[width=0.30 \textwidth,trim=50 30 50 0,clip]{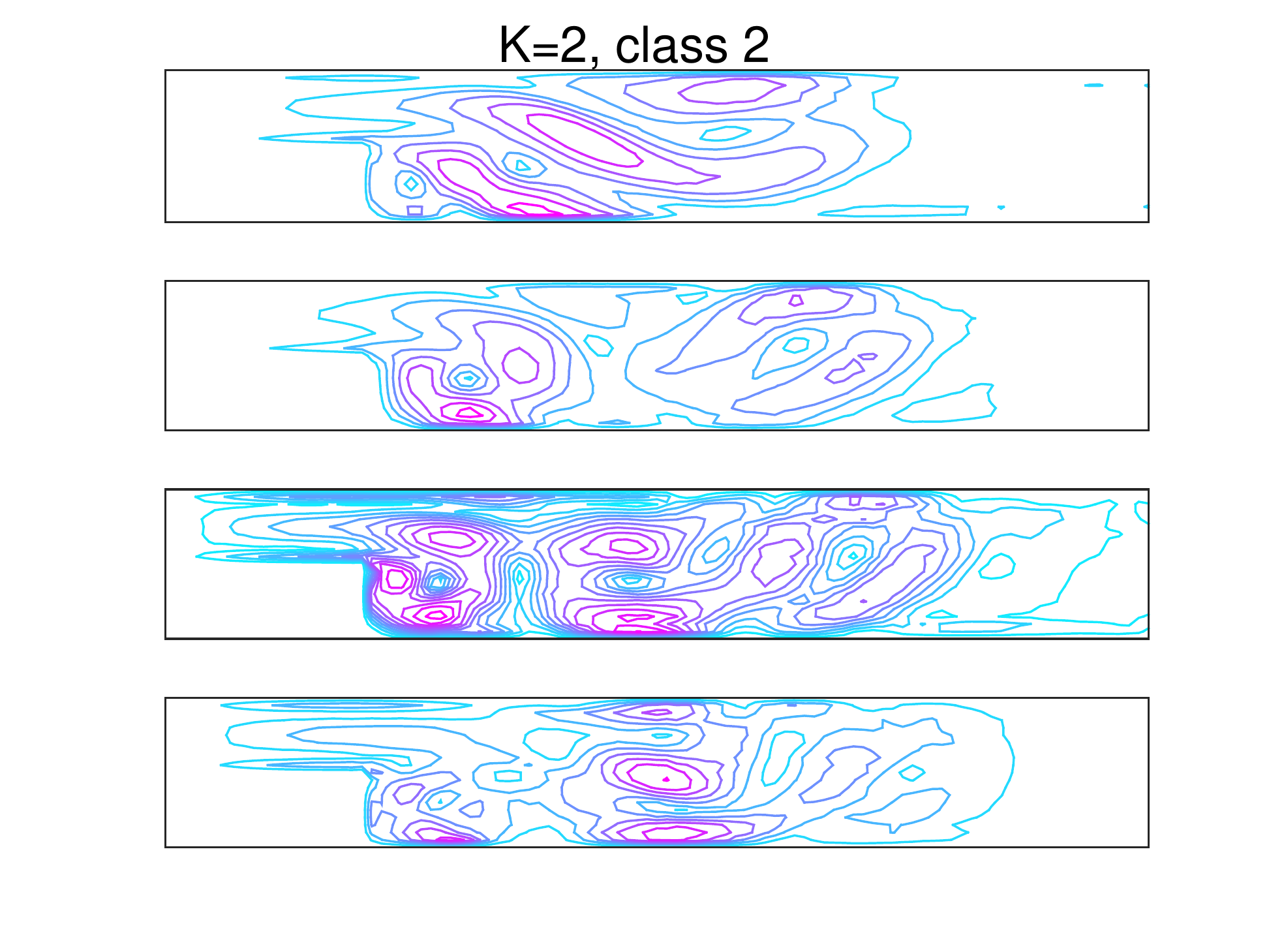} \\
      \includegraphics[width=0.30 \textwidth,trim=50 30 50 0,clip]{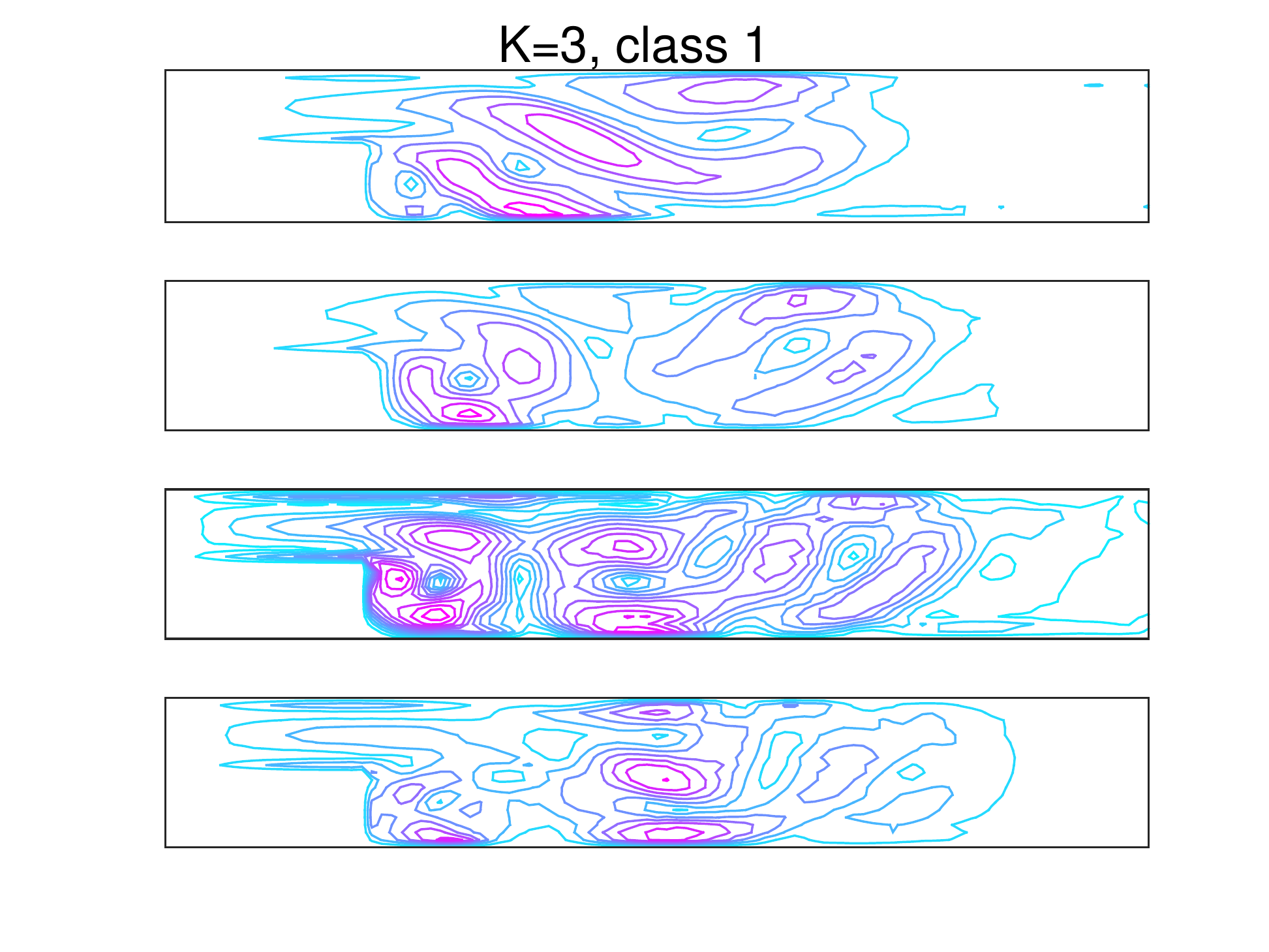}  
      \includegraphics[width=0.30 \textwidth,trim=50 30 50 0,clip]{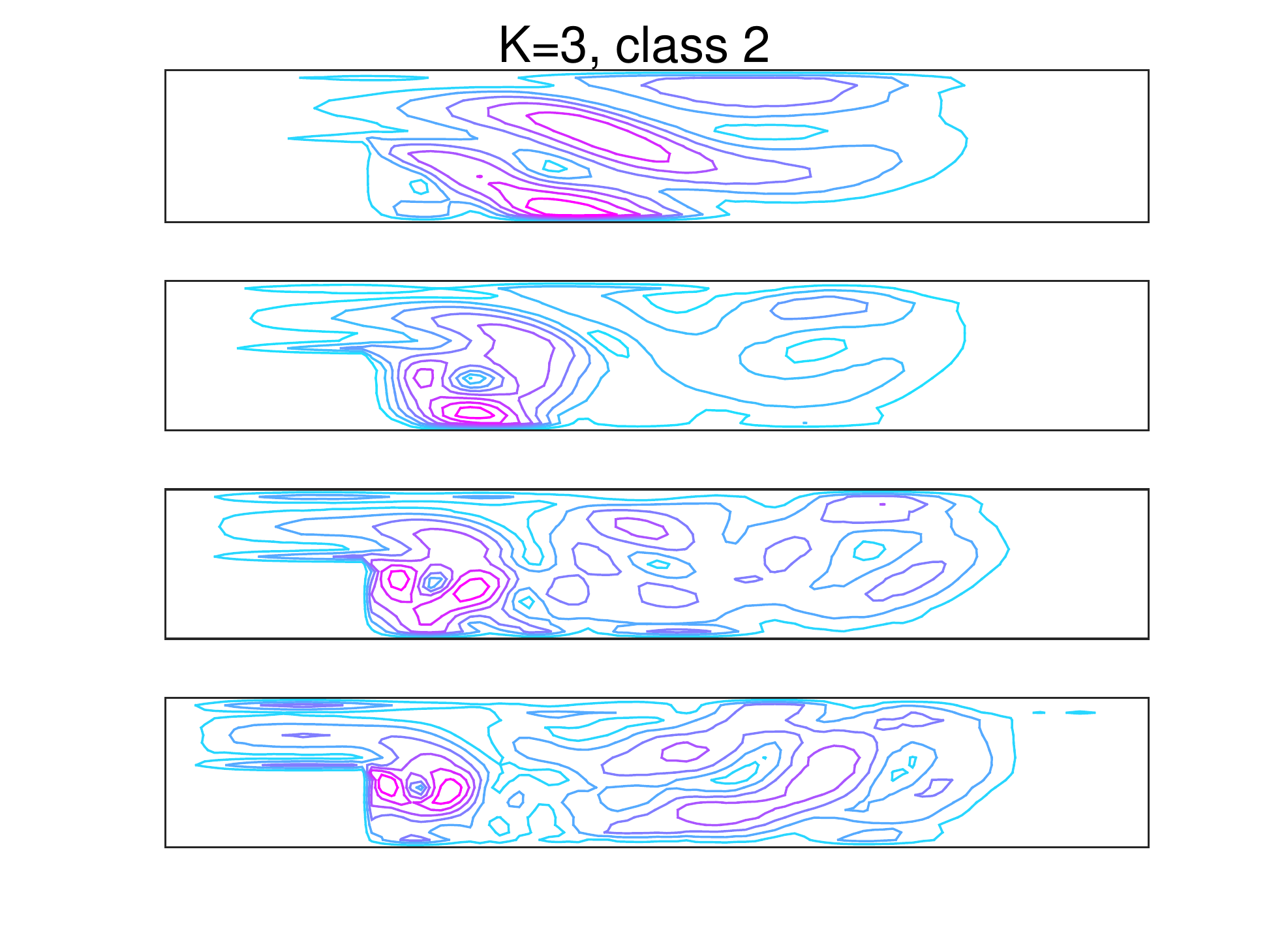} 
      \includegraphics[width=0.30 \textwidth,trim=50 30 50 0,clip]{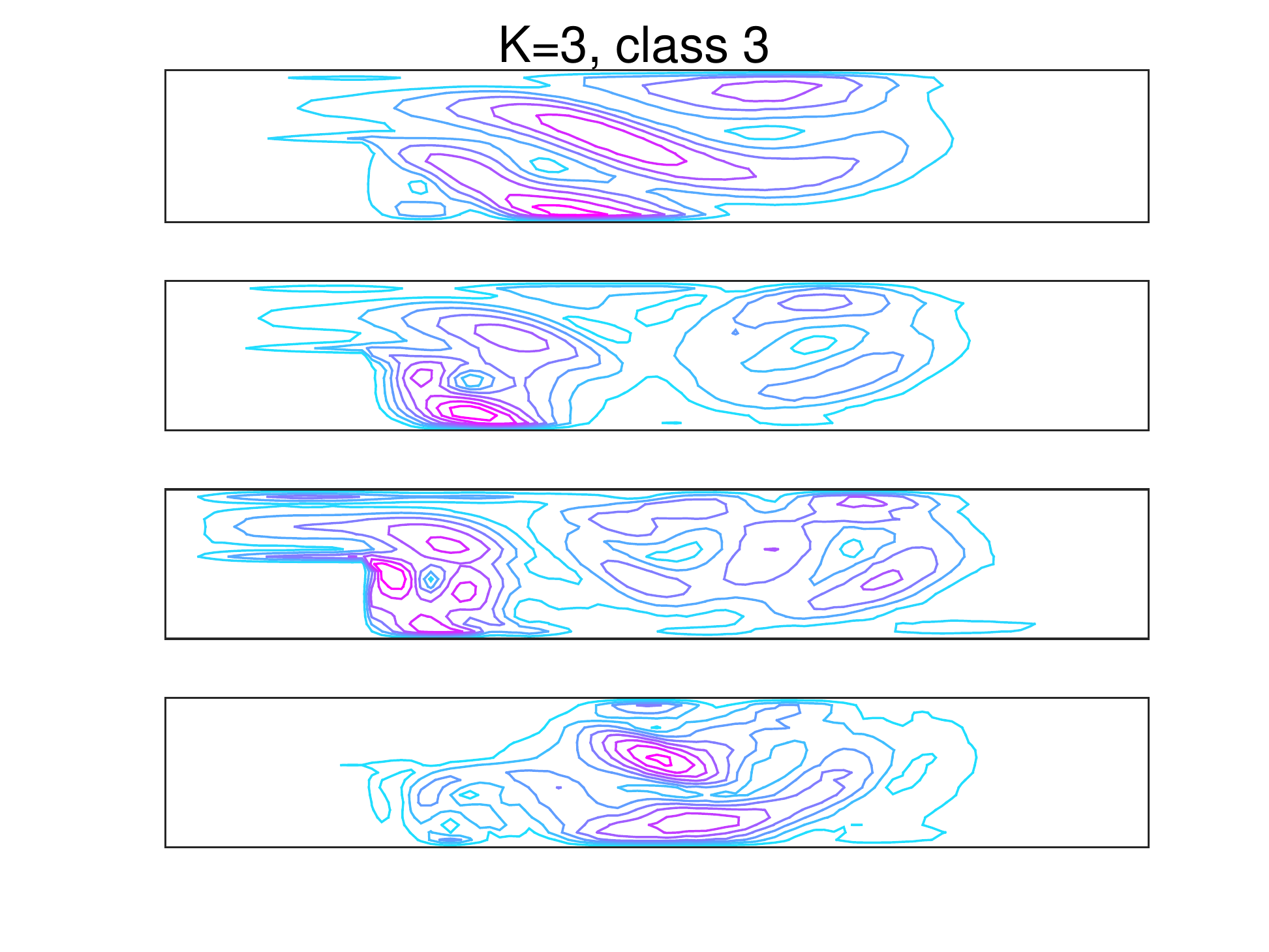} 
      \caption{Contours of the first four CPOD basis functions in each class for $K=1,2$ and $3$ }  \label{F_PODbasis1}
\end{figure}        
        
From the clustering results of modified t-gCVT, the labels of these 300 training samples are known. The errors of CPOD-based SROM that directly use the training data and their known labels are given in the Table~\ref{T_train_err1}, and the statistics of $L^2(D)$-norm error between finite element solution and CPOD reduced-order solution are shown in Figure~\ref{F_train_EVt1}. Clearly, when the class labels of samples are known, the CPOD-based SROM is more accurate and more stable than the standard POD-based SROM. This illustrates that it is feasible to use CPOD basis functions to improve the accuracy of the reduced-order model. Figure~\ref{F_train_simula1} gives the simulation results of two samples in the training set, which more intuitively shows the performance of the CPOD basis functions.        
        
\begin{table}[ht]
        \centering
        \small
        \begin{spacing}{1.3}
        \caption{Error estimates of CPOD-based SROM by using 300 labelled training data $\widehat{U}$ }   \label{T_train_err1}
        \begin{tabular}{@{}ccccc@{}}
              \hline
              $K$ & $\widetilde{\mathcal{E}}_K $ & $\widetilde{\mathcal{E}}_K^r $ & $\widetilde{\mathcal{V}}_K $ & $\widetilde{\mathcal{V}}_K^r $   \\
              \hline
              1 &  0.6736  &  3.0114\%  &  0.6788  &  0.1325\%   \\ 
              2 &  0.6229  &  2.7493\%  &  0.1879  &  0.0361\%   \\
              3 &  0.5516  &  2.4526\%  &  0.1280  &  0.0261\%   \\
              \hline 
       \end{tabular}
       \end{spacing}
\end{table}

\begin{figure}[ht]
        \centering 
        \begin{rotate}{90}
                \hspace{15pt}   \scriptsize $ \mathbb{E} \left[ \| \mathbf{u} - \widetilde{ \mathbf{u}}^K \|^2_{L^2(D)} \right] $
        \end{rotate}
        \hspace{1pt}
        \includegraphics[width=0.32 \textwidth,trim=20 20 30 0,clip]{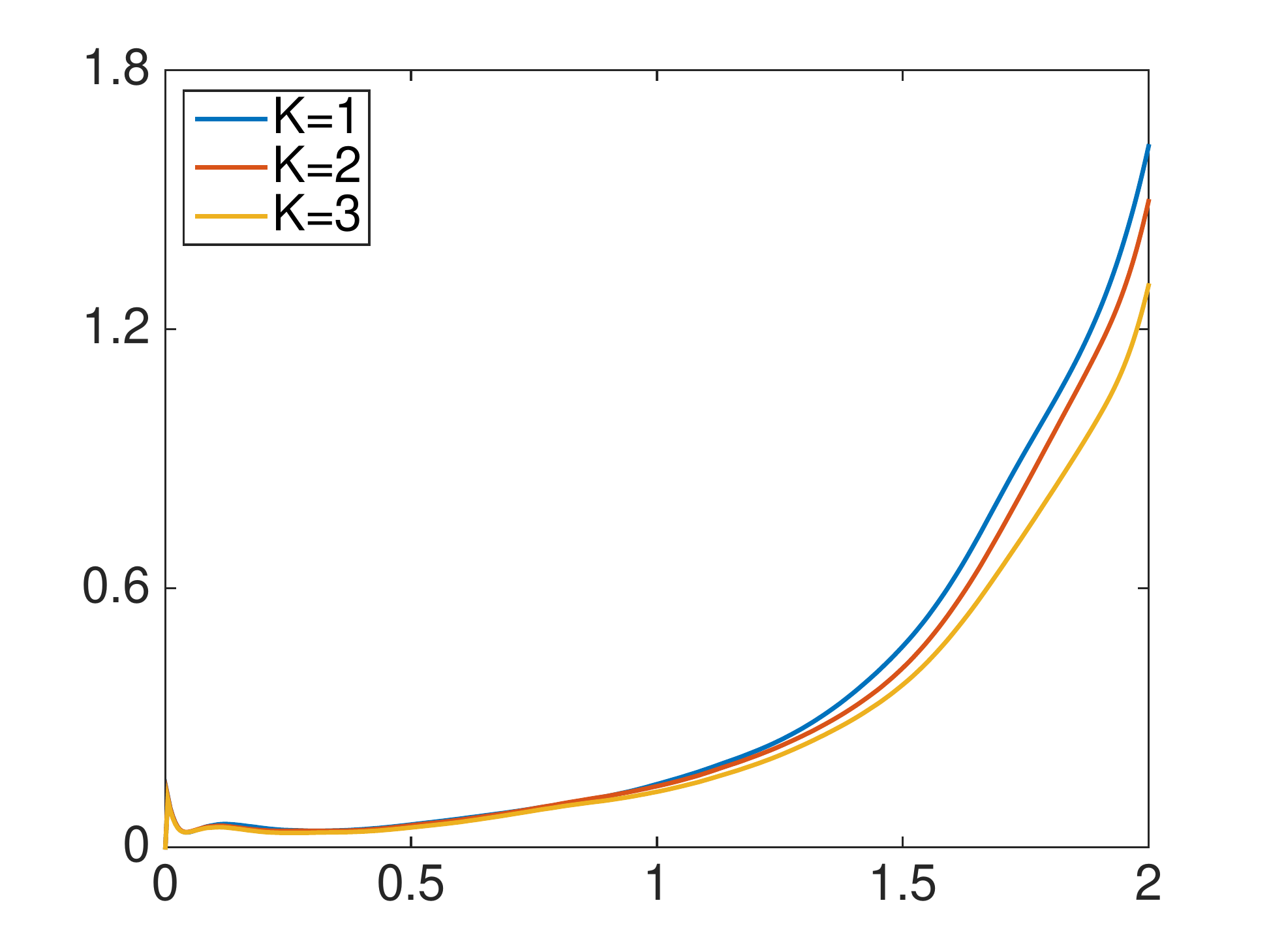} 
        \qquad
        \begin{rotate}{90}
                \hspace{15pt}   \scriptsize $ \mathbb{V} \left[ \| \mathbf{u} - \widetilde{ \mathbf{u}}^K \|^2_{L^2(D)} \right] $
        \end{rotate}
        \hspace{1pt}
        \includegraphics[width=0.32 \textwidth,trim=20 20 30 0,clip]{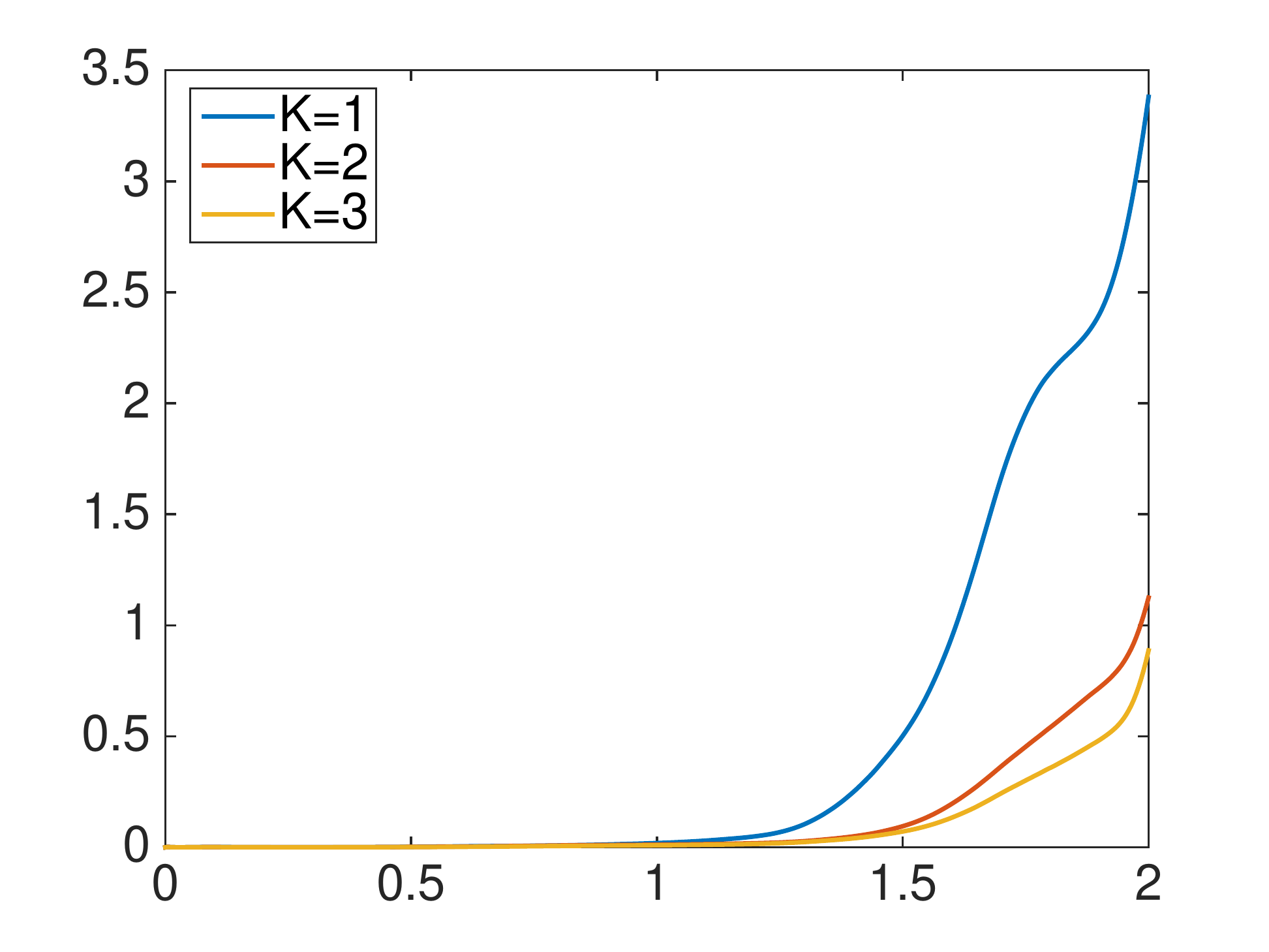} \\
        \footnotesize   \hspace{1.5em} $t$   \hspace{19em}  $t$
        \caption{Error estimates of CPOD-based SROM with different $K$}   \label{F_train_EVt1}
\end{figure}    

\begin{figure}[H]
        \centering 
        \includegraphics[width=0.30 \textwidth,trim=10 2 40 20,clip]{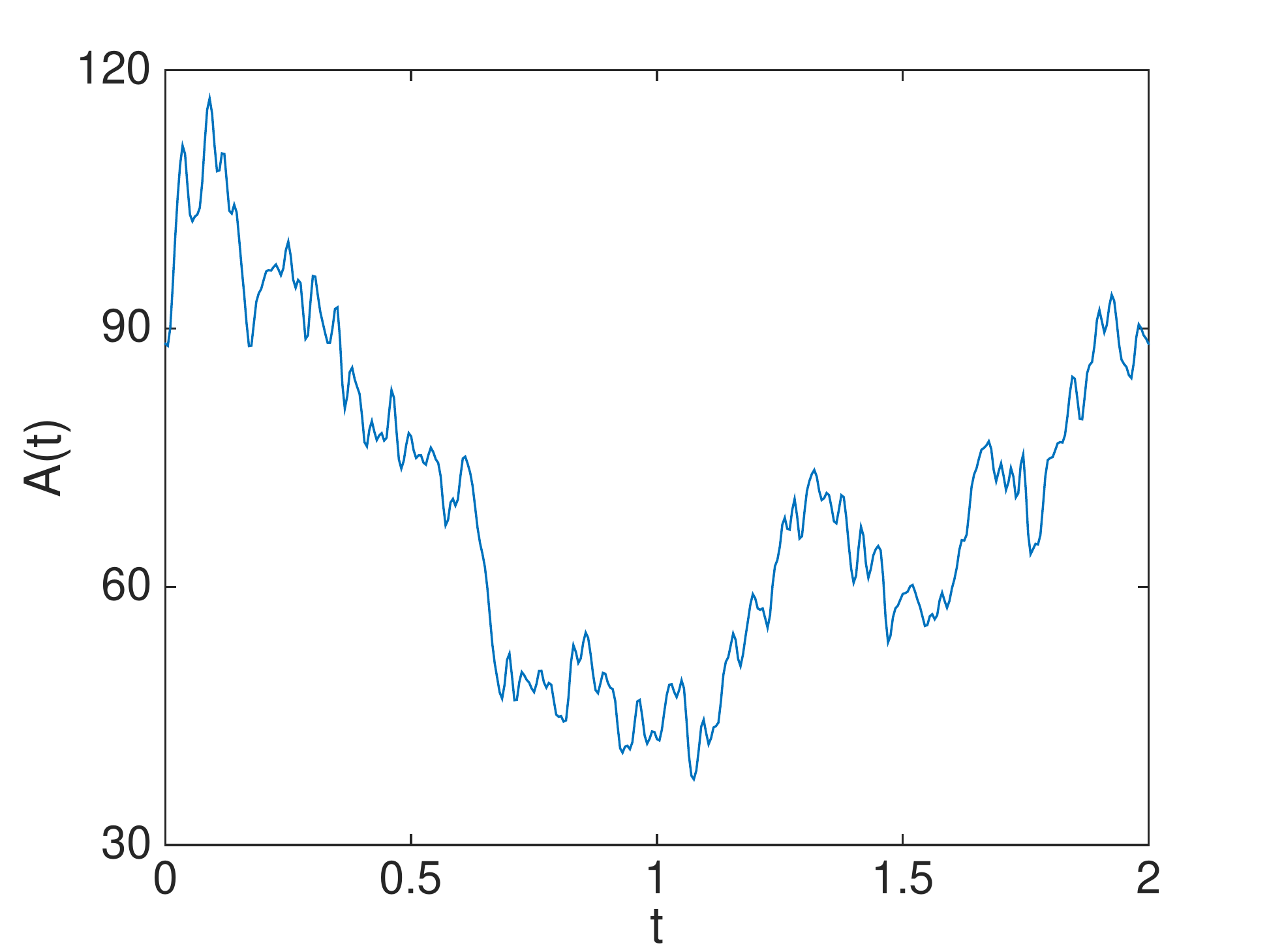}  
        \includegraphics[width=0.30 \textwidth,trim=40 20 35 20,clip]{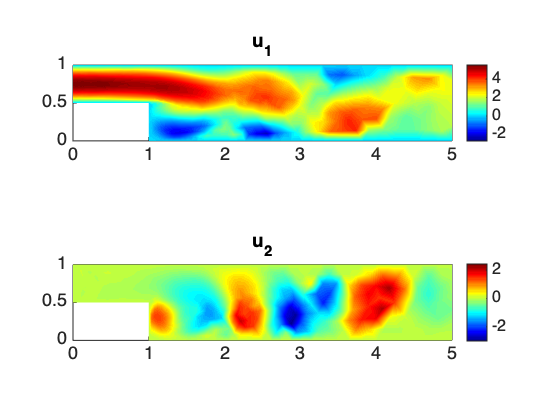}
        \quad
        \begin{rotate}{90}
                \hspace{30pt}   \scriptsize $ \| \mathbf{u} - \widetilde{ \mathbf{u}}^K \|^2_{L^2(D)}$
        \end{rotate}
        \includegraphics[width=0.30 \textwidth,trim=20 3 40 20,clip]{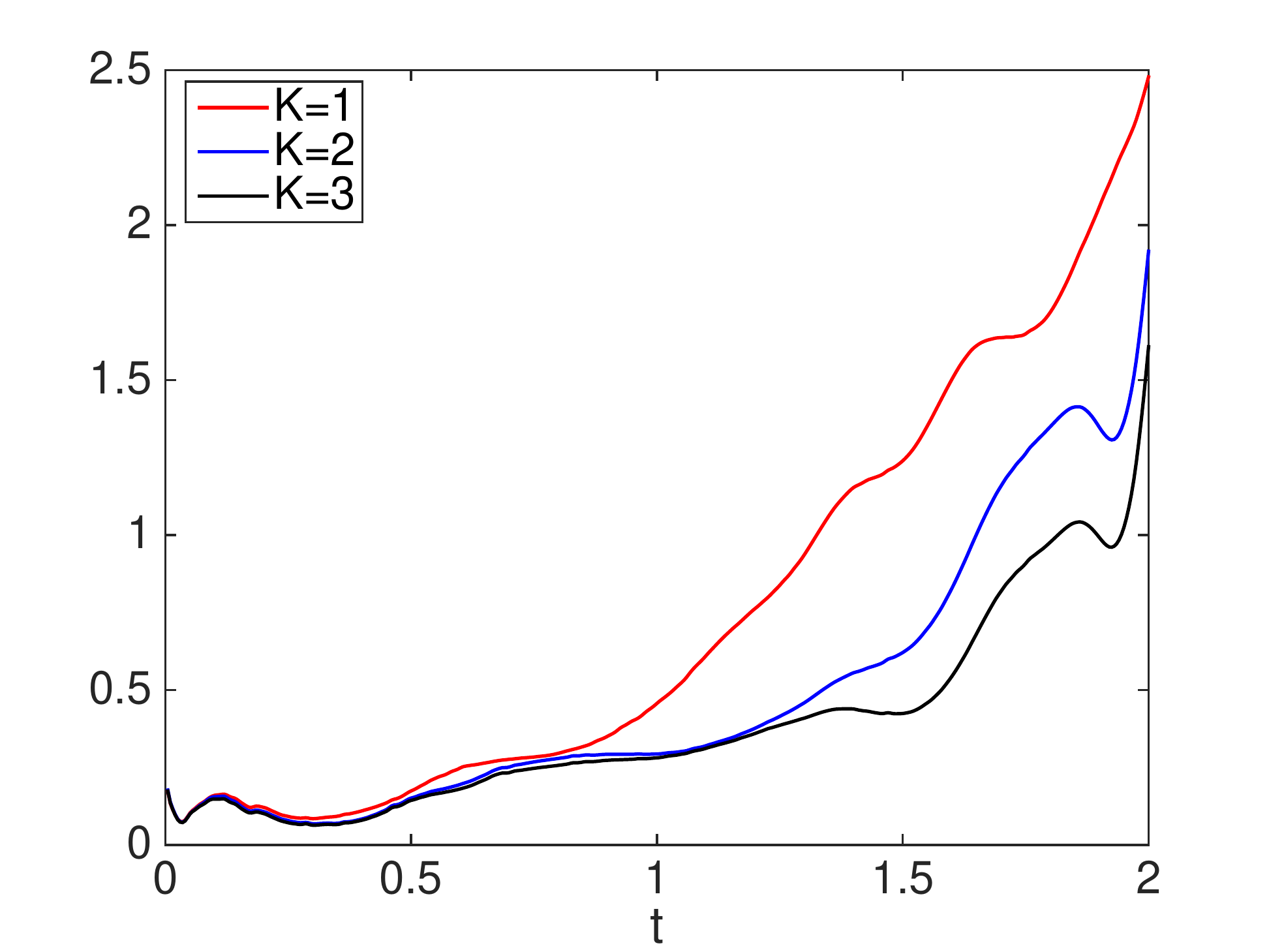}  \\
        \includegraphics[width=0.30 \textwidth,trim=10 2 40 20,clip]{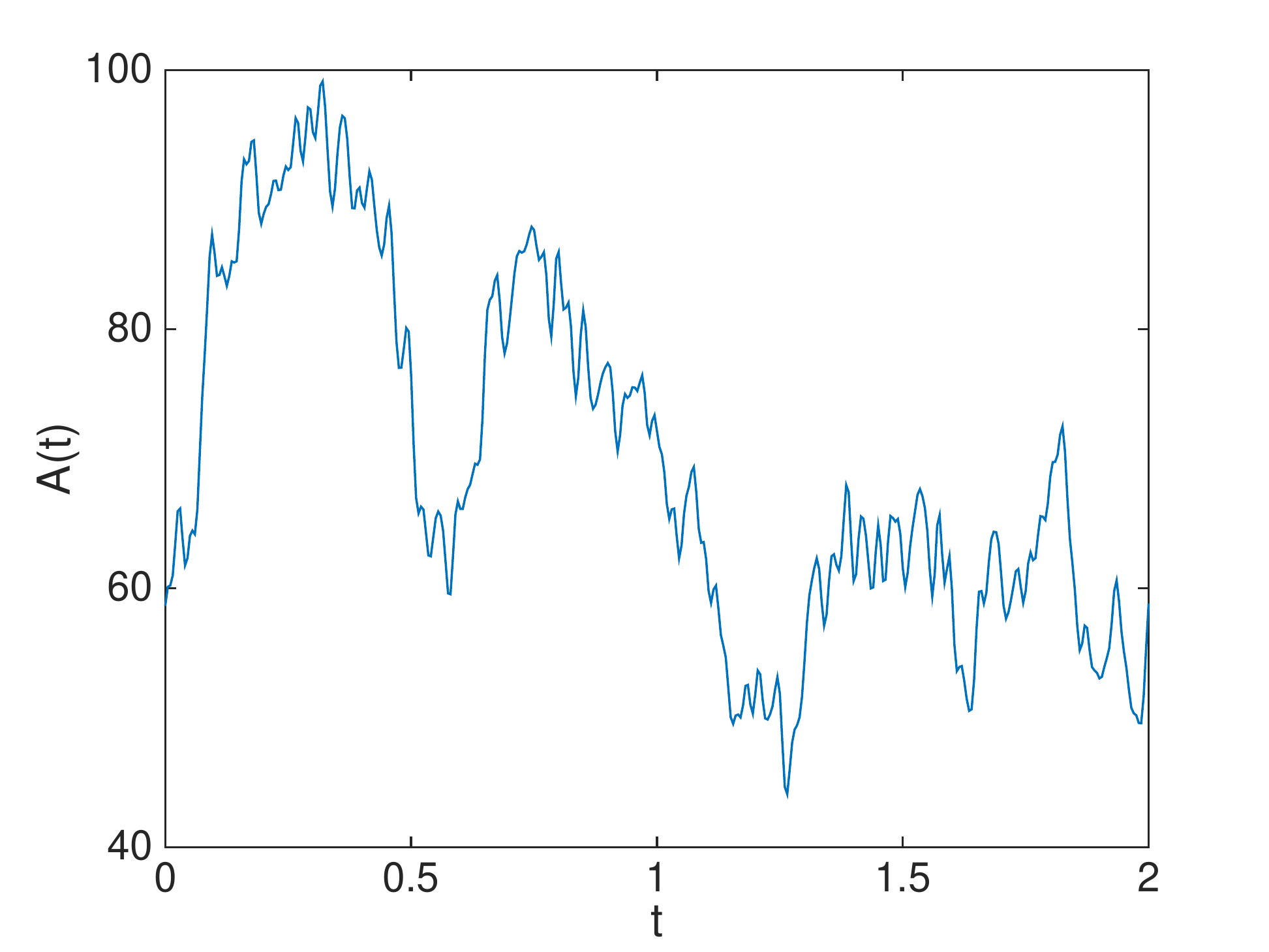}  
        \includegraphics[width=0.30 \textwidth,trim=40 20 40 20,clip]{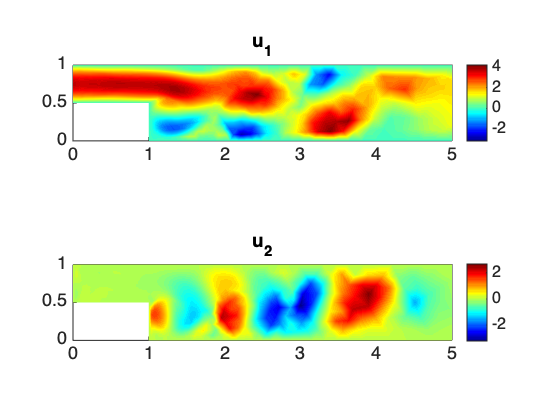}  
        \quad
        \begin{rotate}{90}
               \hspace{30pt}   \scriptsize $ \| \mathbf{u} - \widetilde{ \mathbf{u}}^K \|^2_{L^2(D)}$
        \end{rotate}
        \includegraphics[width=0.30 \textwidth,trim=20 3 40 20,clip]{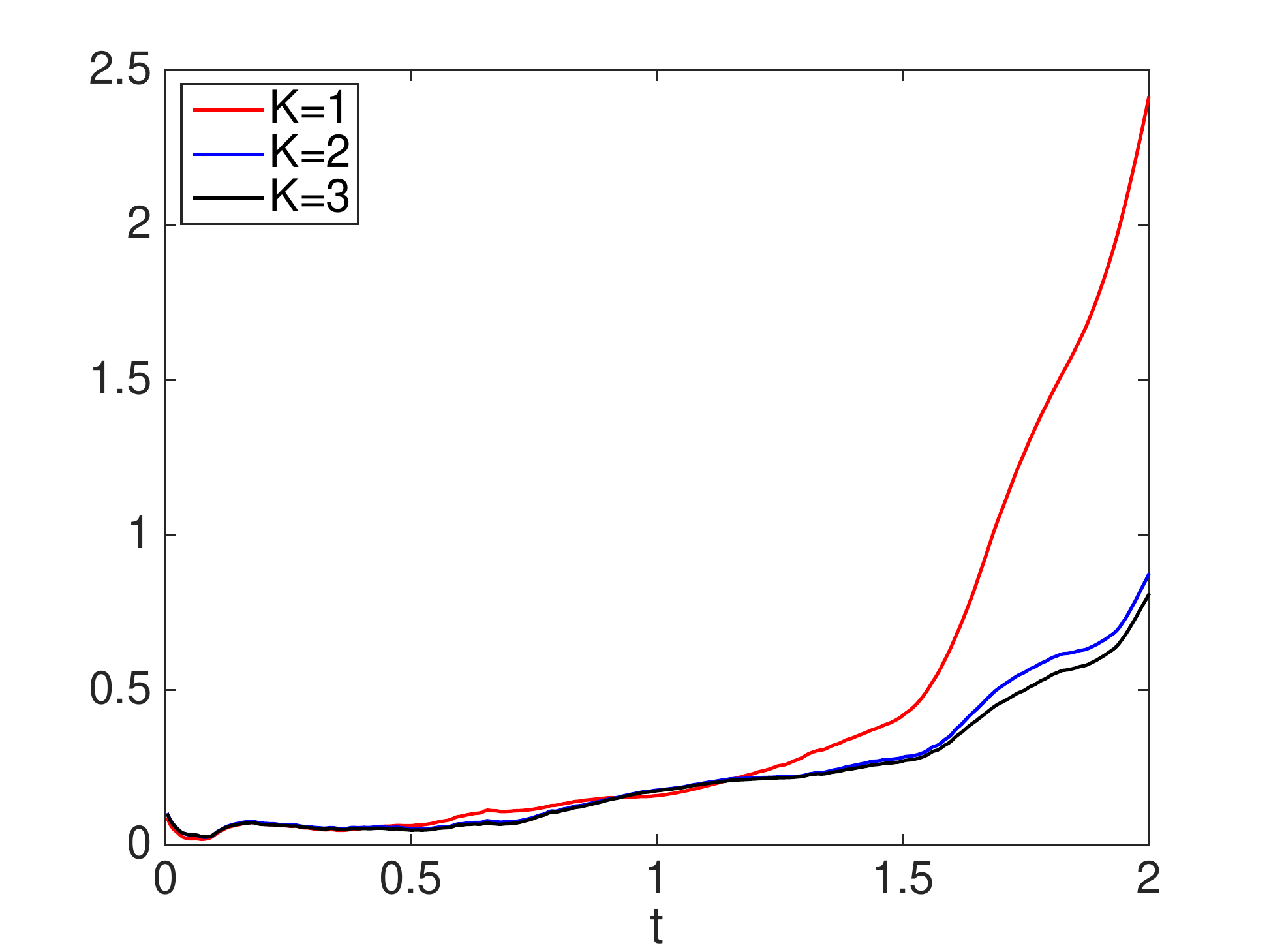}   \\
        \caption{Two realizations of the strength $A(t)$ in stochastic inlet velocity $u_{ \text{in}}$ (left), and their corresponding finite element solutions $\mathbf{u}=(u_1, u_2)^\top$ at time $T$ (middle), and the errors of CPOD approximate solutions (right) }    \label{F_train_simula1}
\end{figure}

\subsubsection{Simulation results of CPOD-NB based SROM}

Use 300 inputs $\{ \bm{\xi}_i\}_{i=1}^{300}$ associated with data set $\widehat{U}$ and the clustering results of modified t-gCVT method to train a naive Bayes pre-classifier. Here, we directly use the naive Bayes classification toolbox of MATLAB. For these 100 test data, use the pre-classifier to get their predicted labels, and use formula (\ref{true_label}) to get their true labels. The resulting confusion matrices are shown in Figure~\ref{F_confusion1}.
It can be observed that when $K=2$, all 53 samples with the true label of 1 are predicted correctly, while 20 of the 47 samples with the true label of 2 are predicted incorrectly. In other words, the predicted labels of 80\% of the test data are consistent with their true labels. Similarly,  70\% of the test samples are correctly predicted for $K=3$. As defined in (\ref{error_rate2}), the error rates of the naive Bayes pre-classifier are 9.22\% when $K=2$ and 20.10\% when $K=3$.
\begin{figure}[ht]
       \centering 
       \begin{rotate}{90}
               \hspace{35pt}   \footnotesize Predicted label
       \end{rotate}
       \includegraphics[width=0.30 \textwidth,trim=0 0 0 2,clip]{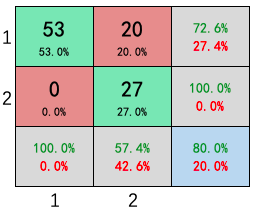}  
       \qquad
       \begin{rotate}{90}
             \hspace{35pt}   \footnotesize Predicted label
       \end{rotate}
       \includegraphics[width=0.30 \textwidth,trim=0 0 0 2, clip]{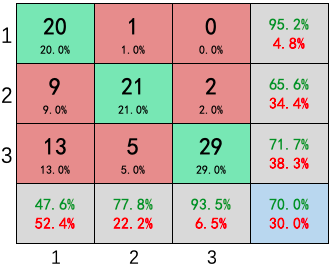} \\
       \footnotesize   True  label    \hspace{120pt}  True  label  \hspace{30pt}
       \caption{Confusion matrices of test data set with 100 samples for $K=2$ and $3$}  \label{F_confusion1}
\end{figure}        
        
Table~\ref{T_test_err1} lists the errors of the CPOD-NB based SROM estimated with the test data. The results on the left are associated with the true labels, while the results on the right are associated with predicted labels. Obviously, whether the true labels or the predicted labels are used, the accuracy of CPOD-NB based SROM is gradually improving with the increase of $K$, even though the misjudgment samples have an impact on the accuracy of our SROM. Figure~\ref{F_test_err1} shows the errors of 4 samples in the test data. It can be seen that the reduced-order solutions calculated by our true best-matched CPOD basis functions have better accuracy than the standard POD reduced-order solution, but the errors may be larger than that of the standard POD method in the case of misjudgment.        
        
\begin{table}[h]
\begin{center}
\begin{minipage}{\textwidth}
\caption{Error estimates of the CPOD-NB based SROM by using 100 test samples under the true labels (left) and predicted labels (right)} \label{T_test_err1}
\begin{tabular*}{\textwidth}{@{\extracolsep{\fill}}lcccccccc@{\extracolsep{\fill}}}
\toprule%
& \multicolumn{4}{@{}c@{}}{True labels} & \multicolumn{4}{@{}c@{}}{Predicted labels} \\\cmidrule{2-5}\cmidrule{6-9}%
$K$ & $\widetilde{\mathcal{E}}_K$ & $\widetilde{\mathcal{E}}_K^r$  & $\widetilde{\mathcal{V}}_K$ & $\widetilde{\mathcal{V}}_K^r$ & $\widetilde{\mathcal{E}}_K$ & $\widetilde{\mathcal{E}}_K^r$  & $\widetilde{\mathcal{V}}_K$ & $\widetilde{\mathcal{V}}_K^r$ \\
\midrule
1  &  0.6256  &  2.8137\%  &  0.4200  &  0.0871\%  &  0.6256  &  2.8137\%  &  0.4200  &  0.0871\%     \\
2  &  0.5062  &  2.2723\%  &  0.0738  &  0.0160\%  &  0.5477  &  2.4340\%  &  0.0925  &  0.0191\%    \\
3  &  0.4576  &  2.0594\%  &  0.0611  &  0.0133\%  &  0.5038  &  2.2353\%  &  0.0754  &  0.0156\%   \\ 
\midrule
\end{tabular*}
\end{minipage}
\end{center}
\end{table}  
        
\begin{figure}[ht]
       \centering 
       \begin{rotate}{90}
                 \hspace{15pt} \scriptsize $ \| \mathbf{u} - \widetilde{ \mathbf{u}}^K \|^2_{L^2(D)}$
       \end{rotate}
       \includegraphics[width=0.24 \textwidth,trim=30 20 40 20,clip]{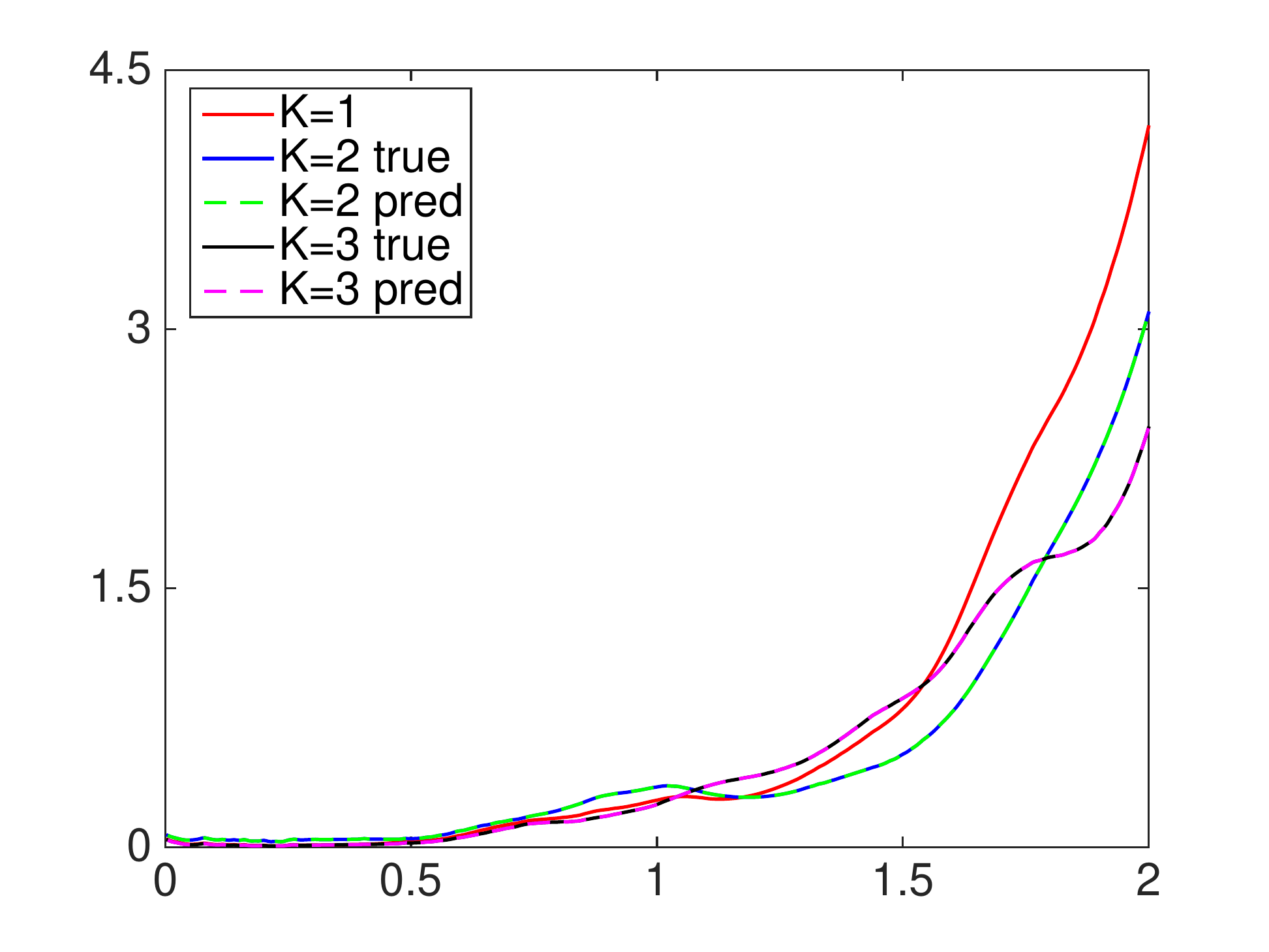}  
       \includegraphics[width=0.24 \textwidth,trim=30 20 40 20,clip]{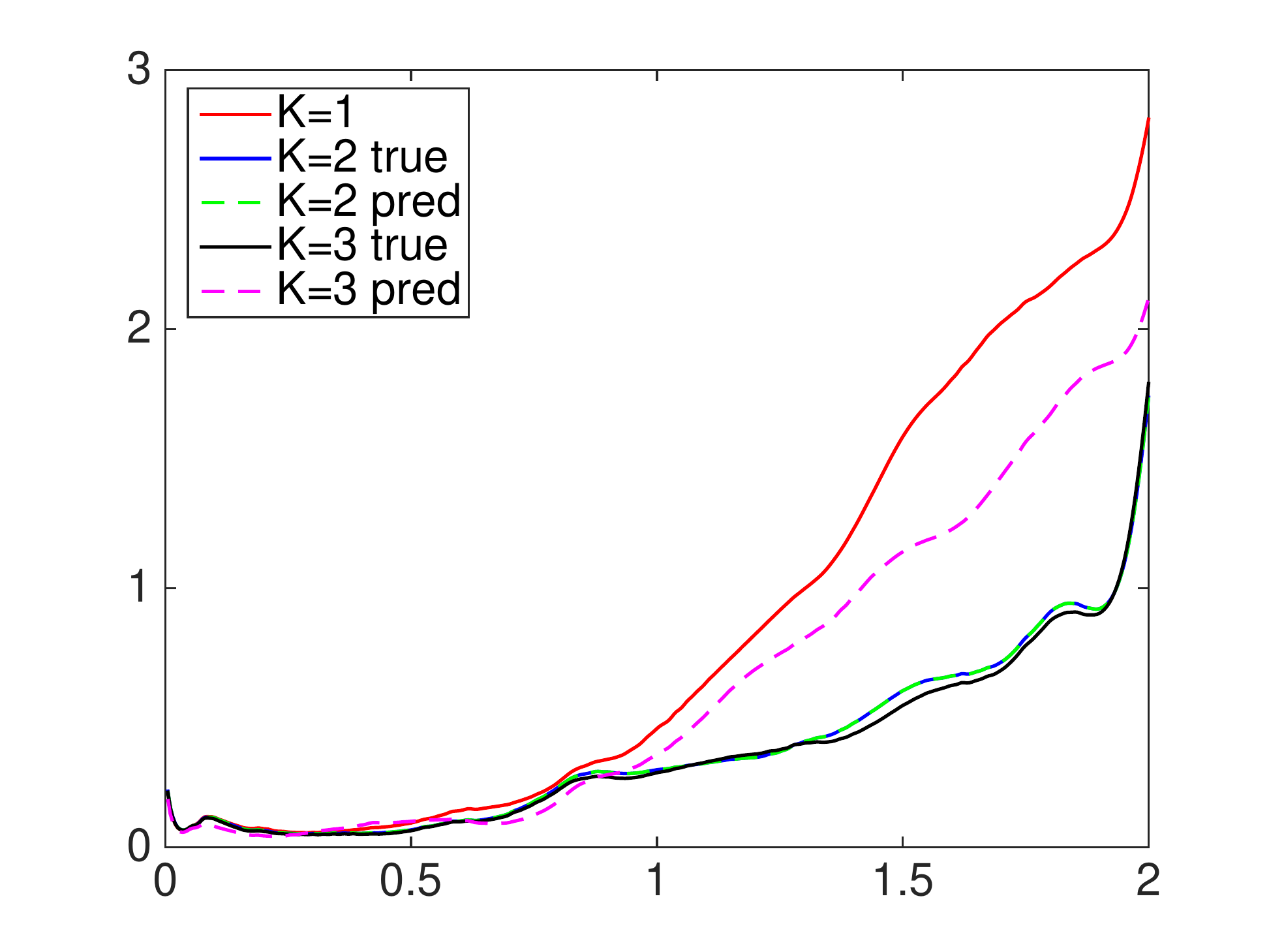} 
       \includegraphics[width=0.24 \textwidth,trim=30 20 40 20,clip]{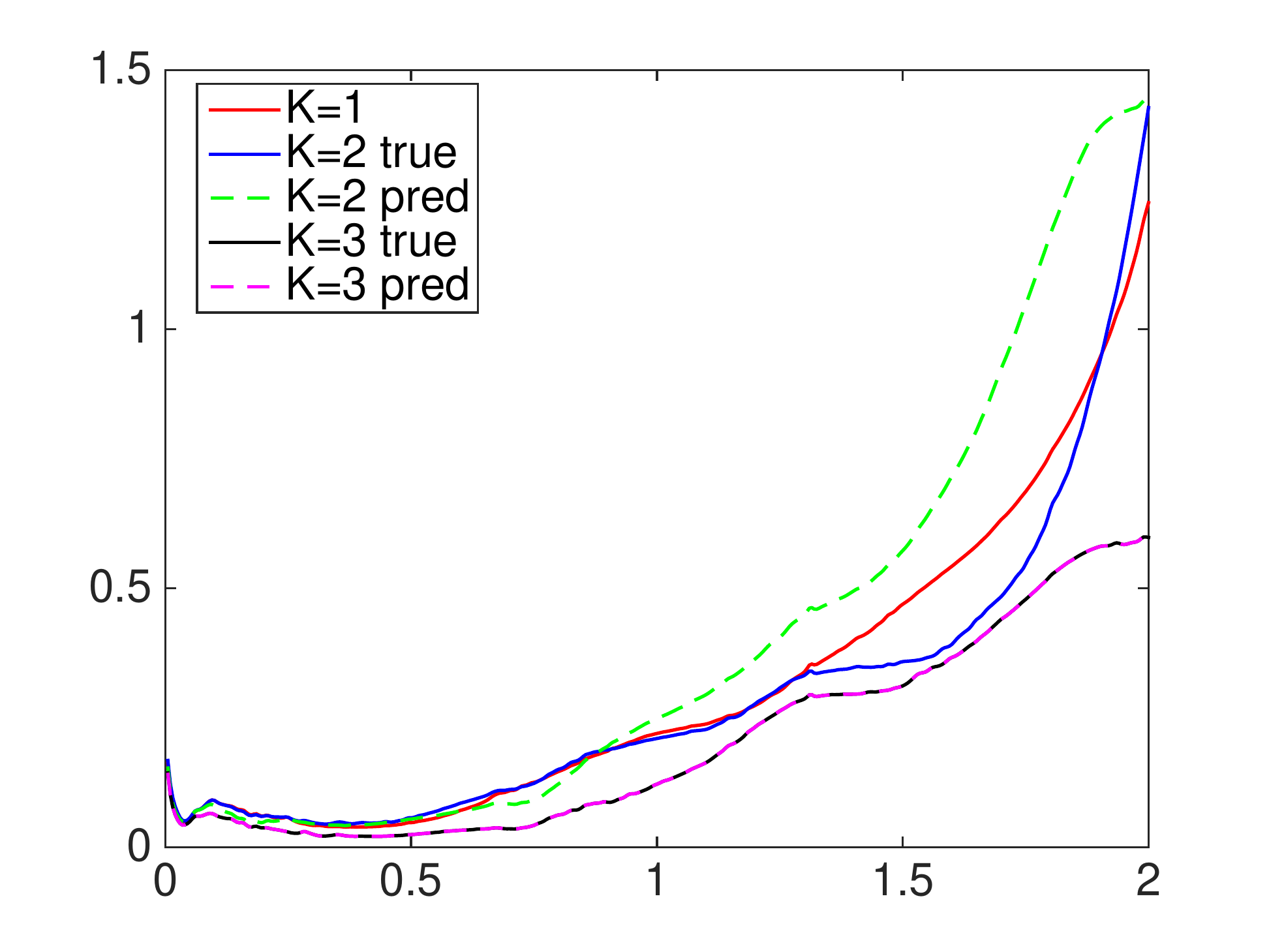}
       \includegraphics[width=0.24 \textwidth,trim=30 20 40 20,clip]{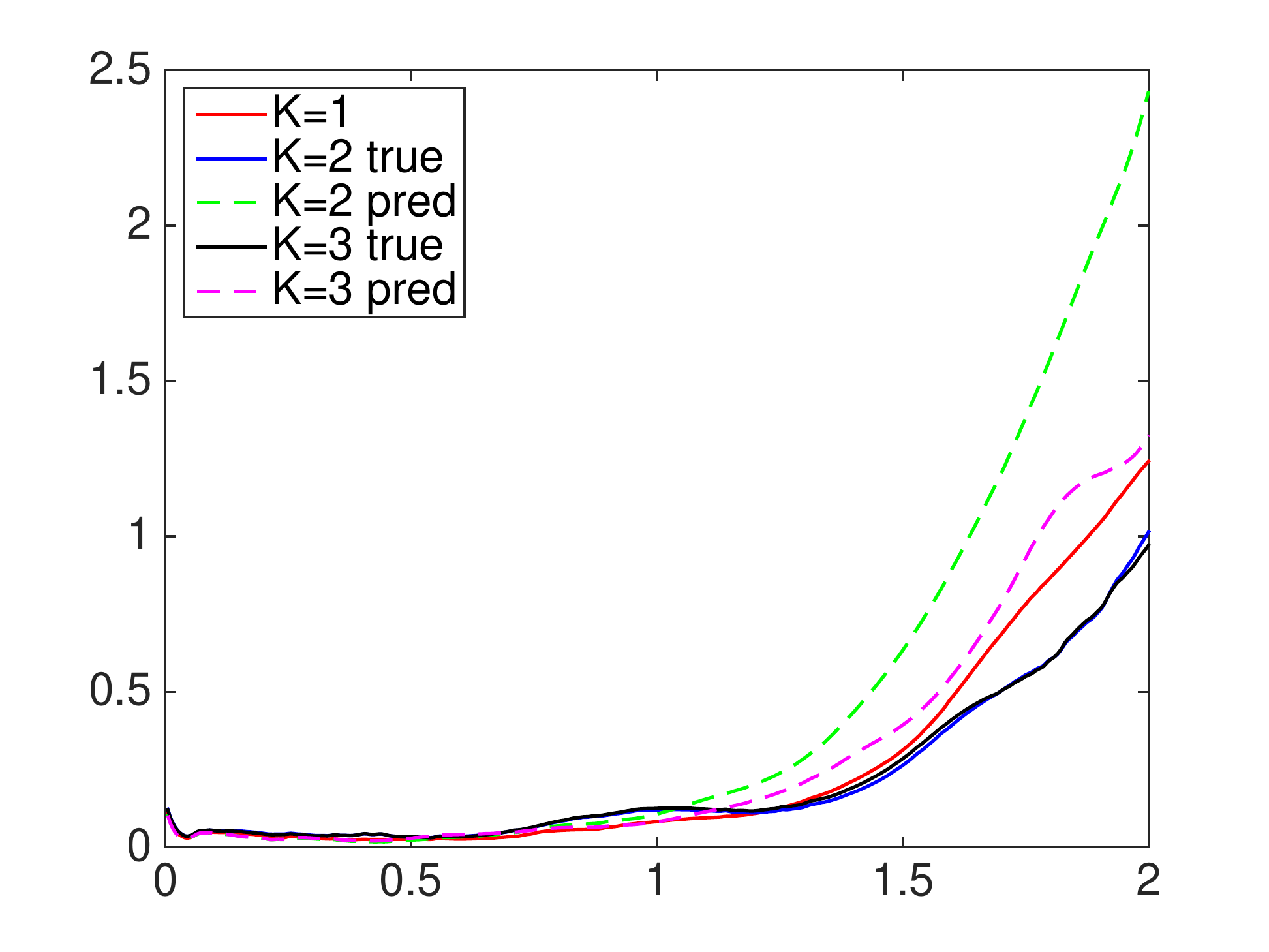}  \\ 
       \footnotesize  \hspace{1em}   $t$  \hspace{12em}  $t$   \hspace{12em}  $t$  \hspace{12em}  $t$  \hspace{10pt} \\
       \caption{The  errors of the CPOD-NB approximate solutions of four samples in the test data}    \label{F_test_err1}
\end{figure}

\subsection{Hat-type strength with white noise}

In this numerical experiment, the strength $A$ takes the following form
\begin{equation}
        A (t; \omega) = \sigma \frac{dW}{dt} + 60 \left\{
        \begin{array}{ll}
               1+at  &  t \in [0,1],  \\
               1+a(2-t)  &  t \in [1,2],
       \end{array}
       \right.
\end{equation}
where the height parameter $a \in \{ 0.8, 0.9, 1.0, 1.1, 1.2 \}$, and the amplification factor of white noise $\sigma=1.5$. 
The white noise $\frac{dW}{dt}$ is approximated by the piecewise constant
\begin{equation}   \label{white_noise}
        \frac{dW^m}{dt}=\frac{1}{\sqrt{ \delta t }} \sum_{i=0}^{m-1} \chi_i (t) \eta_i (\omega),
\end{equation}
where the components of $\bm{\eta} (\omega)=[\eta_0(\omega), \ldots, \eta_{m-1}(\omega) ]^ \top $ are i.i.d. random variables and satisfy the standard normal distribution $\mathcal{N} (0,1)$, and the characteristic function $\chi_i (t)$ is defined by
\begin{equation}    \label{chi}
        \chi_i (t) = \left\{
        \begin{array}{ll}
                1  \qquad t \in [t_i, t_{i+1} ),  \\
                0  \qquad \text{otherwise}. 
        \end{array}
        \right.
\end{equation}
Figure~\ref{F_hat_A0} shows the strengths $A$ corresponding to different coefficients $a$ when not affected by white noise.
\begin{figure}[H]
        \centering 
        \includegraphics[width=0.30 \textwidth,trim=5 0 48 20,clip]{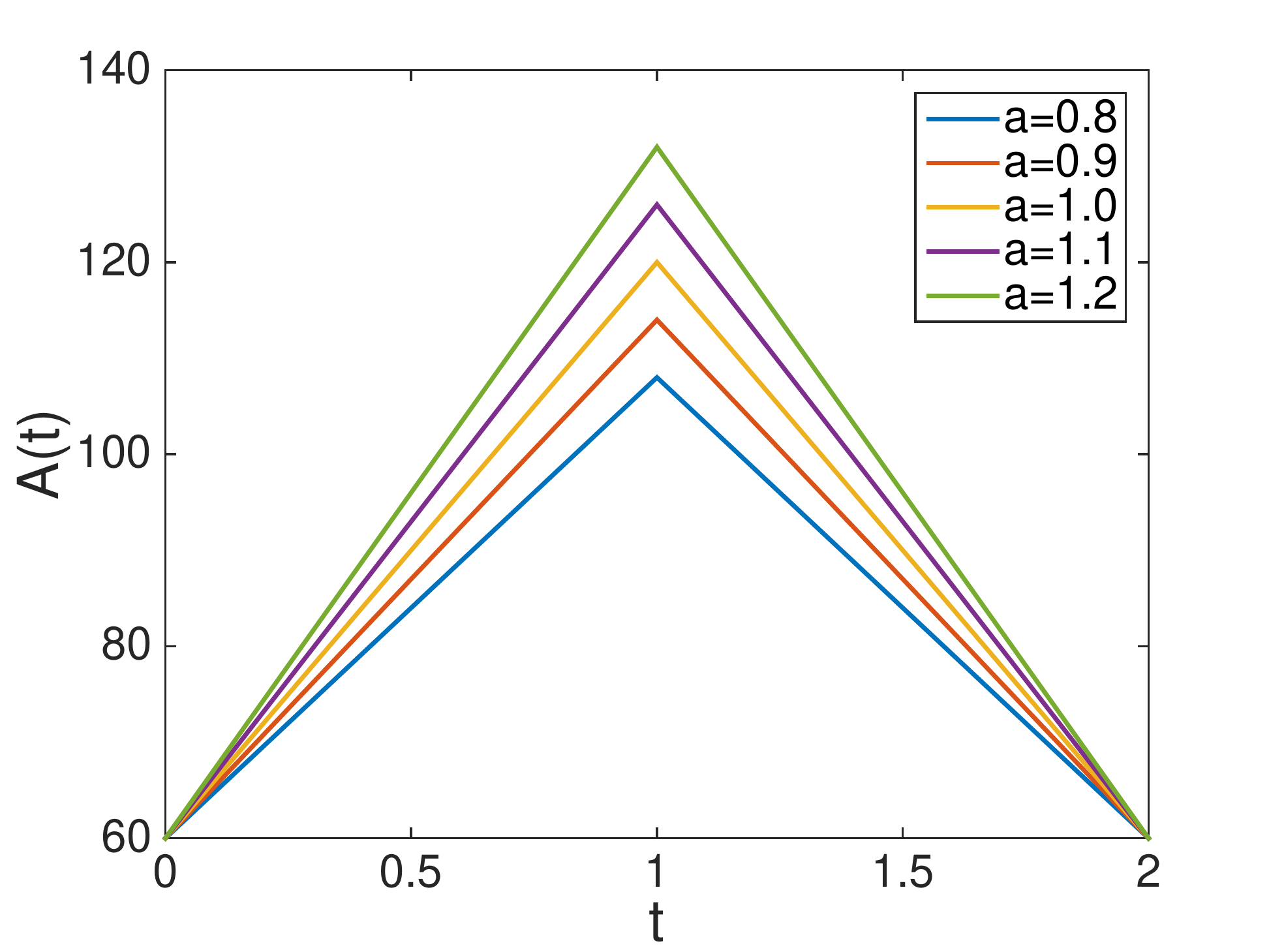} 
        \caption{Strengths $A$ corresponding to different coefficients $a$ ($\sigma=0$)}    \label{F_hat_A0}
\end{figure}
Here, we take $a=0.8, 0.9, 1.0, 1.1$ and 1.2 to generate 80 samples of velocity field respectively, and use these samples to form a data set $\widehat{U}$ for constructing the CPOD basis functions and training the naive Bayes pre-classifier.
In addition, use these coefficients $a$ to generate 20 samples respectively to form a test set for estimating the error of CPOD-NB based SROM.
        
\subsubsection{Generating CPOD basis functions}

Figure~\ref{F_population2} shows the clustering results of these 400 training data by using the modified t-gCVT method. The dimensions and cumulative energy ratios used in this experiment are listed in Table~\ref{T_dim2}. Although the energy ratios of the second class with $K=2$ and the second and third classes with $K=3$ are all slightly smaller than that with $K=1$, the energy ratios of the first class with $K=2$ and $K=3$ are much larger than that of the standard POD basis functions. Figure~\ref{F_PODbasis2} shows the contours of the first four CPOD basis functions in each class for different $K$.        
        
\begin{figure}[ht]
       \centering 
       \includegraphics[width=0.31 \textwidth,trim=35 20 45 20,clip]{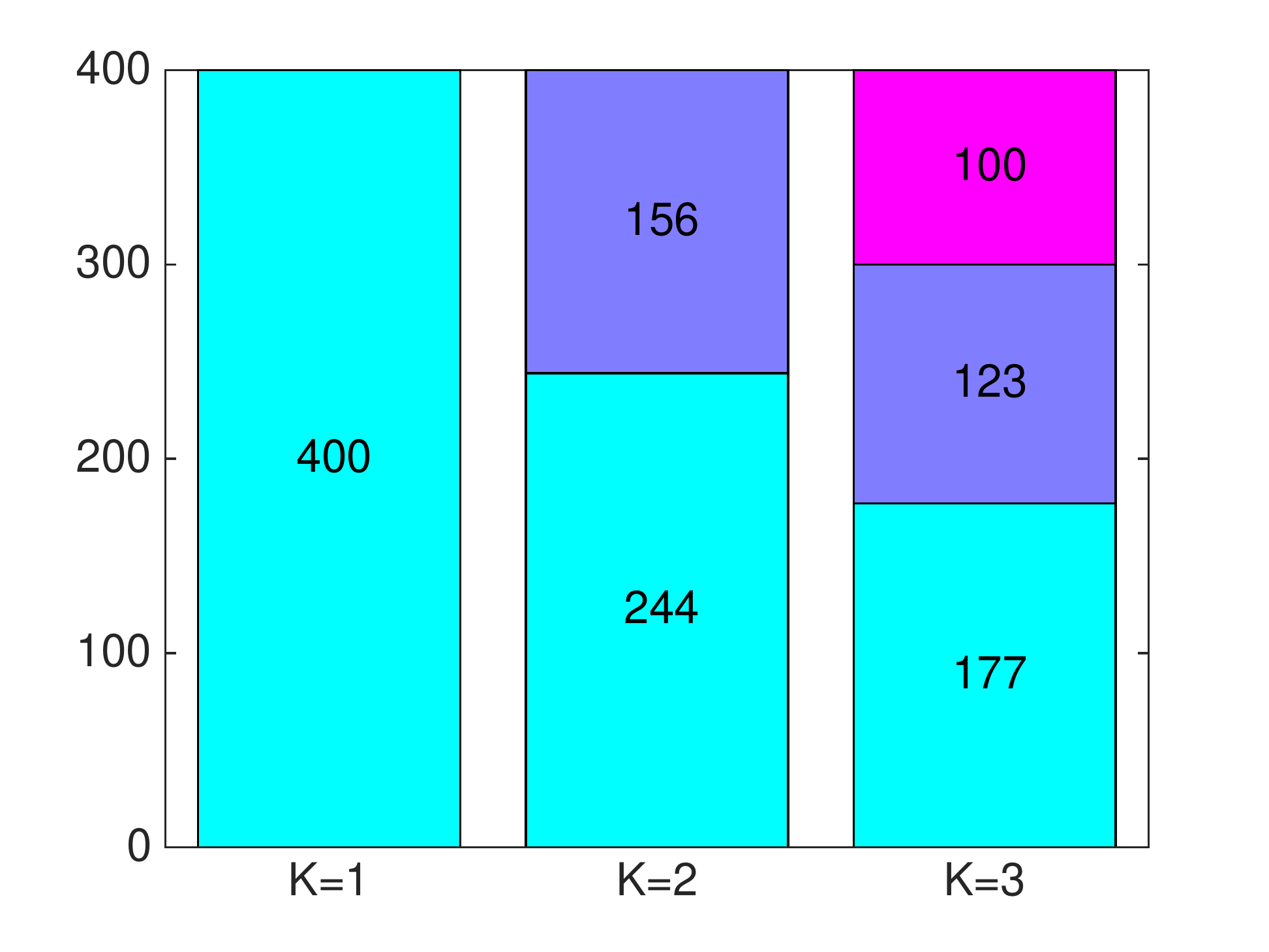}  
       \includegraphics[width=0.31 \textwidth,trim=20 20 45 0,clip]{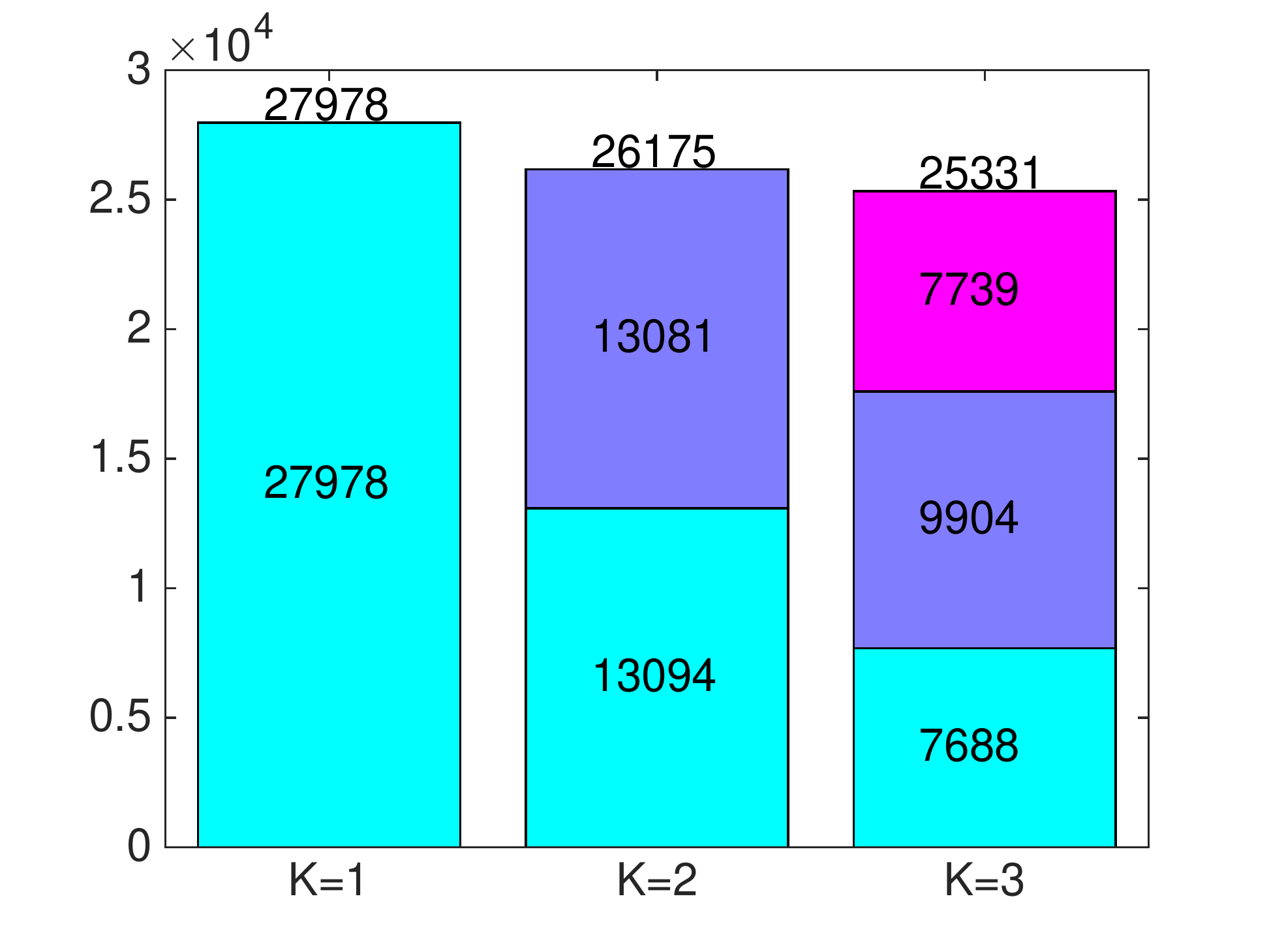}
       \quad
       \begin{rotate}{90}
              \hspace{40pt}   \footnotesize $ \log \lambda $
      \end{rotate}
      \includegraphics[width=0.31 \textwidth,trim=30 20 45 20,clip]{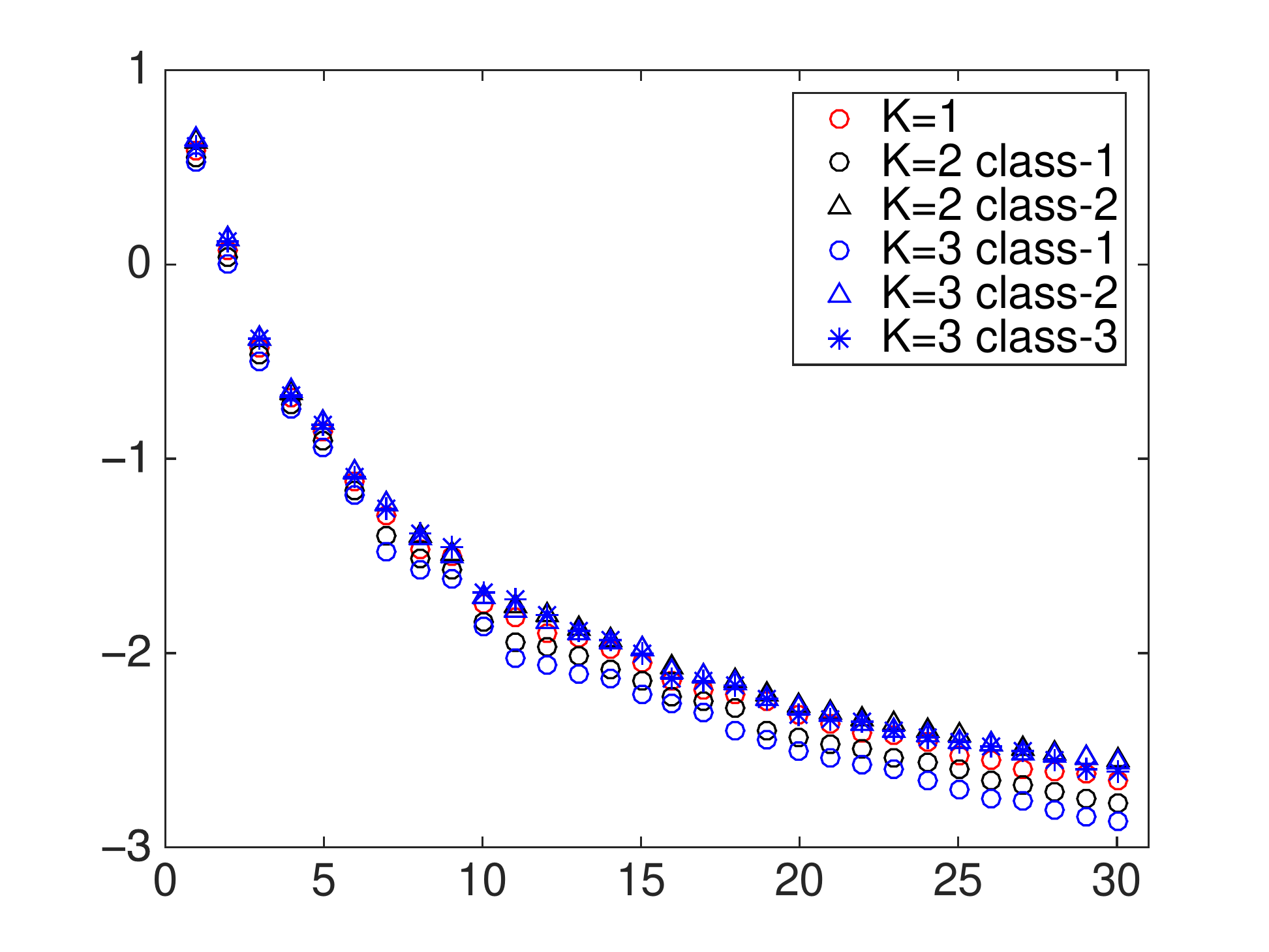}
      \caption{Population $n_k$ (left), energy $\widetilde{\mathcal{E}}^{ \text{t-gCVT} }$ (middle) of data $\widehat{U}$, and the logarithm of eigenvalues (right) corresponding to the first 30 CPOD basis functions in each class for $K=1,2 $ and $3$}  \label{F_population2}
\end{figure}
      
\begin{table}[ht]
       \centering
       \small
       \begin{spacing}{1}
       \caption{The dimension $d_k$ and cumulative energy ratio $\nu_k$ of CPOD basis functions in each class for $K=1, 2$ and $3$}   \label{T_dim2}
       \begin{tabular}{@{}ccccccc@{}}
               \hline
               &  $ K=1 $
               &  \multicolumn{2}{c}{ $ K=2 $ }
               &  \multicolumn{3}{c}{ $ K=3 $ }    \\
               \hline
               class & - & 1 & 2 & 1 & 2 & 3  \\
               $d_k$ & 11 & 11 & 11 & 11 & 11 & 11 \\
               $\nu_k$ &  0.9714  &  0.9761  &  0.9698  &  0.9793  &  0.9710  &  0.9709   \\
               \hline
      \end{tabular}
      \end{spacing}
\end{table}        
        
\begin{figure}[ht]
      \centering 
      \includegraphics[width=0.30 \textwidth,trim=50 30 50 0,clip]{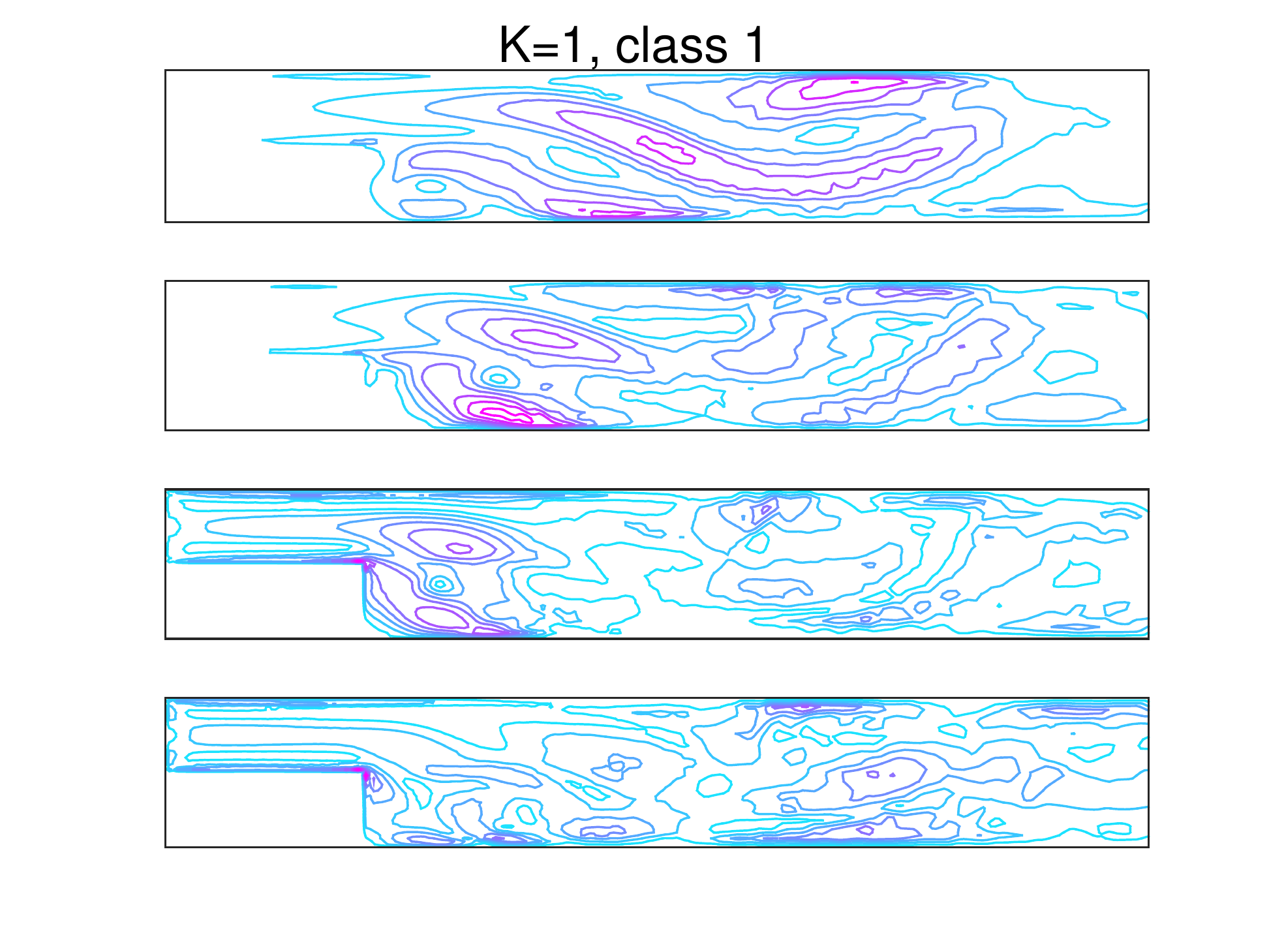}  
      \includegraphics[width=0.30 \textwidth,trim=50 30 50 0,clip]{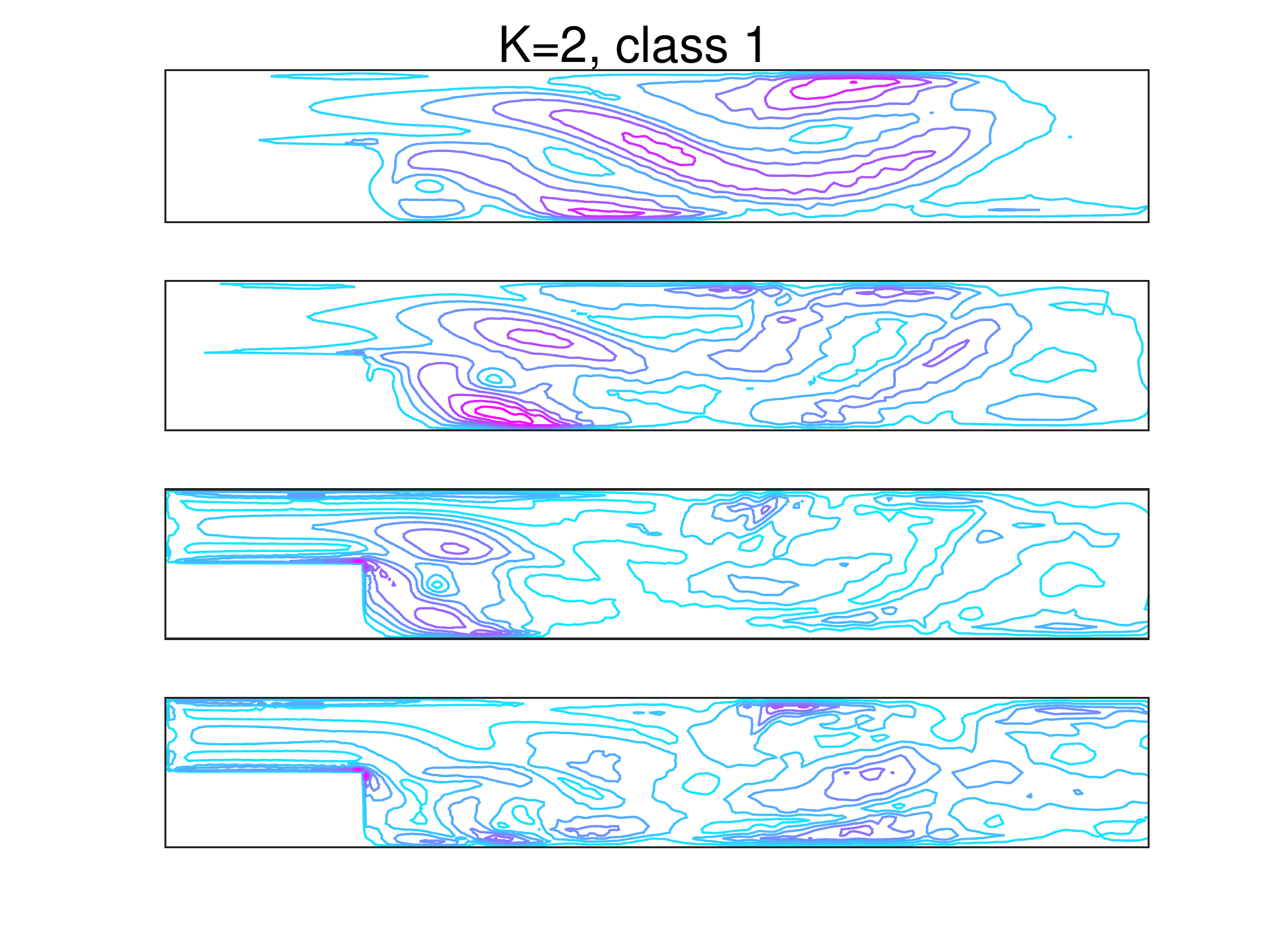} 
      \includegraphics[width=0.30 \textwidth,trim=50 30 50 0,clip]{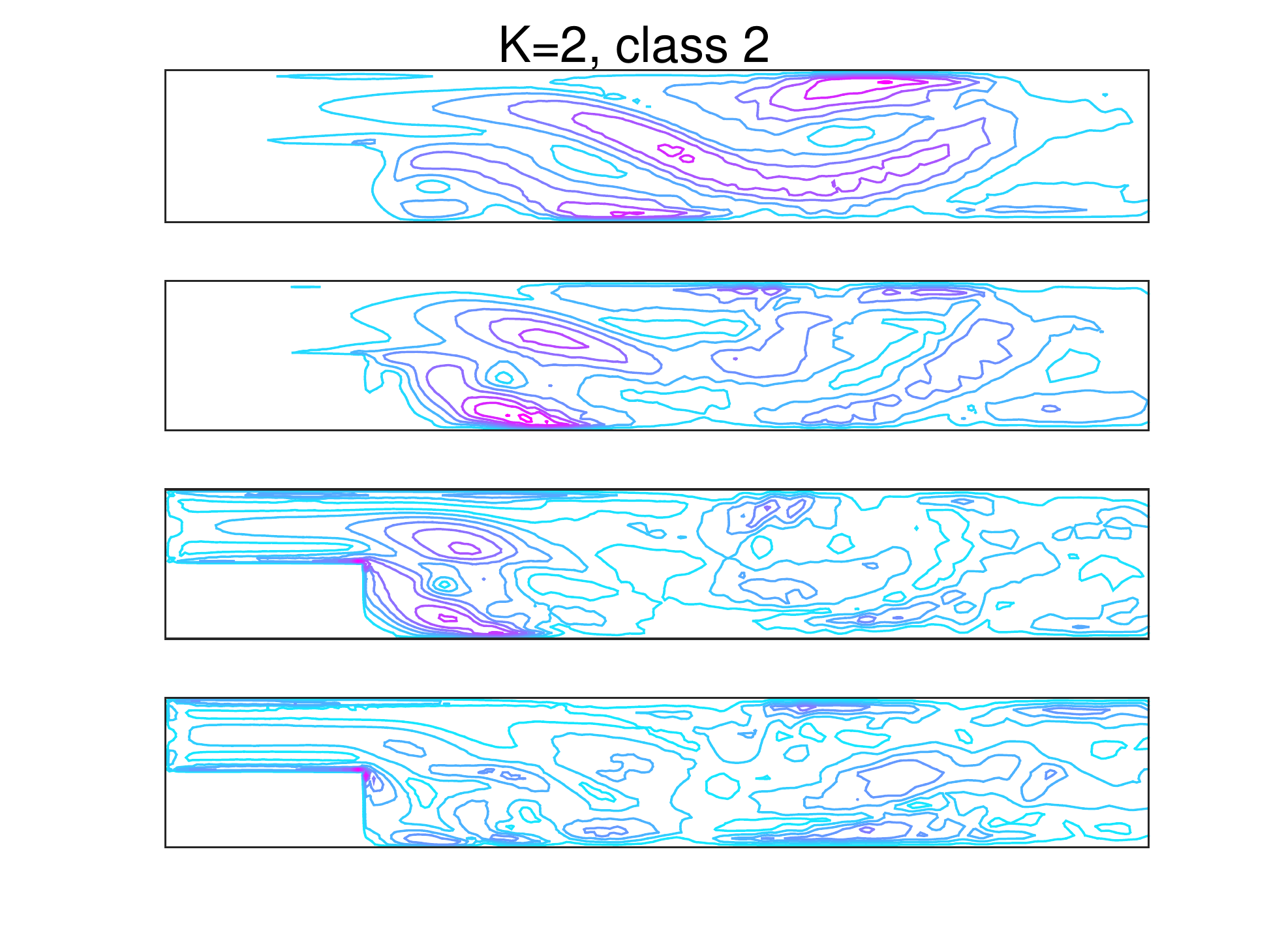} \\
      \includegraphics[width=0.30 \textwidth,trim=50 30 50 0,clip]{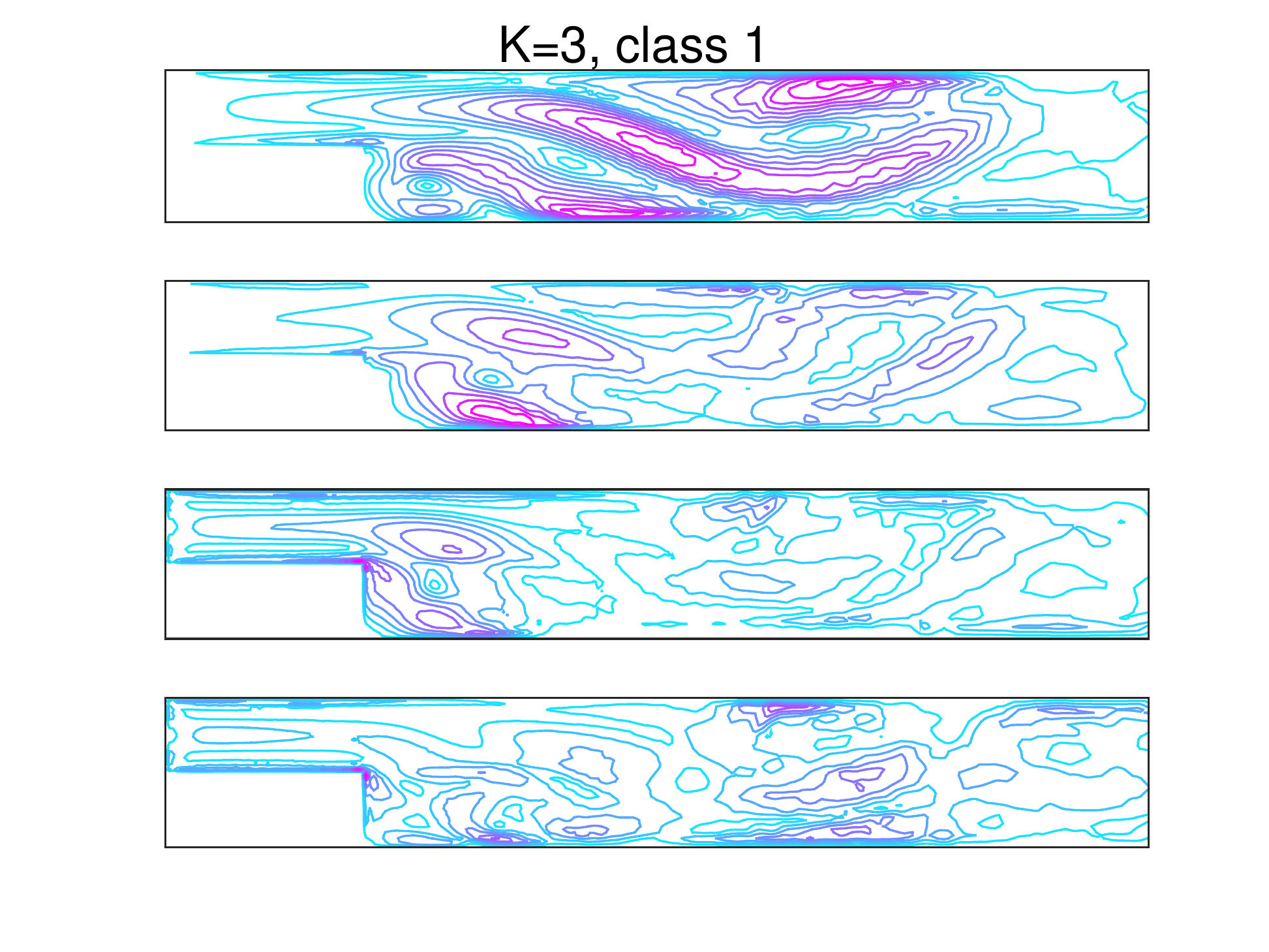}  
      \includegraphics[width=0.30 \textwidth,trim=50 30 50 0,clip]{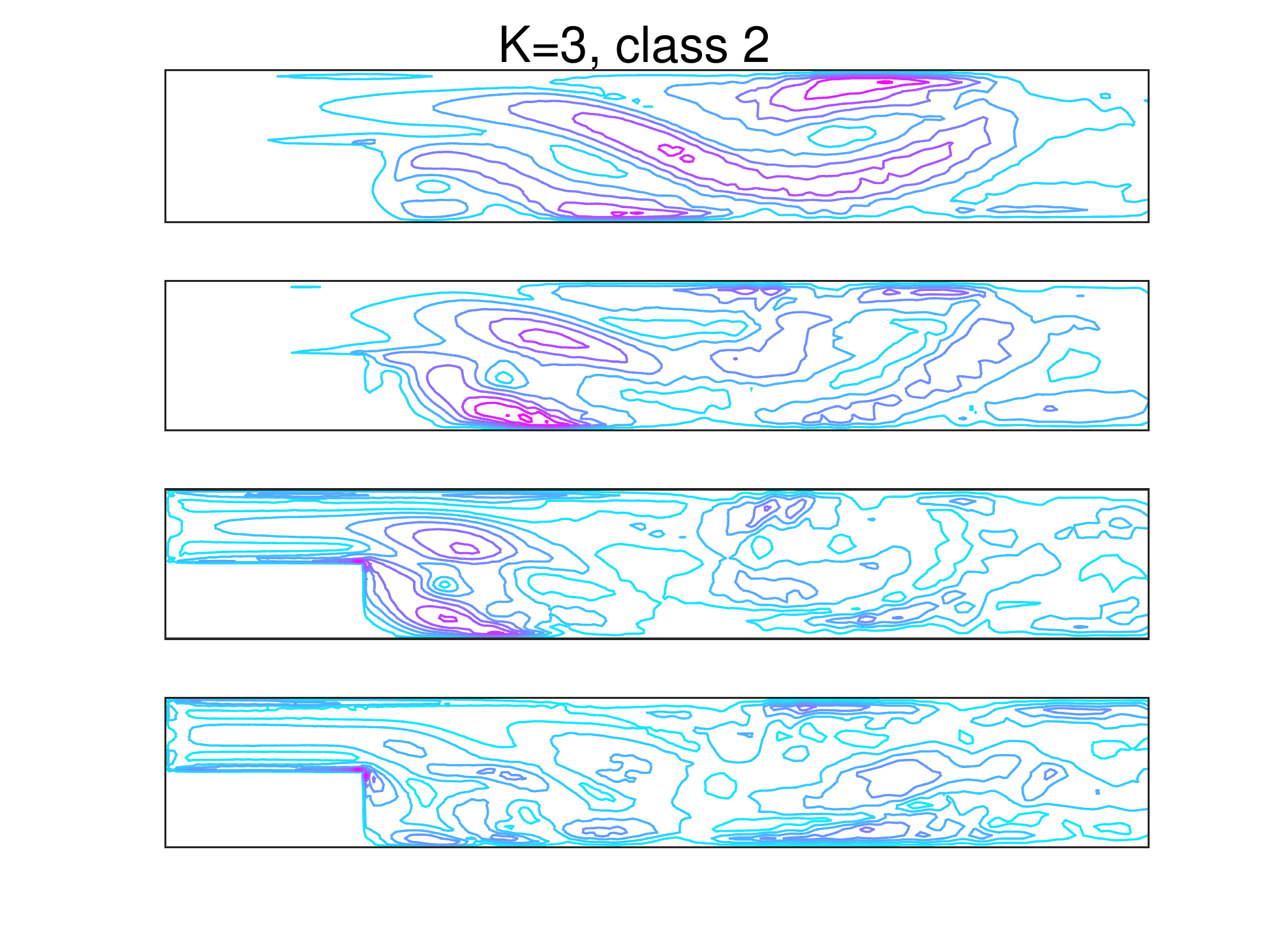} 
      \includegraphics[width=0.30 \textwidth,trim=50 30 50 0,clip]{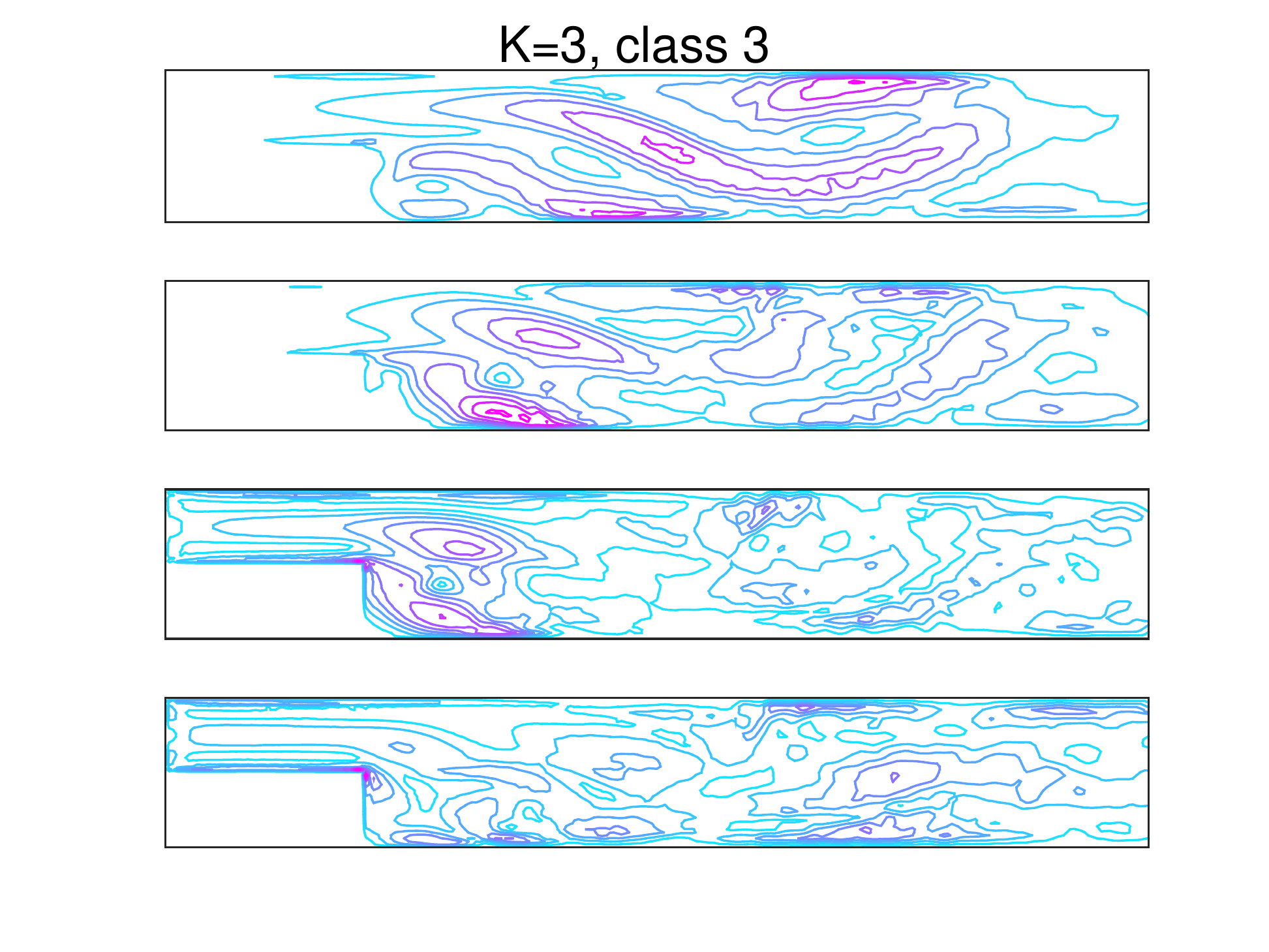} 
      \caption{Contours of the first four CPOD basis functions in each class for $K=1,2$ and $3$}     \label{F_PODbasis2}
\end{figure}        
        
Table~\ref{T_train_err2} gives the estimated error of the CPOD-based SROM by using 400 labelled training data. Obviously, from the perspective of expectation, the accuracy of our SROM increases with the increase of $K$.  The variance of absolute error is also increasing, but only slightly in terms of the relative error.  Figure~\ref{F_train_simula2} shows two samples in the training data and their errors of CPOD approximate solutions.        
        
\begin{table}[ht]
      \centering
      \small
      \begin{spacing}{1.3}
      \caption{Error estimates of CPOD-based SROM by using 400 labelled training data $\widehat{U}$ }
      \label{T_train_err2}
      \begin{tabular}{@{}ccccc@{}}
            \hline
            $K$ & $\widetilde{\mathcal{E}}_K $ & $\widetilde{\mathcal{E}}_K^r $ & $\widetilde{\mathcal{V}}_K $ & $\widetilde{\mathcal{V}}_K^r $   \\
            \hline
            1 &  0.6242  &  1.9596\%  &  0.0454  &  0.0406\%     \\ 
            2 &  0.5915  &  1.8572\%  &  0.0515  &  0.0499\%     \\
            3 &  0.5456  &  1.7731\%  &  0.0519  &  0.0515\%     \\
            \hline 
     \end{tabular}
     \end{spacing}
\end{table}
      
\begin{figure}[ht]
      \centering 
      \includegraphics[width=0.29 \textwidth,trim=5 2 40 20,clip]{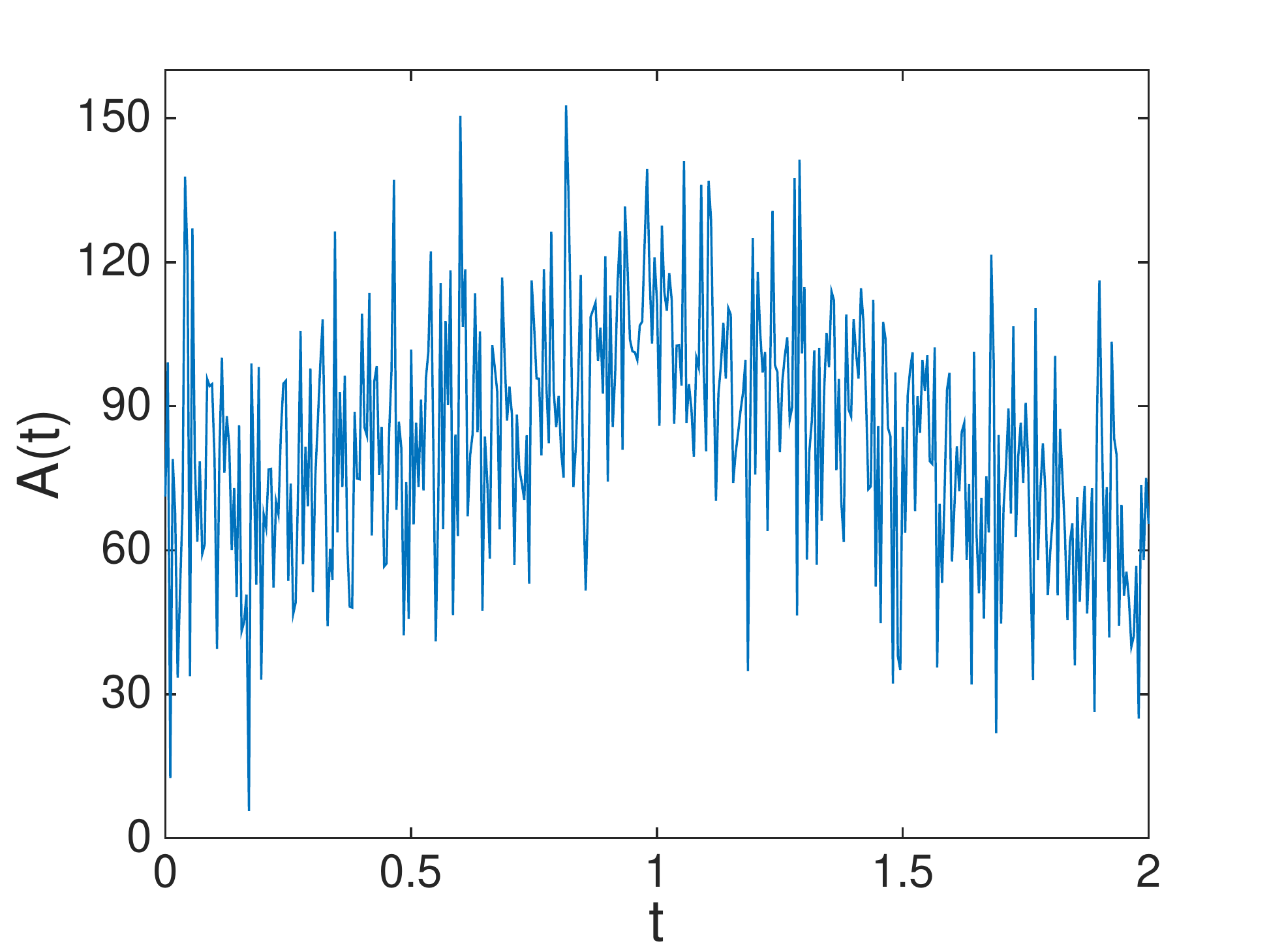}  
      \includegraphics[width=0.30 \textwidth,trim=40 20 40 20,clip]{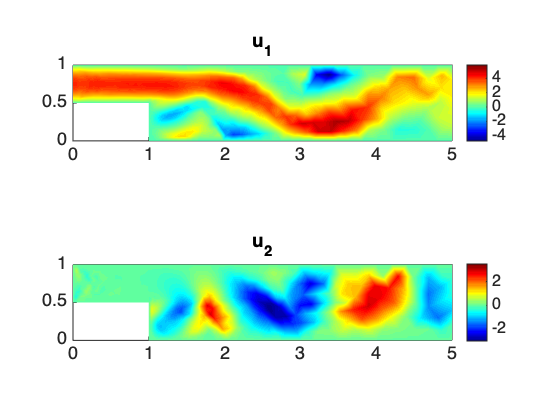}
      \quad
      \begin{rotate}{90}
             \hspace{25pt}   \scriptsize $ \| \mathbf{u} - \widetilde{ \mathbf{u}}^K \|^2_{L^2(D)}$
      \end{rotate}
      \includegraphics[width=0.30 \textwidth,trim=20 3 40 20,clip]{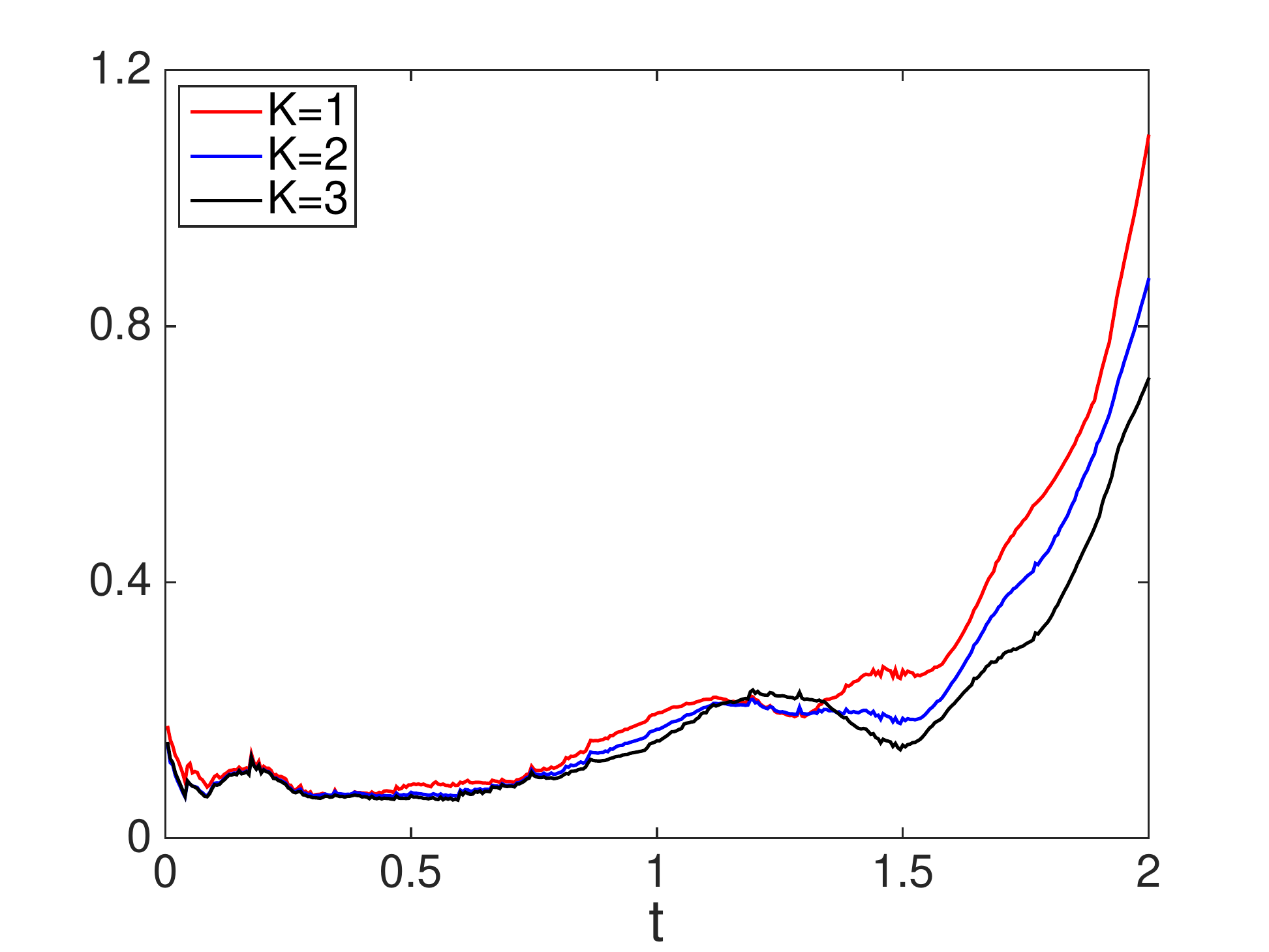}  \\
      \includegraphics[width=0.29 \textwidth,trim=5 2 40 20,clip]{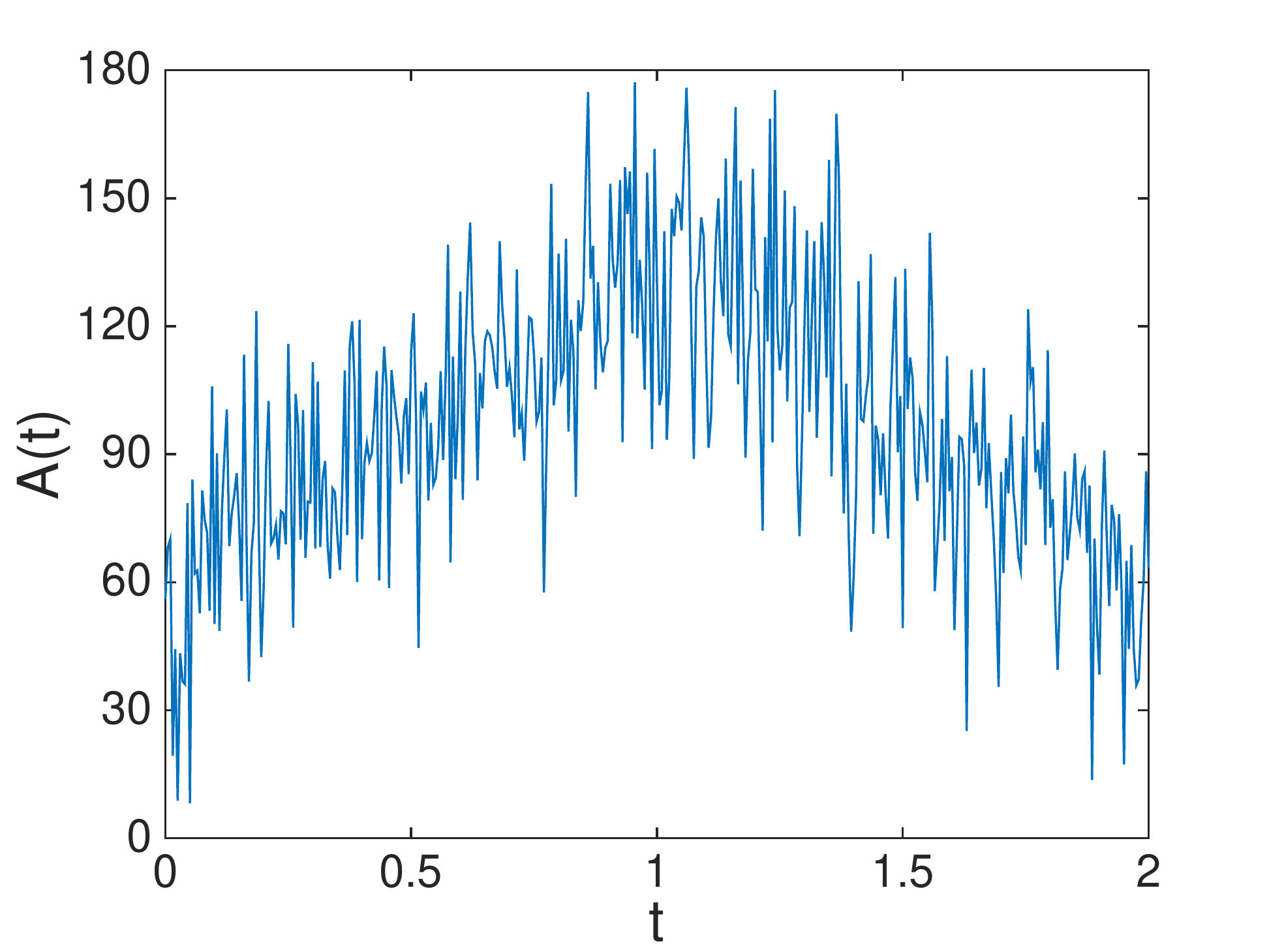} 
      \includegraphics[width=0.30 \textwidth,trim=40 20 40 20,clip]{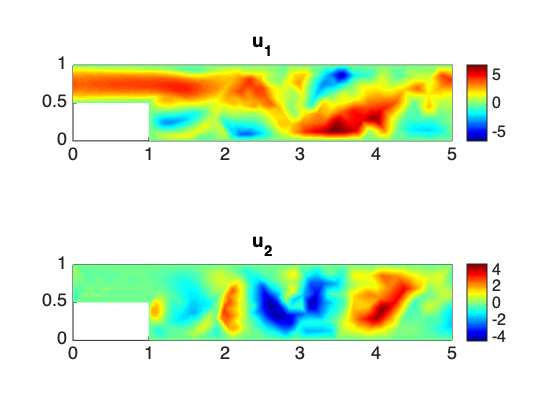}  
      \quad
      \begin{rotate}{90}
            \hspace{25pt}   \scriptsize $ \| \mathbf{u} - \widetilde{ \mathbf{u}}^K \|^2_{L^2(D)}$
      \end{rotate}
      \includegraphics[width=0.30 \textwidth,trim=20 3 40 20,clip]{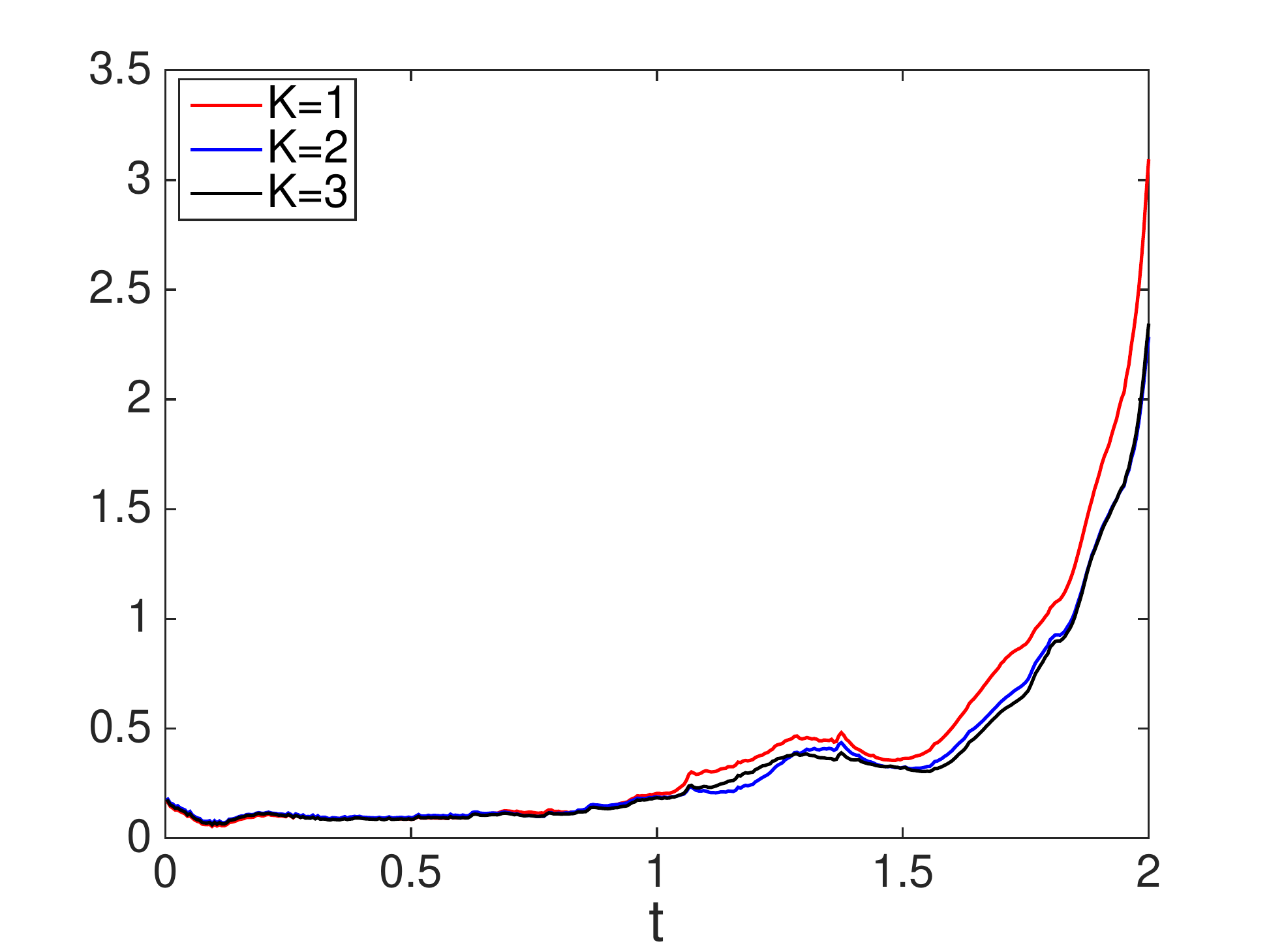}   \\
      \caption{Two realizations of the strength $A(t)$ in stochastic inlet velocity $u_{ \text{in}}$ (left), and their corresponding finite element solutions $\mathbf{u}=(u_1, u_2)^\top$ at time $T$ (middle), and the errors of CPOD approximate solutions (right) }   \label{F_train_simula2}
\end{figure}

\subsubsection{Simulation results of CPOD-NB based SROM}

For these 100 test data, the confusion matrices are shown in Figure~\ref{F_confusion2}. The corresponding error rates of naive Bayes pre-classifier are 15.80\% when $K=2$ and 31.99\% when $K=3$. Although the error rate of the pre-classifier is higher for the high-dimensional data affected by white noise, our SROM can still maintain its advantages within the acceptable range. The errors of the CPOD-NB based SROM estimated by using the test data are given in Table~\ref{T_test_err2}. It is clearly that under the influence of misjudgment samples, our SROM still has a significant improvement compared to the standard POD method. The errors of four samples in test set are shown in Figure~\ref{F_test_err2}.
       
\begin{figure}[ht]
      \centering 
      \begin{rotate}{90}
            \hspace{35pt}   \footnotesize Predicted label
      \end{rotate}
      \includegraphics[width=0.3 \textwidth,trim=0 0 0 2,clip]{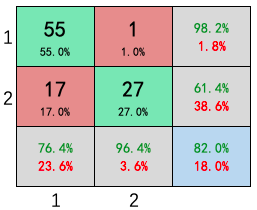}  
      \qquad
      \begin{rotate}{90}
            \hspace{35pt}   \footnotesize Predicted label
      \end{rotate}
      \includegraphics[width=0.3 \textwidth,trim=0 0 0 2, clip]{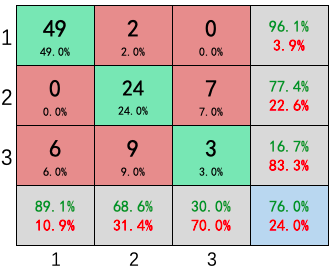} \\
      \footnotesize   True  label    \hspace{120pt}  True  label  \hspace{30pt}
      \caption{Confusion matrices of test data set with 100 samples for $K=2$ and $3$ }   \label{F_confusion2}
\end{figure}        

\begin{table}[H]
\begin{center}
\begin{minipage}{\textwidth}
\caption{Error estimates of the CPOD-NB based SROM by using 100 test samples under the true labels (left) and predicted labels (right)} \label{T_test_err2}
\begin{tabular*}{\textwidth}{@{\extracolsep{\fill}}lcccccccc@{\extracolsep{\fill}}}
\toprule%
& \multicolumn{4}{@{}c@{}}{True labels} & \multicolumn{4}{@{}c@{}}{Predicted labels} \\\cmidrule{2-5}\cmidrule{6-9}%
$K$ & $\widetilde{\mathcal{E}}_K$ & $\widetilde{\mathcal{E}}_K^r$  & $\widetilde{\mathcal{V}}_K$ & $\widetilde{\mathcal{V}}_K^r$ & $\widetilde{\mathcal{E}}_K$ & $\widetilde{\mathcal{E}}_K^r$  & $\widetilde{\mathcal{V}}_K$ & $\widetilde{\mathcal{V}}_K^r$ \\
\midrule
1  &  0.6319  &  1.9262\%  &  0.0478  &  0.0125\%  &  0.6319  &  1.9262\%  &  0.0478  &  0.0125\%        \\
2  &  0.5888  &  1.7472\%  &  0.0510  &  0.0080\%  &  0.6115  &  1.8178\%  &  0.0555  &  0.0081\%        \\
3  &  0.5464  &  1.6587\%  &  0.0502  &  0.0086\%  &  0.5722  &  1.7240\%  &  0.0594  &   0.0090\%       \\ 
\midrule
\end{tabular*}
\end{minipage}
\end{center}
\end{table}

\begin{figure}[ht]
      \centering 
      \begin{rotate}{90}
              \hspace{15pt} \scriptsize $ \| \mathbf{u} - \widetilde{ \mathbf{u}}^K \|^2_{L^2(D)}$
      \end{rotate}
      \includegraphics[width=0.24 \textwidth,trim=30 20 40 20,clip]{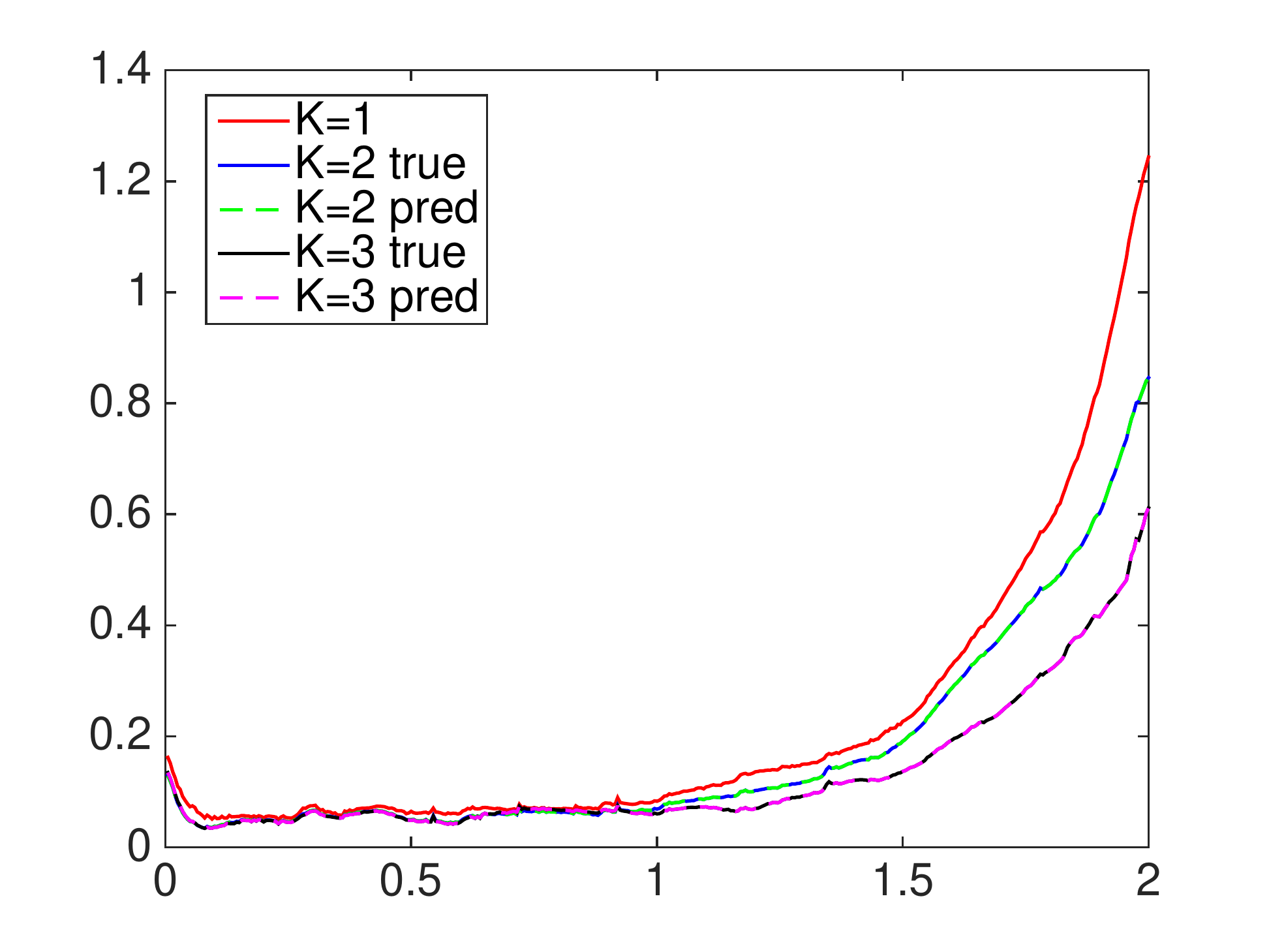}  
      \includegraphics[width=0.24 \textwidth,trim=30 20 40 20,clip]{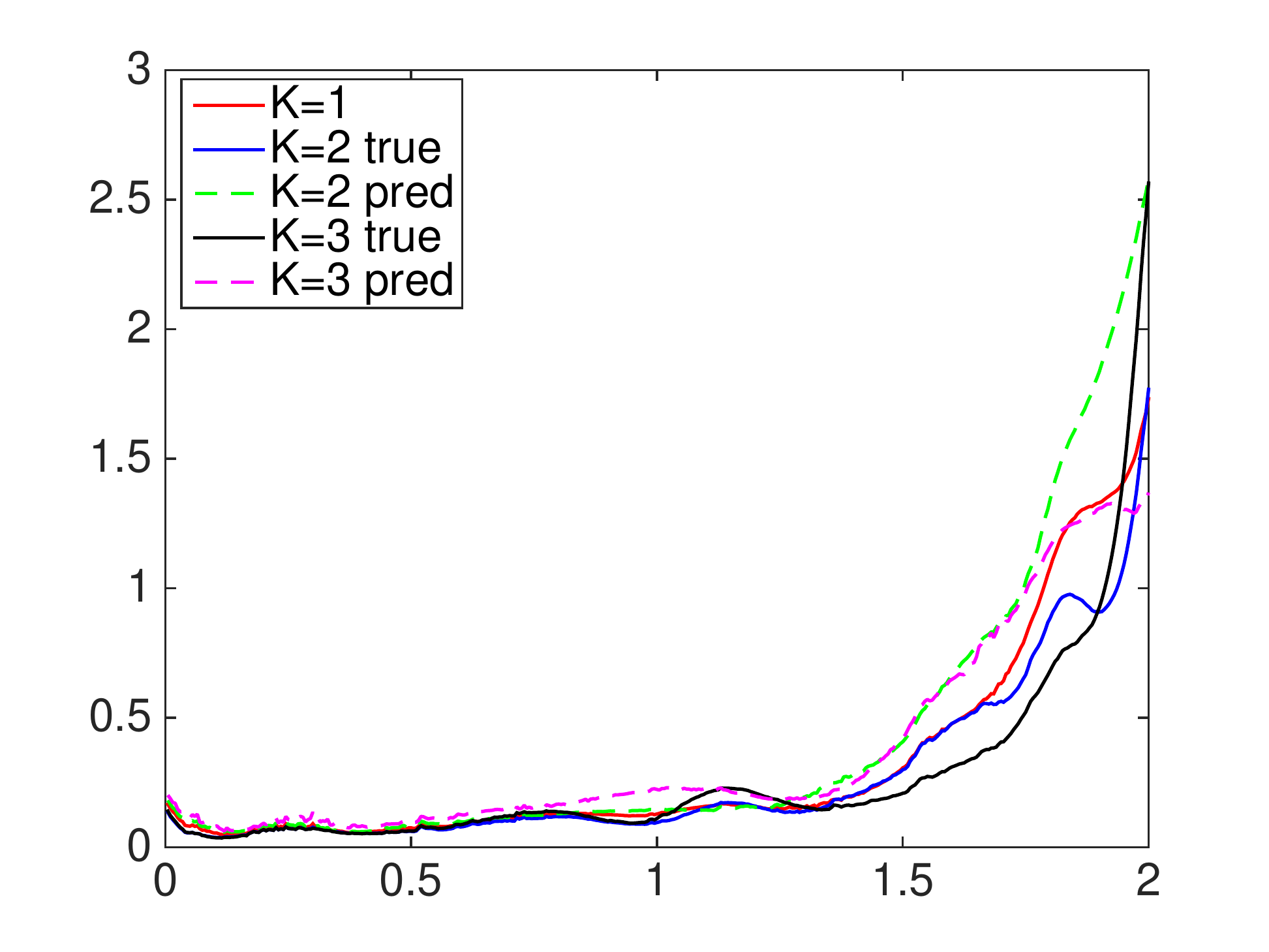} 
      \includegraphics[width=0.24 \textwidth,trim=30 20 40 20,clip]{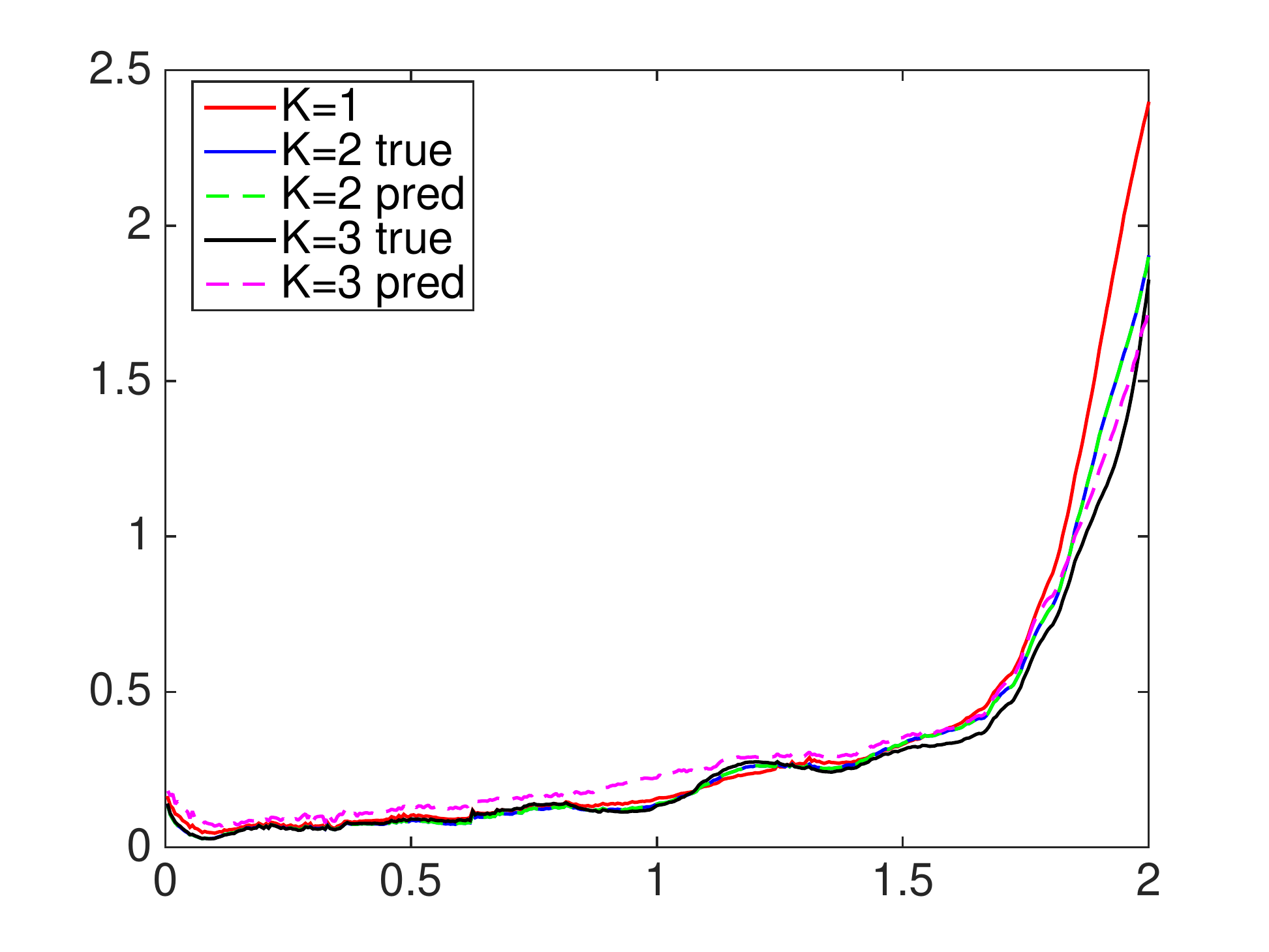}
      \includegraphics[width=0.24 \textwidth,trim=30 20 40 20,clip]{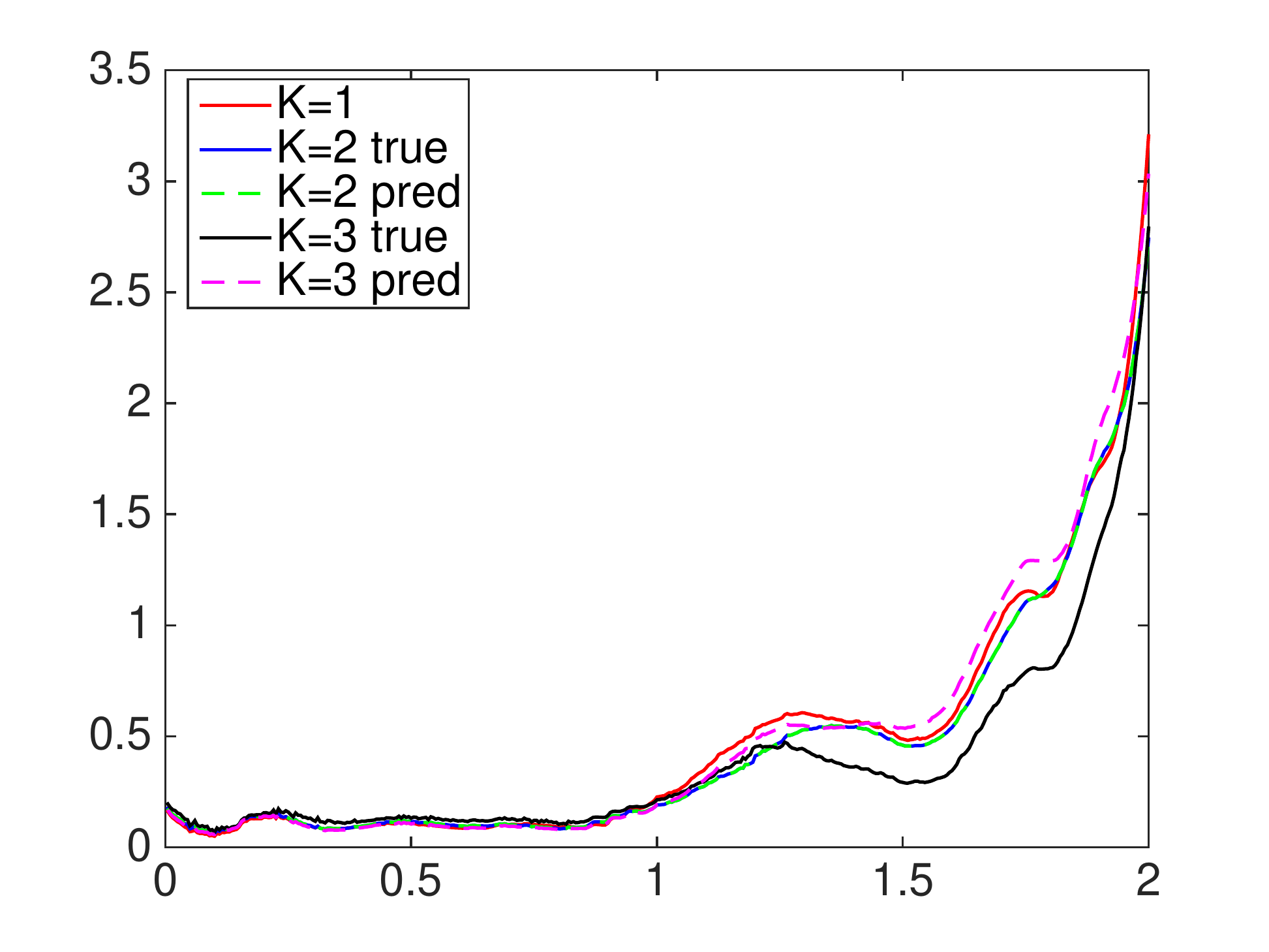}  \\ 
      \footnotesize  \hspace{1em}   $t$  \hspace{12em}  $t$   \hspace{12em}  $t$  \hspace{12em}  $t$   \hspace{10em}  \\
      \caption{The  errors of the CPOD-NB approximate solutions of four samples in the test data}    \label{F_test_err2}
\end{figure}      
      
Compared with the results in section \ref{sec_eg1}, it can be seen from Tables~\ref{T_train_err2} and \ref{T_test_err2} that the improvement of our SROM in this experiment is relatively limited, mainly includes the following two reasons. 
First of all,  although affected by the white noise, the strength $A$ still shows a hat-shaped trend as a whole, so the similarity between the realizations of the velocity field is higher, thereby the resulting CPOD basis functions are less different from the standard POD basis functions.
Secondly, the stronger randomness of input $\bm{\xi}$ leads to worse classification results, which increases the influence of misjudgment.

\section{Conclusion}   \label{Section_Conclusion}

We develop a method for model reduction by combining clustering and classification. According to the mapping relationship between input and output of the system, we use the modified t-gCVT method to cluster the output samples and generate several sets of CPOD basis functions, then use the clustering results to learn the classification mechanism of input. For a given input, compared to the standard POD basis functions, the best-matched CPOD basis functions can reduce the model better.  However, as the number of clusters increases, not only the computational complexity increase due to a large number of distance calculations, but also the error rate of the pre-classifier increases, which will affect the accuracy of our SROM. Therefore, it is necessary to study the appropriate number of clusters. 
In order to improve the stability of our algorithm, the classification of high-dimensional data is also a subject worth studying in the future, such as combining the state-of-the-art deep learning techniques.
This paper is mainly to provide a prototype of reduced-order modelling by using statistical analysis methods, and this idea can be applied to more complex problems, such as uncertainty quantification, optimal control, etc.


\end{document}